\documentclass[10pt,twoside,openright]{report} 
\usepackage[phd]{edmaths}
\usepackage{subfig}
\usepackage{amsmath}
\usepackage{breakcites}
\usepackage{dsfont}
\usepackage{tikz}
\usetikzlibrary{matrix}
\usepackage{amssymb}
\usepackage{amsthm}
\usepackage{multirow}
\usepackage{color}
\usepackage{colortbl}
\usepackage{xcolor}
\usepackage{float}
\usepackage{fancyhdr}
\usepackage{array}
\usepackage{tabularx}
\usepackage{colortbl}
\usepackage{mathrsfs}
\usetikzlibrary{shapes,arrows}
\usepackage{textcomp}
\usepackage{graphicx, hyperref,color}
\usepackage{mathtools}
\usepackage{enumitem}
\usepackage{romannum}
\usepackage{breqn}
\allowdisplaybreaks

\AtBeginDocument{\pagenumbering{arabic}}

\theoremstyle{plain} 
\newtheorem{theorem}{Theorem}[section]
\newtheorem{lemma}{Lemma}[section]
\newtheorem{proposition}{Proposition}[section]
\newtheorem{corollary}{Corollary}[section]

\newtheorem{assumption}{Assumption}[section]

\newcommand{\twopartdef}[4]
{
	\left\{
	\begin{array}{ll}
		#1 &  #2 \\
		#3 &  #4
	\end{array}
	\right.
}

\newcommand{\RN}[1]{%
	\textup{\uppercase\expandafter{\romannumeral#1}}%
}
\newcommand\numberthis{\addtocounter{equation}{1}\tag{\theequation}}
\newcommand{\Rnum}[1]{\lowercase\expandafter{\romannumeral #1\relax}}

\newcommand{\sindex}{k}
\newcommand{\diffuscal}{\hat}

\newcommand{\diffufactor}{\frac{1}{\sqrt{r}}}
\newcommand{\fluidscal}{\bar}
\newcommand{\offeredload}{\frac{\lambda^r}{\bar{\mu}}}
\newcommand{\systemindex}{r}

\newcommand{\viva}[1]{{#1}}

\DeclareMathOperator*{\argmin}{arg\,min} 

\numberwithin{equation}{section}

\title{Many Server Queueing Models with Heterogeneous Servers and Parameter Uncertainty with Customer Contact Centre Applications}
\author{Wenyi Qin}
\date{2018}

\begin{document}

\maketitle

\declaration


\newcommand{\acknowledgements}{%
	\chapter*{Acknowledgements}
}

\acknowledgements
I would like to express deep gratitude to my supervisor : Dr.\ Burak B\"uke, without whom I would never come to Edinburgh to pursue a PhD degree. He taught me to do research from the very beginning, and helped me during all these years. He was always patient with me when I faced troubles in research, from basic mathematics to more advanced problems. I also would like to thank my second supervisor Dr.\ Tibor Antal, who has always been nice and humorous. Thank you for the research advices and being supportive during my study.

I would like to thank Gill Law and Iain Dornan, who provided plenty of assistance and made my study smooth enough. I am very thankful to the China Scholarship Council for funding my research. I also wish to thank the School of Mathematics for financial support which helped me through final stages of my study, as well as allowed me to attend many conferences and summer schools. 

I would like to thank my examiners Dr.\ Gon\c{c}alo dos Reis and Dr.\ Rouba Ibrahim for their valuable comments and feedback.

I would like to thank my friends in Edinburgh for all of the priceless memories we created together, and for their kindness when I needed help, including Xavier Cabezas, Lena Freire, Hanyi Chen, Robert Gower, Marie Humbert, Saranthorn Phusingha(Mook), Minerva Mart\'in de Campo, Ivet Galabova, Marion Lemery, Rodrigo Garcia, Nicolas Loizou, Xiling Zhang, Dominik Csiba, Ba\c{s}ak Gever and many others. I would like to thank my flatmate Hongsheng Dong, who is always kind and willing to help. Special thanks to Tom Byrne for proofreading this thesis, and to Jakub Kone\v{c}n\'y, who offered tremendous help in every facet of my life during these four years.

There are no words that can express my feelings for Edinburgh. It is such a fantastic city that magically combines nature and culture, tranquility and vitality, antiquity and modernity. It is also an open-minded city that welcomes people with all kinds of backgrounds. I am so lucky that I could spend one of the most important period of my life in this romantic city with many amazing people from different culture. Thanks Edinburgh.

\thispagestyle{empty}

\begin{abstract}
In this thesis, we study the queueing systems with heterogeneous servers and service rate uncertainty under the Halfin-Whitt heavy traffic regime. First, we analyse many server queues with abandonments when service rates are i.i.d.\ random variables. We derive a diffusion approximation using a novel method. The diffusion has a random drift, and hence depending on the realisations of service rates, the system can be in \textit{Quality Driven} (QD), \textit{Efficiency Driven} (ED) or \textit{Quality-Efficiency-Driven} (QED) regime. When the system is under QD or QED regime, the abandonments are negligible in the fluid limit, but when it is under ED regime, the probability of abandonment will converge to a non-zero value. We then analyse the optimal staffing levels to balance holding costs with staffing costs combining these three regimes. We also analyse how the variance of service rates influence abandonment rate. 

Next, we focus on the state space collapse (SSC) phenomenon. We prove that under some assumptions, the system process will collapse to a lower dimensional process without losing essential information. We first formulate a general method to prove SSC results inside pools for heavy traffic systems using the hydrodynamic limit idea. Then we work on the SSC in multi-class queueing networks under the Halfin-Whitt heavy traffic when service rates are i.i.d.\ random variables within pools. For such systems, exact analysis provides limited insight on the general properties. Alternatively, asymptotic analysis by diffusion approximation proves to be effective. Further, limit theorems, which state the diffusively scaled system process weakly converges to a diffusion process, are usually the central part in such asymptotic analysis. The SSC result is key to proving such a limit. We conclude by giving examples on how SSC is applied to the analysis of systems.
\end{abstract}

\renewcommand{\abstractname}{Lay Summary}
\begin{abstract}
Contact centres have been playing a more and more important role in the society. Almost everyone has to interact with contact centres such as airline companies, banks, and utility companies. For managers, how to make decisions to balance the cost and service quality of the contact centre becomes a significant problem. Thus we need to analyse call centre properties such as the probability of waiting, staffing costs, holding costs (for delayed customers) and rate of customers abandoning the service without getting service. In reality, servers are unlikely to be identical in such systems. There are many reasons which will cause heterogeneity among servers, such as personal skills, health and weather. Hence we focus on systems where service rates are random variables. We develop approximations of the systems to analyse the basic properties and how they behave when the system becomes large. We then show that under some assumptions, the dimensions of the system process will reduce while still keeping the essential information. So called state space collapse (SSC) results will simplify the analysis of the system and help us  gain key insights using approximations.  
\end{abstract}

\tableofcontents

\chapter{Introduction and Literature Review}
\section{Introduction}
\viva{Many server queues are widely seen nowadays. For example, in banks, customers come for service, then staffs process customers' requests, and after their services are completed, customers leave banks. This is a classical many server queue system with arrivals, services, and departures. It can also be observed in other scenarios such as emergency departments of hospitals, call centres, post offices, and computers. For these systems, servers are usually different from each other. They may possess different types of skills. Even when they have the same skill, their ability on this skill may be different. Such heterogeneous systems are not studied sufficiently in literature yet. We are going to investigate their behaviours and properties in this thesis. More specifically, call centres is a very important area where this heterogeneity can be applied. We are going to focus on the applications of heterogeneous systems in call centres. }

A call centre is a service operation over the phone. It consists of groups of people, called agents or servers, who provide service to customers. Call centres are increasingly important in today's business world since they have become a preferred and prevalent means for companies to communicate with their customers. Call centres are data rich environments which triggers many interesting mathematical problems. For call centre managers, it is important to guide the system to achieve certain service levels while the costs remain reasonable. There are many factors that can influence service levels. For example, the routing scheme which leads arriving customers to specific servers, the deployment of servers, and the scheduling policy which guides servers to accept customers. Existing research provides fruitful results regarding these problems for call centres with identical servers. However, in reality, it is usually the case that servers will be heterogeneous. \viva{Individual servers can possess unique skills, and even when a group of servers have the same skills, their abilities to show the skills are subject to environment.} In this thesis, we are going to focus on call centres with heterogeneous servers and analyse their properties. Our proxy for heterogeneity is the servers' service time distribution.

A call centre can be seen as a queueing system. Customers arrive at the centre according to a stochastic process, then they are routed to available idle servers based on some routing policies. If there are no idle servers available upon their arrival, they will wait in the queue, and they may wait until they get service or abandon the queue before they get service. On the other hand, a prescribed scheduling policy is used to dispatch a server to serve a customer. Once a customer is routed to a server, s/he will be served with a specific rate based on the server, and s/he will leave the system when the service is finished.

The cost-service level trade-off has a central place in quantitative call centre management. \viva{When the cost of capacity is dominated by the waiting cost of customers, the decision maker concentrates on the waiting cost and sets the staffing levels so that the utilisation ends up being less than one}. This is called the \textit{Quality-Driven} (QD) regime. The other extreme is when staffing costs dominate the waiting costs. In such case, utilisation of servers is fixed and is equal to one. In the long run, such a large scale system will be unstable. Customers will accumulate and a significant portion of the customers will abandon. This is called the \textit{Efficiency-Driven} (ED) regime. In between these two extremes is the \textit{Quality-Efficiency-Driven} (QED) regime, under which quality and efficiency are balanced. Under this regime, utilisations will approach to one from below as the system size increases, and the proportion of customers who wait before service converges to a constant which is related to the staffing level.

In this thesis, we focus on the QED regime. \cite{sqrt} proposed the approach to analyse the system performance under the QED regime for homogeneous servers. This work is a milestone in queueing theory which initiated a lot of research. We are going to modify their assumptions and apply it to systems where each server is unique and different. \viva{In particular, we assume service rates are i.i.d.\ random variables, and remain the same once the system starts operating, i.e.\ they do not change with time.}

Apart from the staffing decisions, how to route customers to different servers will also influence the system quality. We have the routing policy to ensure customers are routed to servers in a certain way. For homogeneous systems, routing policy does not play a major role as no matter to which server a customer is routed, there is no change in the system process. For heterogeneous systems, each routing policy will yield different performances. We will mainly focus on two policies : Longest Idle Server First (LISF) policy, which will route a customer to the server who has been idle for the longest time, and Faster Server First (FSF) policy, which will always route a customer to the fastest server among those idle servers. As for scheduling policies, in our study we let it be First-In-First-Out (FIFO), i.e.\ when customers are queueing, they will get service by the order of their arrival.

Routing policies will cause not only difference in service levels, but also different fairness among servers, which is another measure for system quality. We also analyse fairness for heterogeneous servers under different routing policies.

For such large scale systems, exact analysis is usually intractable. Instead, we use asymptotic analysis. We use diffusion processes to approximate original queueing processes. By analysing properties of limiting diffusions, we can get insights for decision making in reality. Based on \cite{Atar}'s result, first we analyse many server queues with random service rates under the QED regime. Then we use a similar framework as in \cite{borst2004dimensioning} to obtain the optimal staffing levels while balancing waiting and staffing costs. Later in this thesis, we prove the diffusion limit result of \cite{Atar} using a novel method. 

Then, we analyse systems with random service rates and abandonments similarly as above. In addition to the same results for systems without abandonments, we also show that the influence of service rates variance on abandonment rate are different under different routing policies. For LISF, the abandonment rate is increasing with variance, while for FSF it may be decreasing.

Finally, we consider the state space collapse (SSC) phenomenon which implies that under some assumptions, the system process is asymptotically equivalent to a lower dimensional process. We first formulate a general method to prove SSC results for a single pool with random servers under the Halfin-Whitt regime, by which we can prove results presented in \cite{Atar} in a different way. Then, adapting the results in \cite{Dai}, we show how SSC in multiclass queueing networks can be obtained under the Halfin-Whitt heavy traffic regime when service rates are i.i.d.\ random variables within pools. 

\viva{Our main contributions can be summarised as follows.}
\viva{\begin{itemize}
\item
Use a martingale method to prove the diffusion limit for many server queues with random service rates and abandonments.
\item
Establish the optimal staffing problem for many server queues with random service rates. Provide a continuous approximation of this optimisation problem and validate it. Tightness of the steady state is proved in order to show the interchangeable limit.
\item
Prove SSC results for parallel random server systems. Use a coupling method to prove the almost Lipschitz condition for departure processes.
\end{itemize}}
The thesis is organised as follows. In the rest of this chapter, we review the literature, comparing existing results with our new results. In Chapter 2, we first present \cite{Atar}'s results, and show the diffusion limit using the method developed by Atar. Then we include abandonments to Atar's model, and also show its diffusion limit. In Chapter 3, we formulate an optimal staffing problem for models in \cite{Atar}. Later we extend the problem to systems with abandonments. We also analyse how the variance of service rates influence abandonment rates. In Chapter 4, we talk about state space collapse results for many server queueing networks with server heterogeneity, and how it can be applied to system analysis. \viva{In Chapter 5, we conclude the thesis by summarising our contributions and point out future research directions.}

\section{Literature review}
\viva{Queueing models are used broadly in many service systems such as call centres, healthcare and computer science. For systems with arrivals, service and departures, it is convenient to model them as queueing systems and analyse their performance. For example, for call centres, there are incoming calls, agents who answer calls and call departures, and such call centres can be analysed using queueing theory. As for queueing analysis used in other areas, \cite{Mandelbaum} discuss fair routing between emergency departments and hospital wards under the QED regime. \cite{deo2011centralized} use a game-theoretic queueing model and find an equilibrium on the accepted diverted ambulance from emergency departments of other hospitals. \cite{tezcan2014routing} consider customer service chat systems where customers can receive real time service from agents using an instant messaging application over the internet. We will focus on call centres in this thesis using queueing modelling and analyse their performance.}\\
The research on call centres can be viewed under different headings. \cite{tutorial} and \cite{multiperspective} review research on call centres, and provide a survey of literature on call centre operations management. They also identify some promising directions for future research. The most famous and basic queueing model for call centres is the Erlang-C or Erlang delay model, which deals with only one type of call and server without abandonments; thus every customer waits until s/he reaches a server. \cite{Koole} gives a general idea of how the Erlang-C formula is used in call centre. More mathematical details on the Erlang-C formula can be found in \cite{cooper}. The Erlang-C formula is an important formula in the early stage of call centre research. It gives an explicit form of statistical-equilibrium distributions of the queueing process given the arrival rate, identical service rates, and number of servers. Thus, the probability of waiting can be calculated, to decide on the number of servers that are needed to make the system achieve a certain service level.

However, such direct analysis becomes impractical when the system grows large. Hence, the asymptotic analysis should be used. Diffusion approximations for stochastic processes in queueing models prove to be quite useful (see \cite{Iglehart}, \cite{Stone}, \cite{sqrt}); The work of \cite{sqrt} is the most relevant to our work. It considers a sequence of $GI/M/s$ systems in which the traffic intensities converge to one from below, which brings the Halfin-Whitt heavy traffic regime into the picture. With arrival rate $\lambda$, service rate $\mu$, and the Halfin-Whitt regime, under a certain scaling, the probability of waiting converges to a constant $\alpha$ which is strictly greater than $0$ and less than $1$. $\alpha$ can be used to indicate the service level of the system. The offered load is a measure of traffic in a queue and is defined to be $\frac{\lambda}{\mu}$. The staffing level of such a system will be offered load $\frac{\lambda}{\mu}$ plus the square root of the offered load $\sqrt{\frac{\lambda}{\mu}}$ multiplied by a constant $\beta$, where the coefficient $\beta$ depends on the service level $\alpha$. The quantity $\beta\sqrt{\frac{\lambda}{\mu}}$ is called the ``safety staffing" level against stochastic variability. Using this approach, staffing and waiting costs are well balanced. The Halfin-Whitt regime has been extended in several directions. \cite{Janssen} propose refinements of the celebrated square-root safety-staffing rule which have the appealing property that they are as simple as the conventional square-root safety-staffing rule. \cite{Atar} introduces a new square root staffing policy for many servers systems with random service rates. \cite{Puhalskii} extend the results to a system with multiple customer classes, priorities, and phase-type service distributions. \cite{Armony} establishes diffusion approximations and staffing levels for inverted-V systems. Our study also focuses on the Halfin-Whitt regime, and extends it to heterogeneous servers instead of identical ones.

Staffing has always been a central issue for call centre managers. There are many research papers on staffing problems for different models.  See \cite{mandelbaum2009staffing}, \cite{koccauga2015staffing}, \cite{whitt2006staffing}, \cite{mandelbaum2009staffing} and \cite{armony2011routing} for different discussions. Based on staffing level, systems can be in a QD, ED or QED regime. When the system is under a QD or QED regime, the abandonments are negligible in the limit, but when it is under an ED regime, the probability of abandonment will converge to a non-zero value. \cite{borst2004dimensioning} determine the asymptotically optimal staffing level for $M/M/N$ queues under different regimes. \cite{whitt2004efficiency} investigates the ED many server heavy traffic regime for queues with abandonments. Our staffing model is based on the framework of \cite{borst2004dimensioning}.

\viva{For heterogeneous systems, how to route arriving customers to servers and how to schedule servers to serve customers are crucial decisions to the system performance. Routing and scheduling have been studied extensively in the literature. \cite{tezcan2014routing} consider customer service chat systems where agents can serve multiple customers simultaneously. They propose routing policies for such system with impatient customers with the objective to minimise the probability of abandonment in steady state. \cite{gurvich2010staffing} consider the staffing problem for call centres with multi-class customers and different agent types operating under QD constraints and arrival rate uncertainty. They propose a two-step solution which contains two actions: the number of agents of each type, and a dynamic routing policy. \cite{Armony} shows that for the inverted-V model, the FSF policy is asymptotically optimal in the QED regime and no thresholds are needed. There is literature that carries out exact analysis and asymptotic analysis under conventional heavy traffic, such as \cite{rykov2004optimal} and \cite{kelly1993dynamic}.} Routing policies also play an important role in staffing optimisation. Under different policies, the steady state behaviour of the system changes and fairness among servers is also different. FSF policy is commonly used. \cite{Armony} establishes diffusion limits for the inverted-V systems under FSF policy and concludes they have a better performance than their corresponding homogeneous systems. \cite{Atar} provides diffusion limits for many server queues with random servers under LISF and FSF policies. \cite{tezcan2008optimal} develops limit theorems for inverted-V systems under minimum-expected-delay faster-server-first (MED-FSF) and minimum-expected-delay load-balancing (MED-LB) routing policies. Notice that LISF is a blind policy, i.e.\, it only needs to track the state of the process in order to make routing decisions, and information about service rates is not needed in this case. We will mainly use this policy in our model since our service rates are random variables and thus their realisations are unknown before the systems start to operate.

State space collapse is an important phenomenon when we analyse the system behaviour. \cite{harrison1997dynamic} explain the dimension reduction in general terms, using an orthogonal decomposition. For some examples of SSC one can check \cite{reiman1984some}. \cite{Puhalskii} prove SSC for a particular system which has phase type distributed service rates. \cite{Bramson} uses the hydrodynamic scaling to build up the state space collapse results for multi-class queueing networks under the conventional heavy traffic regime. He shows that we can use a lower dimensional process, the workload processes of each service station, to represent the system because the original system process, which is the number of each type of customers in every station, can be obtained through the workload processes and some lifting functions. The paper by \cite{Dai} uses the hydrodynamic scaling proposed by \cite{Bramson} in a many server setting. Their contribution is the definition of a SSC function, which is used in that paper to show the dimension reduction for many server networks under the Halfin-Whitt heavy traffic regime. Our SSC result is based on their framework. \cite{tezcan2008optimal} applies this method to a distributed parallel server system and does optimal control analysis.   

The key point in our research is heterogeneity and uncertainty in parameters. The uncertainty in arrival rates is investigated in some prior work. For example, \cite{Zan} analyses the staffing problems when the arrival rate is uncertain. However, for uncertainty in the service rates, there is still plenty of space for us to explore. For a general non-technical introduction to this topic, \cite{gans2010service} is an excellent reference. The heterogeneity in the servers is modelled in various ways. A commonly used one is the inverted-V system, which contains a single customer class and multiple server types. \cite{Armony} considers the asymptotic framework for such systems. She shows that the FSF policy is asymptotically optimal in the QED regime. Later in \cite{ArmonyWard}, an optimisation problem for inverted V systems is formulated. They minimise the steady-state expected customer waiting time subject to a ``fairness" constraint and propose a threshold routing policy which is asymptotically optimal in the Halfin-Whitt regime. \cite{Mandelbaum} introduce the randomised most-idle (RMI) routing policy for the inverted-V model and analyse it in the QED regime. \cite{Atar}'s results about random servers set the cornerstones for our work. The diffusion limit in \cite{Atar} contains a random drift, which comes from the heterogeneity of servers. From the diffusion, we derive its steady state distribution, then formulate the optimisation staffing problem using the distribution.

\chapter{Diffusion Limits for Single Server Pool Systems}

\viva{\section{Introduction}}
\viva{In this chapter we focus on many server queueing systems with random service rates under the Halfin-Whitt heavy traffic regime. Many server queueing models have been studied extensively. However, there are only a few papers about many server queues when service rates are random variables instead of identities. Models with identical servers are not sufficient when it comes to modelling human behaviours. Individual abilities are always influenced by environment thus they can not be constants. }To capture the feature more accurately, we assume our model has $N^{\systemindex}$ exponential servers with i.i.d.\ service rates $\mu_{\sindex}, \sindex = 1,\dots,N^{\systemindex}$. $N^{\systemindex}$ is also assumed to be a random variable. When customers arrive into the system they will either queue in a buffer with infinite room, or be routed to a server according to the LISF routing policy. Customers from the queue are routed to servers according to the FIFO rule. In this chapter, we first assume that the customers do not abandon and leave the system only after their service is completed. We relax this assumption later in the chapter. The routing policy is work conserving, in the sense that no server will be idle when there is at least one customer in the queue. The service policy is non-preemptive, i.e.\, once a customer is assigned to a server, it will continuously receive service until it is completed, i.e.\ the services will not be interrupted.  \viva{This model is considered in \cite{Atar}. \cite{Atar} also analyses the same systems under the FSF routing policy. We will not focus on FSF policy} because our decision is supposed to be made before the system starts running, and FSF policy requires the knowledge of each server's rate, which does not suit our case. This chapter is organised as follows: in Section 2.1, we give the detailed description of the mathematical model and notations used throughout this chapter; in Section 2.2, we rewrite the proof of the central theorem in \cite{Atar}, although it is already proven by Atar, we present it here for completeness of our discussion; in Section 2.3, we formulate the optimal staffing problem and prove the validity of its asymptotic version.

\section{Mathematical modelling and notation}
First we introduce the notation used throughout the thesis. All of the random variables and stochastic processes are defined on a complete probability space $(\Omega, \mathscr{F},\mathbb{P})$. For a positive integer $d$, we denote by $\mathbb{D}(\mathbb{R}^d)$ the space of functions from $\mathbb{R}^+$ to $\mathbb{R}^d$ that are right continuous and left limits exist (RCLL), endowed with the usual Skorohod $J_1$ topology. (See \cite{Billingsley} for the definition.) We use $\Rightarrow$ to denote weak convergence. And for $X \in \mathbb{D}(\mathbb{R})$, we write $||X||_{t}:= \sup_{0 \leq s \leq t}|X(s)|$.

The model is parameterised by $\systemindex \in \mathbb{R}^+$, where for each $\systemindex$, $N^{\systemindex} \in \mathbb{N}$ is a random variable, representing number of servers. Service times for customers served at server $\sindex$ are i.i.d.\ exponentially distributed with rate $\mu_{\sindex},\sindex = 1,\dots,N^{\systemindex}$. The $\mu_{\sindex}$s are assumed to be nonnegative and lie in an interval $[p,q]$. The distribution of $\mu_{\sindex}$ is denoted by $m$, and its expected value is 
\begin{equation}
\bar{\mu}:= \int_{[0,\infty)} x dm \in (0, \infty). \label{service-rate-expectation}
\end{equation}
It is also assumed that $N^{\systemindex}$ satisfies the following two assumptions
\begin{align*}
&\mathbb{P}(N^{\systemindex} \leq 2 \systemindex) = 1,\\
&\frac{N^{\systemindex}}{\systemindex} \Rightarrow 1, \text{ as } \systemindex \to \infty.
\end{align*}
The arrivals are assumed to be renewal processes with finite second moments for the interarrival times. Let the arrival rate be $\lambda^{\systemindex}$ such that $\lim_{\systemindex \to \infty} \frac{\lambda^{\systemindex}}{\systemindex} = \lambda >0$, and a sequence of strictly positive i.i.d.\ random variables $\{\check{U}(l), l \in \mathbb{N}\}$, with mean $\mathbb{E}\check{U}(1) = 1$ and variance $C^2_{\check{U}} = Var(\check{U}(1)) \in [0,\infty)$. The Halfin-Whitt heavy traffic condition, which makes the system critically loaded, is assumed to be
\begin{equation}
\lim_{\systemindex \to \infty} \diffufactor (\lambda^{\systemindex} - \systemindex \lambda) = \hat{\lambda}, \label{arrival-2moment-condition}
\end{equation}
where $\lambda =  \bar{\mu}$, and $\hat{\lambda} < 0$.

For the $\systemindex$th system, let $A^{\systemindex}(t)$ be the total number of arrivals into the system up to time $t$, $X^{\systemindex}(t)$ be the total number of customers in the $\systemindex$th system at $t$ (including customers both being served and waiting in the queue), $D^{\systemindex}_{\sindex}(t)$ be the number of jobs completed by server $\sindex$ up to time $t$, and $T^{\systemindex}_{\sindex}(t)$ be the accumulated busy time of server $\sindex$ by time $t$. Let $B^{\systemindex}_{\sindex}(t) = 1$ if server $\sindex$ is busy at time $t$, and it equals to zero if the server is idle. Accordingly, let $I^{\systemindex}_{\sindex}(t) = 1 - B^{\systemindex}_{\sindex}(t)$, indicating that server $\sindex$ is idle if $I^{\systemindex}_{\sindex}(t)$ equals $1$. 

\section{Atar's results on many server queues with random servers}
When there are no abandonments, \viva{it can be shown} the systems satisfy the equation
\begin{equation}
X^{\systemindex}(t) = X^{\systemindex}(0) + A^{\systemindex}(t) - \sum_{\sindex=1}^{N^{\systemindex}} D^{\systemindex}_{\sindex}(t). \label{Atar-system-equation0}
\end{equation}

We let $\diffuscal{X}^{\systemindex}(t)$ be a renormalized version of the process $X^{\systemindex}(t)$, which is defined as
\begin{equation}
\diffuscal{X}^{\systemindex}(t) = \diffufactor (X^{\systemindex}(t) - N^{\systemindex}). \label{Atar-diffusion-scaling}
\end{equation}
The initial value of $X(t)$ and the random variable $N^{\systemindex}$ are assumed to satisfy
\begin{equation}
(\diffuscal{X}^{\systemindex}(0), \diffuscal{N}^{\systemindex}):= \left(\diffufactor(X^{\systemindex}(0) - N^{\systemindex}), \diffufactor(N^{\systemindex} - \systemindex))\right) \Rightarrow (\xi(0) ,\nu), \label{initial-condition-convergence}
\end{equation}
where $(\xi(0), \nu)$ is an $\mathbb{R}^2$-valued random variable.

\viva{\cite{Atar}'s main result is the following theorem. It states that under the LISF policy, the sequence of the scaled processes of total number of customers weakly converges to a diffusion. What makes the diffusion distinctive from other diffusion limits is that it contains a random drift which arises from the randomness of service rates.}

\begin{theorem}[\cite{Atar}]
	Assume $\int x^2 d m < \infty$. Then, under the LISF policy, the processes $\diffuscal{X}^{\systemindex}(t)$ weakly converge to the solution of the following SDE
	\begin{equation}
	\xi(t) = \xi(0) + \sigma w(t) + \beta t + \gamma \int_0^t \xi(s)^- ds, t \geq 0,
	\end{equation}
	where $\sigma^2 = \lambda C^2_{\check{U}} + \bar{\mu} = \bar{\mu}(C^2_{\check{U}} + 1)$, and $\beta = \hat{\lambda} - \zeta - \bar{\mu} \nu$, $\zeta$ is a normal random variable with parameters $(0, \int (x - \bar{\mu})^2 d m)$, $\gamma = \frac{\int x^2 d m}{\int x d m}$, $w(t)$ is a standard Brownian motion, and the three random elements $(\xi(0),\nu), \zeta,$ and $w(t)$ are mutually independent. \label{Atar-theorem}
\end{theorem}

\viva{\cite{Atar} proves his result using a method particular to this model. We give an explanation of his proof here. The detailed mathematical proof is included in Appendix \ref{Appendix-Atar-proof} for completeness. Later we will use another approach inspired by state space collapse phenomenon to prove this theorem again. For more detailed discussion on the new proof, see Section 4.5.\\
The process of scaled total number of customers are
\begin{align*}
\diffuscal{X}^{\systemindex}(t) 
&= \diffufactor(X^{\systemindex}(0) - N^{\systemindex}) + \diffufactor A^{\systemindex}(t) - \diffufactor \sum_{\sindex=1}^{N^{\systemindex}} D^{\systemindex}_{\sindex}(t)\\
&= \diffuscal{X}^{\systemindex}(0) + \diffufactor \left(A^{\systemindex}(t) - \lambda^{\systemindex}t \right) + \diffufactor \lambda^{\systemindex}t - \diffufactor \sum_{\sindex=1}^{N^{\systemindex}} D^{\systemindex}_{\sindex}(t)\\
&= \diffuscal{X}^{\systemindex}(0) + \diffuscal{A}^{\systemindex}(t) + \diffufactor \lambda^{\systemindex}t - \diffufactor \sum_{\sindex=1}^{N^{\systemindex}} D^{\systemindex}_{\sindex}(t). \numberthis \label{explain-Atar-proof-1}
\end{align*}	

The convergence of arrival process $\diffuscal{A}^{\systemindex}(t)$ follows from the basic functional central limit theorem for renewal processes. Analysing the convergence of departure process $D^{\systemindex}_{\sindex}(t)$ is considerably harder. Each server needs to be considered individually because they have their unique service rate, and thus it is an $N^{\systemindex}$ dimensional process. However, as the system grows large, the number of servers tends to infinity; hence analysing each $D^{\systemindex}_{\sindex}(t)$ becomes intractable.

\cite{Atar} solves this problem by partitioning the servers into finite $I$ pools. Within each pool, the supremum of the difference between two service rates is bounded by some small positive number $\epsilon$. }

\viva{To see how this idea is used in the proof, first we need to obtain more insights about the departure processes. The departure process of each server is treated as a time changing Poisson process. Let $\{S_{\sindex}(t), \sindex \in \mathbb{N}\}$ be independent standard Poisson processes. Since $T^{\systemindex}_{\sindex}(t)$ is the accumulated busy time for server $\sindex$ by time $t$, then by random time change, $D^{\systemindex}_{\sindex}(t)$ satisfies
\begin{equation}
D^{\systemindex}_{\sindex}(t) = S_{\sindex}(\mu_{\sindex} T^{\systemindex}_{\sindex}(t)), \ \sindex = 1,\dots, N^{\systemindex}. \label{Atar-Dk}
\end{equation}
where
\begin{equation}
T^{\systemindex}_{\sindex}(t) = \int_{0}^{t}  B^{\systemindex}_{\sindex}(s) ds.
\end{equation}
}

\viva{Denote $\Sigma^{\systemindex} = (A^{\systemindex}, X^{\systemindex}, Q^{\systemindex}, I^{\systemindex}, \{B^{\systemindex}_{\sindex}, D^{\systemindex}_{\sindex}\}_{\sindex=1,\dots,N^{\systemindex}})$ which has a.s.\ piecewise constant and right-continuous sample paths. \cite{Atar} partitions servers into pools as mentioned above. Pools are indexed by $\{i, i=1,2,\dots, I\}$. Using such configuration, the system can be regarded approximately as an inverted-V system, hence departure processes can be considered aggregately in each pool. Denote the total departure process of pool $i$ as $D^{\systemindex,(i)}(t)$. He further defines the total service rate of pool $i$ to be the sum of $\mu_{\sindex} T^{\systemindex}_{\sindex}(t)$ (the sum is over all of the servers in pool $i$). Then $D^{\systemindex,(i)}(t)$ is also a time changing Poisson process and 
\begin{align*}
D^{\systemindex,(i)}(t) = S^{(i)}\left(\sum_{\sindex \in \mbox{ pool }_i} \mu_{\sindex} T^{\systemindex}_{\sindex}(t)\right),
\end{align*}
where $\{S^{(i)}(\cdot)\}$ are independent standard Poisson processes. For simplicity denote the total service rate in pool $i$ to be
\begin{align*}
T^{\systemindex,(i)}(t) = \sum_{\sindex \in \mbox{ pool }_i} \mu_{\sindex} T^{\systemindex}_{\sindex}(t).
\end{align*}
He shows in Proposition 3.1 (see Appendix\ref{Appendix-Atar-prop3.1}) that with this pooling method, the total departure process of the original system $\sum_{\sindex=1}^{N^{\systemindex}} D^{\systemindex}_{\sindex}(t)$ and the total departure process of the approximate inverted-V system $\sum_{i=1}^I D^{\systemindex, (i)}(t)$ are equal in distribution. }

\viva{Using this equivalence, system equation \eqref{explain-Atar-proof-1} is equal in distribution to
\begin{align*}
\diffuscal{X}^{\systemindex}(t) 
&= \diffuscal{X}^{\systemindex}(0) + \diffuscal{A}^{\systemindex}(t) + \diffufactor \lambda^{\systemindex}t - \diffufactor \sum_{i=1}^{I}D^{\systemindex,(i)}(t).
\end{align*}
With further manipulations, it can be expressed as
\begin{align*}
\diffuscal{X}^{\systemindex}(t) 
&= \diffuscal{X}^{\systemindex}(0) + \diffuscal{A}^{\systemindex}(t) + \diffufactor \lambda^{\systemindex}t - \diffufactor \sum_{i=1}^{I} \left( S^{(i)} \left( T^{\systemindex,(i)}(t) \right) - T^{\systemindex,(i)}(t)  \right) - \diffufactor \sum_{i=1}^{I} T^{\systemindex,(i)}(t)\\
&= \diffuscal{X}^{\systemindex}(0) + \diffuscal{A}^{\systemindex}(t) + \diffufactor (\lambda^{\systemindex}t - \systemindex \bar{\mu}t) - \diffufactor \left(\sum_{\sindex=1}^{N^{\systemindex}} \mu_{\sindex} t - \systemindex \bar{\mu}t\right)\\
& \quad - \diffufactor \sum_{i=1}^{I} \left( S^{(i)} \left( T^{\systemindex,(i)}(t) \right) - T^{\systemindex,(i)}(t)  \right) - \diffufactor \left( \sum_{i=1}^{I} T^{\systemindex,(i)}(t) - \sum_{\sindex=1}^{N^{\systemindex}} \mu_{\sindex} t \right).
\end{align*}
}  

\viva{
Convergence of the first four items are easy. Let us pay attention to the latter two. First, it is proved that $\frac{1}{\systemindex}T^{r,(i)}(t)$ converges to its fluid limit $\rho_i t$, where $\rho_i$ is the product of the expectation of $\mu_{\sindex}$ and the weight probability of servers in pool $i$. $\diffufactor \left( S^{(i)} \left( T^{\systemindex,(i)}(t) \right) - T^{\systemindex,(i)}(t)\right) $ is a martingale, then \cite{Atar} uses Functional Central Limit Theorem, and applies random time change to the fluid limit $\rho_i t$ to show that $\diffufactor \sum_{i=1}^{I} \left( S^{(i)} \left( T^{\systemindex,(i)}(t) \right) - T^{\systemindex,(i)}(t) \right)$ converges to some Brownian motion.}

\viva{$\diffufactor \left( \sum_{i=1}^{I} T^{\systemindex,(i)}(t) - \sum_{\sindex=1}^{N^{\systemindex}} \mu_{\sindex} t \right)$ is the most concerning part in the proof. It is actually equal to
\begin{align*}
\diffufactor  \sum_{\sindex=1}^{N^{\systemindex}} \left( \mu_{\sindex}T^{\systemindex}_{\sindex}(t) -  \mu_{\sindex} t \right)
&= \diffufactor  \sum_{\sindex=1}^{N^{\systemindex}} \left( \mu_{\sindex} \int_0^t B^{\systemindex}_{\sindex}(s) ds - \mu_{\sindex}t\right) = \diffufactor  \sum_{\sindex=1}^{N^{\systemindex}} \left( -\mu_{\sindex} \int_0^t I^{\systemindex}_{\sindex}(s) ds \right),
\end{align*}
which is the total amount of unused service capacities due to idleness. Denote this lost capacity as $F^{\systemindex}(t)$. }

\viva{$F^{\systemindex}(t)$ contains $N^{\systemindex}$ different idleness processes. It is later proved that there is a state space collapse (SSC) in such systems in the sense that, in the limit, the total lost capacities $F^{\systemindex}(t)$ can be represented as the total accumulated idle time multiplied by some coefficient $\gamma$ related to the service rate distribution and routing policy. Such a SSC result reduces the dimension of the original processes, which eventually helps to get a one dimensional diffusion limit. \\
To show such SSC result, \cite{Atar} considers the difference between total lost capacities $F^{\systemindex}(t)$ and the product of total accumulated idle time and $\gamma$. The difference consists of four $e_i(t)$s: $e_1(t),e_2(t),e_3(t),e_4(t)$. He proves that each $e_i(t)$ tends to zero in the limit. These four $e_i(t)$s are unique to this system, thus cannot be directly extended to other models. In Chapter 4, we explain the SSC phenomenon in detail. Then we use a more generic method to prove this SSC result again.
}

\section{Extension of Atar's results to include abandonments}
In this section, we assume customers may abandon the system prior to being served. Each customer has an associated patience time, and abandons the system without obtaining any service if the waiting time in the queue exceeds the customer's patience. Once his/her service starts, s/he cannot abandon the system. Assume each customer's patience time is exponentially distributed with rate $\nu$. \viva{We will not deal with the abandonment processes directly due to complications of analysing each customer's patience individually. Instead, we use a ``perturbed'' abandonment processes similar to the one described in Section 2.1 of \cite{Dai}. In perturbed systems, only the customer at the head of the queue will be able to abandon, and her/his abandonment rate is the sum of abandonment rates of all of the customers in the queue. Under the assumption of exponential service and patience time, the equivalence of systems with original abandonment processes and perturbed abandonment processes is proved in \cite{Dai}. Note that \cite{Dai} use perturbed system technique to analyse both abandonment processes and service processes, while we only use it for our abandonment processes since our service rates are no longer deterministic thus the equivalence to the perturbed systems is invalid.}

Let $Q(t)$ be the queue length at time $t$, and let $M(t)$ denotes the number of customers who have abandoned queue by time $t$. The systems are assumed to be under LISF policy again.

Let $S_Q(t)$ be a standard Poisson process. We define
\begin{equation}
G^r(t) = \int_0^t Q^r(s) ds, \ \ t \geq 0.
\end{equation}
Then for the perturbed abandonment process
\begin{equation}
R^r(t) = S_Q(\nu G^r(t)). \label{ABprocess}
\end{equation}
Using the same notations as in the previous section, we can write the system dynamic equations
\begin{align}
X^r(t) &= X^r(0) + A^r(t) - \sum_{\sindex = 1}^{N^r} D_{\sindex}^r(t) - R^r(t).
\end{align}
We can show similarly that diffusion limits exist in the presence of abandonments.
\begin{theorem}
	Assume $\int x^2 d m < \infty$. Then, under LISF policy, the diffusively scaled processes $\diffuscal{X}^{\systemindex}(t)$ weakly converge to the solution of the following SDE
	\begin{equation}
	\xi(t) = \xi(0) + \sigma w(t) + \beta t + \gamma \int_0^t \xi(s)^- ds - \nu \int_0^t \xi(s)^+ ds, t \geq 0, \label{diffusion-with-abandonments-equation}
	\end{equation}
	where $\sigma, \beta, \zeta, \gamma$ and $w(t)$ are as in Theorem \ref{Atar-theorem}. \label{diffusion-with-abandonments}
\end{theorem}
\begin{proof}
	The proof is quite similar to Theorem \ref{Atar-theorem}. We will focus on abandonment processes here since the proofs of other parts are the same.
	
	Following arguments in the previous section we again omit the symbol $\systemindex$, and thus we have
	\begin{align}
	\diffuscal{X}(t) 
	&= \diffuscal{X}(0) + \diffufactor A(t) - \diffufactor \sum_{i=1}^I S^{(i)}(T^{(i)}(t)) - \diffufactor R(t)\\
	&= \diffuscal{X}(0) + W(t) + b^r t + F(t) - \diffuscal{R}(t),\ \ \  \left(\text{let}\ \  \diffuscal{R}(t) = \diffufactor R(t)\right)
	\end{align}
	under the LISF policy, 
	\begin{equation}
	\diffuscal{X}(t) = \diffuscal{X}(0) + W(t) + b^r t + \gamma \int_0^t \diffuscal{X}(s)^- ds - \diffuscal{R}(t) + e(t), \ \ \left( \gamma = \frac{\int x^2 d m}{\int x d m} \right). \label{dynamicequ}
	\end{equation}
	
	We already showed in the previous section that $W(t) \Rightarrow \sigma w$, $b^r \Rightarrow \beta$, and $e(t) \to 0$ u.o.c. in probability. For the newly added term $\diffuscal{R}(t)$, note that
	\begin{align*}
	\diffuscal{R}(t) 
	&= \diffufactor S_Q(\nu G(t)) = \diffufactor S_Q\left( \nu \int_0^t Q(s) ds\right)\\
	&= \diffufactor \left( S_Q\left( \nu \int_0^t Q(s) ds \right) - \nu \int_0^t Q(s) ds + \nu \int_0^t Q(s) ds \right)\\
	&= \diffufactor \left( S_Q\left( \nu \int_0^t (X(s) - N)^+ ds \right) - \nu \int_0^t (X(s) - N)^+ ds + \nu \int_0^t (X(s) - N)^+ ds  \right)\\
	&= \diffufactor \left( S_Q\left( r \frac{\nu}{r} \int_0^t (X(s) - N)^+ ds \right) - r \frac{\nu}{r} \int_0^t  (X(s) - N^r)^+ ds \right) + \nu \int_0^t \diffuscal{X}(s)^+ ds.
	\end{align*}
	Denote $\diffuscal{M}(t) = \diffufactor \left( S_Q\left( \nu \int_0^t (X(s) - N)^+ ds \right) -  \nu \int_0^t (X(s) - N)^+ ds \right)$. Then by Theorem 7.2 in \cite{pang2007martingale}, $\diffuscal{M}(t)$ is a square-integrable martingale with respect to the filtrations $\mathbf{F}_r \equiv \{ \mathcal{F}_{r,t}: t \geq 0 \}$ defined by
	\begin{equation}
	\begin{split}
	\mathcal{F}_{r,t} \equiv \sigma  \left( \vphantom{\int_{0}^{s}} X(0), A(s), S^{(1)}(T^{(1)}(s)), \dots, S^{(q)}(T^{(q)}(s)), \right. \\
	\left. S_Q \left( \nu \int_0^s (X(u) - N)^+ du \right): 0 \leq s \leq t  \right), t \geq 0,
	\end{split}
	\end{equation}
	augmented by including all null sets. Its predictable quadratic variation is 
	\begin{equation}
	\langle \diffuscal{M} \rangle (t) = \frac{\nu}{r} \int_{0}^{t} \left( X(s) - N \right)^+ ds, t \geq 0.
	\end{equation}
	By the same reasoning in Section 7.1 in \cite{pang2007martingale}, we obtain the deterministic limits
	\begin{equation}
	\langle \diffuscal{M} \rangle (t) \Rightarrow 0. \label{zeroconvergence}
	\end{equation}
	
	We explain the proof briefly. For more details, see \cite{pang2007martingale}. We have that the sequence $\{ \diffuscal{X}^{\systemindex}\}$ is stochastically bounded in $\mathbb{D}(\mathbb{R})$. Then, by Lemma 5.9 and Section 6.1 in \cite{pang2007martingale}, we get the Functional Weak Law of Large Numbers (FWLLN) corresponding to Lemma 4.3, from which we can prove \eqref{zeroconvergence}. Recall the basic Functional Central Limit Theorem (FCLT): $\diffuscal{S}(t) = \frac{S(rt) - rt}{\sqrt{r}} \Rightarrow B$, where $S$ is a standard Poisson process, and $B$ is a standard Brownian motion. Then
	\begin{equation}
	\diffuscal{R}(t) = \hat{S} \left( \langle \diffuscal{M} \rangle (t) \right) + \nu \int_0^t \diffuscal{X}(s)^+ ds, \label{abandonconverge}
	\end{equation}
	and by the lemma on random change of time in \cite[p.~150]{Billingsley}, we have  $\hat{S} \left( \langle \diffuscal{M} \rangle (t) \right) \Rightarrow 0$, in the uniform topology on the compact set $[0,T]$ for any $T>0$. Thus,
	\begin{equation}
	\diffuscal{R}(t) \Rightarrow \nu \int_0^t \xi(s)^+ ds, \ as \ r \to \infty.
	\end{equation} 
	
	By Skorohod Representation Theorem, we can assume without loss of generality that the random variables $\diffuscal{X}(0), b, \xi(0),$ and $\beta$, and the processes $W, \diffuscal{S} \left( \langle \diffuscal{M} \rangle (t) \right)$, and $w$ are realized in such a way that 
	\begin{equation}
	\left( \diffuscal{X}(0), b, W, \diffuscal{S} \left( \nu \int_0^t \frac{1}{\systemindex} Q(s) ds \right) \right) \to \left(\xi(0), \beta, \sigma w, 0 \right) \ \ \  \text{in probability, as } \systemindex \to \infty. \label{skorohod}
	\end{equation}
	 
	Recall that $||X||_{t} := \sup_{0 \leq s \leq t}|X(s)|$. Combining \eqref{dynamicequ}, \eqref{abandonconverge}, and \eqref{diffusion-with-abandonments-equation}, the inequalities $|x^- - y^-| \leq |x - y|$, $|x^+ - y^+| \leq |x - y|$ and Gronwall's inequality ($u(t) \leq \alpha(t) \exp(\int_a^t c(s) ds)$ if $c$ is non-negative and $u$ satisfies $u(t) \leq \alpha(t) + \int_a^t c(s) u(s) ds$.) together show that
	\begin{align*}
	&||\diffuscal{X}(t) - \xi(t)||_{T} \\
	& = \left\Vert \diffuscal{X}^r(0) - \xi(0) + W(t) - \sigma w(t) + b^r t - \beta t + e(t) \right.\\
	& \left.\quad + \gamma \int_0^t \left( \diffuscal{X}(s)^- - \xi(s)^- \right) ds - \nu \int_0^t \left( \diffuscal{X}(s)^+ -  \xi(s)^+ \right) ds \right\Vert_{T}\\
	& \leq ||\diffuscal{X}(0) - \xi(0)||_{T} + || W(t) - \sigma w(t)||_{T} + ||b^r t - \beta t||_{T} + ||e(t)||_{T} \\
	& \quad + \gamma \left\Vert \int_0^t \diffuscal{X}(s)^- - \xi(s)^- ds\right\Vert_{T} + \nu \left\Vert \int_0^t \diffuscal{X}(s)^+ - \xi(s)^+ ds\right\Vert_{T}\\
	& \leq ||\diffuscal{X}(0) - \xi(0)||_{T} + || W(t) - \sigma w(t)||_{T} + ||b^r t - \beta t||_{T} + ||e(t)||_{T} \\
	& \quad + M \int_0^t \left\Vert \diffuscal{X}(s)^- - \xi(s)^- \right\Vert_{T} ds + M \int_0^t \left\Vert \diffuscal{X}(s)^+ - \xi(s)^+ \right\Vert_{T} ds\\
	& \leq ||\diffuscal{X}(0) - \xi(0)||_{T} + || W(t) - \sigma w(t)||_{T} + ||b^r t - \beta t||_{T} + ||e(t)||_{T} \\
	& \quad + 2M \int_0^t \left\Vert \diffuscal{X}(s) - \xi(s) \right\Vert_{T} ds,
	\end{align*}
	where $M=\max\{\gamma, \nu\}$. By \eqref{skorohod} and the uniform convergence of $e$ to zero, we have shown that $\diffuscal{X}(t)$ converges to $X(t)$ in probability, uniformly on $[0,T]$. Since $T$ is arbitrary, this shows that $\diffuscal{X}(t) \Rightarrow X(t)$.
\end{proof}

\section{Summary}
\viva{In this chapter, we show that diffusion limits are an effective way to approximate queuing processes because of its continuity feature. First we restate the limit theorem for many server queues with random service rates proved by \cite{Atar}. The key part of this theorem is its random drift $\beta$, which comes from the Central Limit Theorem applied on random service rates. Another thing which needs our attention is the coefficient $\gamma$ of the integral of negative part of the diffusion limit, i.e.\ $\int_0^t \xi(s)^- ds$. $\gamma \int_0^t \xi(s)^- ds$ approximates the capacities that are lost due to idleness. To some extent, $\gamma$ reflects fairness of the routing policy. We will talk about this fairness issue more in the end of next chapter.}

\viva{Then we extend the result of \cite{Atar} to systems with abandonments. This is an important extension as it ensures the stability of the diffusion limits. We use a martingale central limit theorem to prove weak convergence.}
\viva{In the next chapter, we show how the diffusion limits are applied to our optimal staffing problems.}

\chapter{Staffing and Routing for Single Server Pool Systems}

\viva{\section{Introduction}}
\viva{For call centre managers, how to decide on number of servers to be scheduled servers is one of the major problems. Overstaffing and understaffing will both cause immense unnecessary costs in the long term. Using queueing models, we can help managers make wise decisions on staffing levels. Particularly, the diffusion limits will be used in the analysis. In the last chapter, we proved diffusion limits for single server pool systems. Diffusion limits give us approximations of how the system processes behave. If steady states exist for diffusion limits, we can then have estimates for system steady states in the long run.} In this chapter, first we formulate the optimisation problem for staffing single server pools without abandonments. Then in the second section, we extend this result to systems with abandonments. In the third section, we focus on the variance of service rates and show how the variance influences the abandonment rate, and thus the total costs. Finally in the fourth section, we analyse the coefficient $\gamma$ in the diffusion limits and show how it reflects fairness among servers under different routing policies.

\section{Staffing many server queues with random servers}
We use a similar framework as in \cite{borst2004dimensioning} for asymptotic optimisation of many-server queuing systems with random service rates. Consider the many-server queueing model without abandonments. In the $\systemindex$th system, arrival rate $\lambda^{\systemindex}$, service rates $\mu_{\sindex}$ and its expectation $\bar{\mu}$ are as defined in Chapter 2.

Recall the second moment condition for arrival rate
\begin{equation}
\lim_{\systemindex \to \infty} \diffufactor (\lambda^{\systemindex} - \systemindex \lambda) = \hat{\lambda} < 0, 
\end{equation}
where $\lambda = \bar{\mu}$. It is easy to see that when $\sum_{\sindex=1}^{N^{\systemindex}} \mu_{\sindex} \leq \lambda^{\systemindex}$, the system is unstable, thus in the limit, every customer will have to wait before getting service. To this end, we assume there is a fixed waiting cost $C_{un}$ for unstable systems.

In this work, we will mainly focus on the scenario where the service rates satisfy $\sum_{\sindex=1}^{N^{\systemindex}} \mu_{\sindex} > \lambda^{\systemindex}$. For the diffusion limit in Theorem \ref{Atar-theorem}, this condition corresponds to $\beta <  0$. i.e.\ when the system is stable.  From \cite{cooper}, we know that, given the realisation of $\sum_{\sindex=1}^{N^{\systemindex}}\mu_{\sindex}$, the waiting time distribution is given by
\begin{equation}
\mathbb{P} \left( \mathrm{Wait}^{\systemindex} > t \middle|\ \sum_{\sindex=1}^{N^{\systemindex}} \mu_{\sindex} = H
\right) = \pi^{\systemindex} e^{-(H - \lambda^{\systemindex})t},
\end{equation} 
where $\pi^{\systemindex} = \mathbb{P} \left( \mathrm{Wait}^{\systemindex} > 0 \middle| \sum_{\sindex=1}^{N^{\systemindex}} \mu_{\sindex} = H \right)$ is the probability of waiting. Notice that $H$ has to be greater than $\lambda^{\systemindex}$ for stability.\par
Let $F(N^{\systemindex})$ be the staffing cost per unit time, and $D^{\systemindex}(t)$ be the waiting cost of a customer when s/he waits for $t$ time units. Without loss of generality we may take $D^{\systemindex}(0) = 0$. Then the conditional expected total cost per unit of time is given by
\begin{align*}
C\left(N^{\systemindex}, \lambda^{\systemindex} \middle|\ \sum_{\sindex=1}^{N^{\systemindex}} \mu_{\sindex} = H \right) 
&= F(N^{\systemindex}) + \lambda^{\systemindex} \mathbb{E}\left( D^{\systemindex}(\mathrm{Wait}) \middle| \sum_{\sindex=1}^{N^{\systemindex}} \mu_{\sindex} = H \right)\\
&= F(N^{\systemindex}) + \lambda^{\systemindex} \pi^{\systemindex} G\left(N^{\systemindex}, \lambda^{\systemindex} \middle| \sum_{\sindex=1}^{N^{\systemindex}} \mu_{\sindex} = H \right),
\end{align*}
where
\begin{align*}
G\left(N^{\systemindex}, \lambda^{\systemindex} \middle| \sum_{\sindex=1}^{N^{\systemindex}} \mu_{\sindex} = H \right) 
&= \mathbb{E} \left( D^{\systemindex}(\mathrm{Wait}) \middle| \mathrm{Wait}>0, \sum_{\sindex=1}^{N^{\systemindex}} \mu_{\sindex} = H \right)\\
&= (H - \lambda^{\systemindex}) \int_0^{\infty} D^{\systemindex}(t) e^{-(H - \lambda^{\systemindex})t} dt. \numberthis \label{G1}
\end{align*}
We are interested in determining the expected optimal staffing level
\begin{equation}
N^{\systemindex}_* := \argmin_{N^{\systemindex} > \lambda^{\systemindex}/\bar{\mu}} C(N^{\systemindex}, \lambda^{\systemindex}), \label{Ob1}
\end{equation}
where $C(N^{\systemindex}, \lambda^{\systemindex}) = \int_{\lambda^{\systemindex}}^{\infty} C\left(N^{\systemindex}, \lambda^{\systemindex} \middle|\ \sum_{\sindex=1}^{N^{\systemindex}} \mu_{\sindex} = H \right) f_{\systemindex}(H) dH \frac{1}{\mathbb{P}(\sum_{\sindex=1}^{N^{\systemindex}} \mu_{\sindex} > \lambda^{\systemindex})}$, and $f_{\systemindex}(\cdot)$ is the density function of $\sum_{\sindex=1}^{N^{\systemindex}} \mu_{\sindex}$.

\subsection{Framework of the asymptotic optimisation problem}
We use a similar framework as in \cite{borst2004dimensioning}, \viva{where, the cost function contains an expected waiting cost, which is the product of the arrival rate and the expected waiting time of a single customer. We use the same concept, but our expected waiting cost contains two recursive expectations instead of one. The first one is the normal expectation of the waiting cost when the sum of the service rate is given, then we take the second expectation over the sum.} As the first step, we translate the discrete optimisation problem \eqref{Ob1} to a continuous one, and approximate the latter problem by a related continuous version, which is easier to solve. Finally, we prove that the optimal solution to the approximating continuous problem provides an asymptotically optimal solution to the original discrete problem. \viva{Our main contribution is in the last step. To show the validity of the continuous approximation, we need to prove that limits are interchangeable as shown in Figure \ref{figure2}. And to show this, we prove tightness of the sequence of steady state distributions.}\par
We first translate the discrete problem into a continuous one. Let
\begin{equation}
N^{\systemindex}(x) = \offeredload + x \sqrt{\offeredload},
\end{equation}
so that the variable $x=(N^{\systemindex} - \offeredload)/\sqrt{\offeredload}$ is the normalized number of servers in excess of the minimum number $\offeredload$ required for stability. In terms of $x$, we define
\begin{align*}
&F^{\systemindex}(x):= F(N^{\systemindex}(x)) - F \left(\offeredload \right),\\
&G^{\systemindex} \left( x \middle| \sum_{\sindex=1}^{N^{\systemindex}(x)} \mu_{\sindex} = H \right):= \lambda^{\systemindex} G \left(N^{\systemindex}(x), \lambda^{\systemindex} \middle| \sum_{\sindex=1}^{N^{\systemindex}(x)} \mu_{\sindex} = H \right),\\
&C^{\systemindex} \left(x \middle| \sum_{\sindex=1}^{N^{\systemindex}(x)} \mu_{\sindex} = H \right):= C \left(N^{\systemindex}(x), \lambda^{\systemindex} \middle| \sum_{\sindex=1}^{N^{\systemindex}(x)} \mu_{\sindex} = H \right) - F \left(\offeredload \right),\\
&\pi^{\systemindex} \left(x \middle| \sum_{\sindex=1}^{N^{\systemindex}(x)} \mu_{\sindex} = H \right):= \text{ The probability of waiting in the } \systemindex\text{th} \text{ system}\\
& \qquad \qquad \qquad \qquad \qquad \qquad \text{given the sum of service rates to be } H.
\end{align*}
Then the total cost per unit of time can be rewritten as 
\begin{equation}
C^{\systemindex}(x) = F^{\systemindex}(x) + \int_{\lambda^{\systemindex}}^{\infty} \pi^{\systemindex} \left(x \middle| \sum_{\sindex=1}^{N^{\systemindex}(x)} \mu_{\sindex} = H \right) G^{\systemindex} \left(x \middle| \sum_{\sindex=1}^{N^{\systemindex}(x)} \mu_{\sindex} = H \right) f_{\systemindex}(H) dH \frac{1}{\mathbb{P}(\sum_{\sindex=1}^{N^{\systemindex}(x)} \mu_{\sindex} > \lambda^{\systemindex})}. \label{cost2}
\end{equation}
Denote 
\begin{equation}
x^{\systemindex}_*:= \argmin_{x>0} C^{\systemindex}(x).
\end{equation}

Next, we a use simpler version of continuous function to approximate the function $C^{\systemindex}(x)$.
If we can find simpler approximations for both $\pi^{\systemindex}$ and $G^{\systemindex}$, we will have an approximation for the new cost function \eqref{cost2}. In the next section, we will show the feasibility of such approximations.

\vspace{-4mm}
\subsection{Validity of the approximating model}
In this section, we provide lemmas with proofs to show that we can use a simpler function to estimate the continuous cost function \eqref{cost2}.

\subsubsection{Approximation of $\pi^r$}
First, we will find an approximating version for $\pi^{\systemindex}(x \mid \sum_{\sindex=1}^{N^{\systemindex}(x)} \mu_{\sindex} = H^{\systemindex})$. We achieve this by calculating the steady state of the diffusion limit in Theorem 2.1 in \cite{Atar}. As we discussed before, when $\beta > 0$, the diffusion $\xi(t)$ is not stable, so here we focus on the situation when $\beta<0$, under which the limiting diffusion $\xi(t)$ has following expressions
\begin{equation}
\xi(t) = \twopartdef{\xi(0) + \sigma w(t) + \beta t}{\xi(t) \geq 0}{\xi(0) + \sigma w(t) + \beta t -\gamma \int_0^t \xi(s) ds}{\xi(t) <0},
\end{equation}
where
\begin{equation}
\sigma^2 = \bar{\mu}(C^2_{\check{U}} + 1), \text{ and } \gamma = \frac{\int x^2 d m}{\int x d m}. 
\end{equation}
Then by Section 4 of \cite{browne}, when $\xi(t) \geq 0$, it is a reflected Brownian motion. Thus its steady state density function conditional on $\xi(\infty) \geq 0$ is exponential, and will be $g_e(x)= \frac{-2\beta}{\sigma^2} e^{\frac{2\beta}{\sigma^2}x}$. Similarly, the process $\xi(t)$ restricted to the negative half-line is an O-U process , thus its steady state conditional on $\xi(\infty) < 0$ is normally distributed with density function $g_{\systemindex}(x)= \frac{\sqrt{2\gamma}/\sigma \phi(\frac{\sqrt{2\gamma}}{\sigma}x -\frac{\sqrt{2}\beta}{\sqrt{\gamma}\sigma})}{\Phi(\frac{-\sqrt{2}\beta}{\sqrt{\gamma}\sigma})}$. Hence, let $\varrho$ be the probability of waiting when $\beta < 0$, the steady state distribution of $\xi(t)$ has density function
\begin{equation}
f(x) = \twopartdef{g_e(x) \varrho,}{x \geq 0}{g_{\systemindex}(x) (1 - \varrho),}{x < 0}.
\end{equation}
Since the variance of the limiting diffusion is a constant on the real line,  by \cite{browne}, the density function of its steady state should be continuous. $\varrho$ may be solved by equating the right-handed limit of $g_e(\cdot) \varrho$ and the left-handed limit of $g_{\systemindex}(\cdot) (1 - \varrho)$ at $0$. This gives us 
\begin{equation}
\varrho = P(y) = \mathbb{P}(\xi(\infty) \geq 0 \mid \beta=y) = \left( 1 - \frac{\frac{\sqrt{2}y}{\sqrt{\gamma} \sigma} \Phi(-\frac{\sqrt{2}y}{\sqrt{\gamma} \sigma})}{\phi(\frac{\sqrt{2}y}{\sqrt{\gamma} \sigma})} \right)^{-1}. \label{PoWlimit}
\end{equation}
Thus we will have an analogue of Proposition 1 in \cite{sqrt} and Lemma 5.1 in \cite{borst2004dimensioning}, \viva{which gives us an approximation of $\pi^{\systemindex}$.} 
\begin{lemma}
	For any function $x^{\systemindex}$ with $\limsup_{\systemindex \to \infty} x^{\systemindex} < \infty$,
	\begin{equation}
	\lim_{\systemindex \to \infty} \frac{ \pi^{\systemindex} \left( x^{\systemindex} \middle| \sum_{\sindex=1}^{N^{\systemindex}(x^{\systemindex})} \mu_{\sindex} = H \right) }{ P \left(y^{\systemindex} \right) } = 1, \ \forall H > N^{\systemindex}(x^{\systemindex})\bar{\mu} - x \bar{\mu} \sqrt{\systemindex}, \label{estimatePoW}
	\end{equation}
	where $y^{\systemindex} = -\left( H - N^{\systemindex}(x^{\systemindex}) \bar{\mu} \right) \diffufactor - x^{\systemindex} \bar{\mu}$.
	\label{lemmaestimatePoW}
\end{lemma}
\begin{figure}
	\centering{ \includegraphics[scale=1.4]{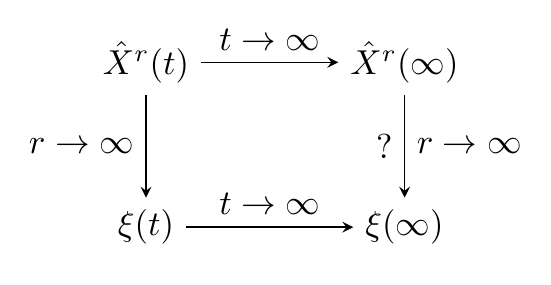} }
	\caption{The interchange-of-limit diagram} \label{figure2}
\end{figure}

\viva{Lemma 5.1 in \cite{borst2004dimensioning} is a direct application of Proposition 1 in \cite{sqrt}, which says that, under the Halfin-Whitt heavy traffic condition, the probability of waiting in many server queue (identical servers) converges to a constant between zero and one. For our systems, service rates are no longer identical, and thus we cannot use this result directly. We come up with another approach to prove the convergence, which involves an associated sequence of homogeneous systems.}

\viva{The idea of the proof lies in Figure \ref{figure2}. In order to show the convergence of the probability of waiting in \eqref{estimatePoW}, we need to show that the sequence of steady states $\hat{X}^{\systemindex}(\infty)$ weakly converges to the steady state $\xi(\infty)$ of the diffusion limit. To show weak convergence, the usual way is to first show tightness, then use the relations in Figure \ref{figure2} to show that tightness of $\hat{X}^{\systemindex}(\infty)$ actually reflects the convergence on the right hand side.}\par
\viva{The challenging part is to show the tightness. It is hard to come up with a direct way to show tightness if service rates are random and unknown. Instead, we compare our heterogeneous systems with a sequence of homogeneous systems which have identical service rates being less than $\bar{\mu}$, then use the properties of homogeneous systems to get tightness results.} \par
\viva{More specifically, consider homogeneous systems with $N^{\systemindex}$ servers and identical service rates $\mu^{\systemindex} < \bar{\mu}$. Assume there is a fixed number $M^{\systemindex} < N^{\systemindex}$ for each $\systemindex$, and denote $\Delta_{M^{\systemindex}}$ as the family of sets containing $M^{\systemindex}$ numbers out of $\{1,2,\dots,N^{\systemindex}\}$. Now our problem can be considered in the following two cases:
\begin{itemize}
	\item
	When $\inf_{\delta \in \Delta_{M^{\systemindex}}} \sum_{\delta} \mu_{\sindex} < M^{\systemindex} \mu^{\systemindex}$, denote this scenario as $A_{\systemindex}$. (This $A_r$ is \viva{independent from} the arrival process $A^{\systemindex}(t)$.) There is a possibility that the heterogeneous systems serve faster than their corresponding homogeneous ones. 
	\item
	When $\inf_{\delta \in \Delta_{M^{\systemindex}}} \sum_{\delta} \mu_{\sindex} \geq M^{\systemindex} \mu^{\systemindex}$, this scenario is marked as $A_{\systemindex}^c$. The heterogeneous systems always serve faster than their corresponding homogeneous ones. Thus under such situations, $X^{\systemindex}(t) \leq_{st} X_{hom}^{\systemindex}(t)$. Therefore to show tightness of heterogeneous systems for this scenario, we only need to show the tightness of their corresponding homogeneous systems, i.e.\ 
	\begin{equation}
	\forall \epsilon>0, \exists K^{\epsilon}>0, \text{ s.t. } \forall \systemindex, \mathbb{P}(X_{hom}^{\systemindex}(t) \geq K^{\epsilon}) \leq \epsilon. \label{homotight}
	\end{equation}
\end{itemize}
Combining these two scenarios, let $\bar{K^{\epsilon}} = K^{\epsilon} \vee M^{\systemindex}$. Then
\begin{align*}
\mathbb{P}(X^{\systemindex}(t) \geq \bar{K^{\epsilon}})
&= \mathbb{P}(X^{\systemindex}(t) \geq \bar{K^{\epsilon}}, A_{\systemindex}) + \mathbb{P}(X^{\systemindex}(t) \geq \bar{K^{\epsilon}}, A_{\systemindex}^c)\\
&\leq \mathbb{P}(A_{\systemindex}) + \mathbb{P}(X^{\systemindex}_{hom}(t) \geq \bar{K^{\epsilon}}). \numberthis \label{tightness1}
\end{align*}
If \eqref{homotight} is true, then after a reselection of $\epsilon$, we can easily show $\mathbb{P}(X^{\systemindex}_{hom}(t) \geq \bar{K^{\epsilon}}) < \epsilon/2$. Furthermore, if we can also show $\mathbb{P}(A_{\systemindex}) \to 0$ as $\systemindex \to \infty$, then $\mathbb{P}(X^{\systemindex}(t) \geq \bar{K^{\epsilon}}) \to 0$ as $\systemindex \to \infty, \forall t>0$, the tightness result will follow. Since the tightness result holds for every $t$, it should also hold as $t \to \infty$, i.e.\ $\mathbb{P}(X^{\systemindex}(\infty) \geq \bar{K^{\epsilon}}) < \epsilon$.}\par
\viva{From the discussion above, it is important to define proper $\mu^r$ and $M^r$ to get our convergence results. We have three lemmas for that.
\begin{lemma}
Let $\mu^r= \bar{\mu} - \frac{1}{r^p}$ for any $p>\frac{1}{2}$, then the sequence $\{X^r_{hom}(\infty)\}$ is tight. \label{homo-tightness}
\end{lemma}
}
\begin{proof}
		In order to show $\{X^r_{hom}(\infty)\}$ is tight, first we need to specify the heavy traffic condition for the homogeneous systems and the existence of their steady states.\par
		When
		\begin{equation}
		\mu^{\systemindex} = \bar{\mu} - \frac{1}{{\systemindex}^p},  p>\frac{1}{2},
		\label{IDrate}
		\end{equation} 
		the heavy traffic condition becomes
		\begin{align*}
		&\lim_{\systemindex \to \infty} \diffufactor (\systemindex \mu^{\systemindex} - \lambda^{\systemindex}) = \lim_{\systemindex \to \infty} \diffufactor \left(\systemindex \left(\bar{\mu} - \frac{1}{{\systemindex}^p}\right) - \lambda^{\systemindex} \right)\\
		& \quad = \lim_{\systemindex \to \infty} \diffufactor \left(\systemindex \bar{\mu} - \lambda^{\systemindex} - {\systemindex}^{1-p}\right) = \lim_{\systemindex \to \infty} \left( \diffufactor \left(\systemindex \bar{\mu} - \lambda^{\systemindex} \right) - {\systemindex}^{\frac{1}{2}-p} \right),\\
		& \quad = \lim_{\systemindex \to \infty} \diffufactor \left(\systemindex \bar{\mu} - \lambda^{\systemindex} \right) = \hat{\lambda}>0.
		\end{align*}
		This also means $\systemindex \bar{\mu} > \lambda^{\systemindex}, \forall \systemindex$.\par
		To guarantee the existence of their steady states, we need 
		\begin{equation*}
		\systemindex \mu^{\systemindex} = \systemindex \left(\bar{\mu} - \frac{1}{{\systemindex}^p} \right) > \lambda^{\systemindex},
		\end{equation*}
		i.e.\
		\begin{equation}
		\systemindex \bar{\mu} - {\systemindex}^{1-p} > \lambda^{\systemindex}, \label{homostable}
		\end{equation}
		and $p > \frac{1}{2}$ makes the inequality above hold. \par
		Since we already showed the existence of $X^r_{hom}(\infty)$, now we can prove the tightness of it, i.e.\ \eqref{homotight} is true when $t \to \infty$. Denote $\rho^{\systemindex} = \frac{\lambda^{\systemindex}}{ \systemindex \mu^{\systemindex}}$, then $\rho^{\systemindex} < 1, \forall \systemindex$. Fix $\systemindex$. We want to show
		\begin{equation}
		\forall \epsilon>0, \exists K^{\epsilon}_{\systemindex}>0, \text{ s.t. } \mathbb{P}(X_{hom}^{\systemindex}(\infty) \geq K^{\epsilon}_{\systemindex}) \leq \epsilon. \label{homotight1}
		\end{equation}
		Choose $K^{\epsilon}_{\systemindex} > N^{\systemindex}$. By (1.1) and (1.3) in \cite{sqrt}, we have
		\begin{align*}
		& \mathbb{P}(X^{\systemindex}_{hom}(\infty) \geq K^{\epsilon}_{\systemindex}) = \sum_{k = K^{\epsilon}_{\systemindex}}^{\infty} \frac{(N^{\systemindex})^{N^{\systemindex}} (\rho^{\systemindex})^k}{N^{\systemindex}!} \eta =  \eta \frac{(N^{\systemindex})^{N^{\systemindex}}}{N^{\systemindex}!} \sum_{k = K^{\epsilon}_{\systemindex}}^{\infty}(\rho^{\systemindex})^k ,
		\end{align*}
		where $\eta = \left( \frac{(N^{\systemindex} \rho^{\systemindex})^{N^{\systemindex}}}{N^{\systemindex}!(1 - \rho^{\systemindex})} + \sum_{k=0}^{N^{\systemindex}-1} \frac{(N^{\systemindex} \rho^{\systemindex})^k}{k!} \right)^{-1}$. Substitute $\eta$ into the equation above, and since $0 < \rho^{\systemindex} < 1$, it becomes
		\begin{align*}
		& \mathbb{P}(X^{\systemindex}_{hom}(\infty) \geq K^{\epsilon}_{\systemindex}) = \eta \frac{(N^{\systemindex})^{N^{\systemindex}}}{N^{\systemindex}!} \frac{(\rho^{\systemindex})^{K^{\epsilon}_{\systemindex}}(1 - (\rho^{\systemindex})^x)}{1 - \rho^{\systemindex}} \xrightarrow{x \to \infty} \eta \frac{(N^{\systemindex})^{N^{\systemindex}}}{N^{\systemindex}!} \frac{(\rho^{\systemindex})^{K^{\epsilon}_{\systemindex}}}{1 - \rho^{\systemindex}}\\
		& = \frac{(N^{\systemindex})^{N^{\systemindex}}}{N^{\systemindex}!} \left( \frac{(N^{\systemindex} \rho^{\systemindex})^{N^{\systemindex}}}{N^{\systemindex}!(1 - \rho^{\systemindex})} + \sum_{k=0}^{N^{\systemindex}-1} \frac{(N^{\systemindex} \rho^{\systemindex})^k}{k!} \right)^{-1} \frac{(\rho^{\systemindex})^{K^{\epsilon}_{\systemindex}}}{1 - \rho^{\systemindex}}\\
		& = \left( \frac{(\rho^{\systemindex})^{N^{\systemindex}}}{1 - \rho^{\systemindex}} + \frac{N^{\systemindex}!}{(N^{\systemindex})^{N^{\systemindex}}} \sum_{k=0}^{N^{\systemindex} - 1} \frac{(N^{\systemindex})^k}{k!} (\rho^{\systemindex})^k \right)^{-1} \frac{(\rho^{\systemindex})^{K^{\epsilon}_{\systemindex}}}{1 - \rho^{\systemindex}}\\
		& = \left( (\rho^{\systemindex})^{N^{\systemindex}} + (1 - \rho^{\systemindex}) \frac{N^{\systemindex}!}{(N^{\systemindex})^{N^{\systemindex}}} \sum_{k=0}^{N^{\systemindex} - 1} \frac{(N^{\systemindex})^k}{k!} (\rho^{\systemindex})^k \right)^{-1} (\rho^{\systemindex})^{K^{\epsilon}_{\systemindex}}.
		\end{align*}
		$\forall \epsilon > 0$, as long as 
		\begin{align*}
		K^{\epsilon}_{\systemindex} > \log \left( \epsilon \left( \frac{(\rho^{\systemindex})^{N^{\systemindex}}}{1 - \rho^{\systemindex}} + \frac{N^{\systemindex}!}{(N^{\systemindex})^{N^{\systemindex}}} \sum_{k=0}^{N^{\systemindex} - 1} \frac{(N^{\systemindex})^k}{k!} (\rho^{\systemindex})^k \right) \right),
		\end{align*}
		$\mathbb{P}(X^{\systemindex}_{hom}(\infty) \geq K^{\epsilon}_{\systemindex}) < \epsilon$ holds. Here $\log$ has base $\rho^{\systemindex}$. Thus \eqref{homotight1} is proved.\par
		For a sequence of homogeneous systems, take $K^{\epsilon} = K^{\epsilon}_{1} \vee K^{\epsilon}_{2} \vee \cdots \vee K^{\epsilon}_{n} \vee \cdots$, then \eqref{homotight} holds when $t \to \infty$.
\end{proof}

\viva{From Lemma \ref{homo-tightness}, in order to guarantee heavy traffic condition and stability of the homogeneous systems, service rate of homogeneous systems is chosen with  
\begin{equation}
p > \frac{1}{2}
\end{equation} 
in the rest of this chapter.
}

\viva{As for the choice of $M^r$, we will explain in the following two lemmas that when $M^r$ grows slower than $N^r$, the convergence does not hold. Only when $M^r$ grows with the same rate as $N^r$, the probability $\mathbb{P}(A_r)$ converges to zero. }

\viva{
\begin{lemma}
Let $\mu^r= \bar{\mu} - \frac{1}{r^p}$. If $M^{\systemindex}$ grows slower than $N^{\systemindex}$ in the sense that $\lim_{\systemindex \to \infty} \frac{M^{\systemindex}}{N^{\systemindex}} = 0$, then $\mathbb{P}(A_r) \to \infty$. \label{divergence-P-o-A_r}
\end{lemma}	
}
\vspace{-5mm}
\begin{proof}
	Fix $\systemindex$. Since $\delta$ is one element in $\Delta_{M^{\systemindex}}$, denote $B_{\delta} = \{ \sum_{\delta} \mu_{\sindex} < M^{\systemindex} \mu^{\systemindex} \}$, we have
	\begin{equation}
	\mathbb{P}(A_{\systemindex}) = \mathbb{P}\left(\bigcup_{\delta} B_{\delta} \right) \leq \sum_{\delta \in \Delta_{M^{\systemindex}}} \mathbb{P}(B_{\delta}) = {N^{\systemindex} \choose M^{\systemindex}} \mathbb{P}(B_{\delta}).  \label{PA_rUB}
	\end{equation}
	Using stirling's approximation, we have a lower bound for the combination factor
	\begin{align*}
	{N^{\systemindex} \choose M^{\systemindex}} 
	& = \frac{N^{\systemindex}!}{M^{\systemindex}!(N^{\systemindex} - M^{\systemindex})!} \geq \frac{\sqrt{2 \pi} (N^{\systemindex})^{N^{\systemindex} + \frac{1}{2}} e^{-N^{\systemindex}}}{e (M^{\systemindex})^{M^{\systemindex} + \frac{1}{2}} e^{-M^{\systemindex}} e (N^{\systemindex} - M^{\systemindex})^{N^{\systemindex} - M^{\systemindex} + \frac{1}{2}} e^{-(N^{\systemindex} - M^{\systemindex})}}\\
	& = \frac{\sqrt{2 \pi}}{e^2} \frac{(N^{\systemindex})^{N^{\systemindex} + \frac{1}{2}}}{(M^{\systemindex})^{M^{\systemindex} + \frac{1}{2}} (N^{\systemindex} - M^{\systemindex})^{N^{\systemindex} - M^{\systemindex} + \frac{1}{2}}}. \numberthis \label{combLB}
	\end{align*}
	\vspace{-3mm}
	Thus as $\systemindex \to \infty$, 
	\begin{align*}
	{N^{\systemindex} \choose M^{\systemindex}} 
	&\mathbb{P} \left(\sum_{\delta} \frac{\mu_{\sindex}}{M^{\systemindex}} < \mu^{\systemindex} \right) \to {N^{\systemindex} \choose M^{\systemindex}} \mathbb{P}(Y^{\systemindex} < \mu^{\systemindex})\\
	& = \frac{N^{\systemindex}!}{M^{\systemindex}!(N^{\systemindex}-M^{\systemindex})!} \frac{1}{\sqrt{2 \pi}} \int_{-\infty}^{-\frac{\sqrt{M^{\systemindex}}}{\sigma {\systemindex}^p}} e^{-\frac{t^2}{2}} dt\\
	& \geq \frac{\sqrt{2 \pi}}{e^2} \frac{(N^{\systemindex})^{N^{\systemindex} + \frac{1}{2}}}{(M^{\systemindex})^{M^{\systemindex} + \frac{1}{2}} (N^{\systemindex} - M^{\systemindex})^{N^{\systemindex} - M^{\systemindex} + \frac{1}{2}}} \frac{1}{\sqrt{2 \pi}} \int_{-\infty}^{-\frac{\sqrt{M^{\systemindex}}}{\sigma {\systemindex}^p}} e^{-\frac{t^2}{2}} dt\\
	& = \frac{1}{e^2} \frac{(N^{\systemindex})^{M^{\systemindex}}}{(M^{\systemindex})^{M^{\systemindex}}} \frac{(N^{\systemindex})^{N^{\systemindex} - M^{\systemindex}}}{(N^{\systemindex} - M^{\systemindex})^{N^{\systemindex} - M^{\systemindex}}} \left( \frac{N^{\systemindex}}{M^{\systemindex}(N^{\systemindex} - M^{\systemindex})} \right)^{\frac{1}{2}} \int_{-\infty}^{-\frac{\sqrt{M^{\systemindex}}}{\sigma {\systemindex}^p}} e^{-\frac{t^2}{2}} dt. \numberthis \label{probLB}  
	\end{align*}
	Notice that the second item $(\frac{N^{\systemindex}}{M^{\systemindex}})^{M^{\systemindex}} \to \infty$ as $\systemindex \to \infty$, the third item can be considered as follows:
	\begin{align*}
	& \left( \frac{N^{\systemindex}}{N^{\systemindex} - M^{\systemindex}} \right)^{(N^{\systemindex} - M^{\systemindex})} = \left( 1 - 1 + \frac{N^{\systemindex}}{N^{\systemindex} - M^{\systemindex}} \right)^{(N^{\systemindex} - M^{\systemindex})} = \left( 1 + \frac{M^{\systemindex}}{N^{\systemindex} - M^{\systemindex}} \right)^{(N^{\systemindex} - M^{\systemindex})}\\
	& \quad = \left( 1 + \frac{1}{\frac{N^{\systemindex}}{M^{\systemindex}} - 1} \right)^{(\frac{N^{\systemindex}}{M^{\systemindex}} - 1) M^{\systemindex}} \to \exp(M^{\systemindex}) \text{ as } \systemindex \to \infty, 
	\end{align*}
	and the fourth item equals  $\sqrt{\frac{1}{M^{\systemindex} (1 - \frac{M^{\systemindex}}{N^{\systemindex}})}}$, where $\sqrt{\frac{1}{1 - \frac{M^{\systemindex}}{N^{\systemindex}}}}$ converges to $1$. As for the integral, the upper limit is equal to 
	\begin{align*}
	-\frac{\sqrt{M^{\systemindex}}}{\sqrt{N^{\systemindex}}} \frac{\sqrt{N^{\systemindex}}}{\sqrt{\systemindex}} \frac{\sqrt{\systemindex}}{{\systemindex}^p} = -\frac{\sqrt{M^{\systemindex}}}{\sqrt{N^{\systemindex}}} \frac{\sqrt{N^{\systemindex}}}{\sqrt{\systemindex}} \frac{1}{{\systemindex}^{p-\frac{1}{2}}} \to 0 
	\end{align*}
	as $\systemindex \to \infty$, because $p > \frac{1}{2}$. So as $\systemindex \to \infty$, \eqref{probLB} is equivalent to 
	\begin{align*}
	& \frac{1}{e^2} \left(\frac{N^{\systemindex}}{M^{\systemindex}}\right)^{M^{\systemindex}}  \left( 1 + \frac{1}{\frac{N^{\systemindex}}{M^{\systemindex}} - 1} \right)^{(\frac{N^{\systemindex}}{M^{\systemindex}} - 1) M^{\systemindex}} \sqrt{\frac{1}{M^{\systemindex}}} \sqrt{\frac{1}{1 - \frac{M^{\systemindex}}{N^{\systemindex}}}} \int_{-\infty}^{-\frac{\sqrt{M^{\systemindex}}}{\sqrt{N^{\systemindex}}} \frac{\sqrt{N^{\systemindex}}}{\sqrt{\systemindex}} \frac{1}{{\systemindex}^{p-\frac{1}{2}}}} e^{-\frac{t^2}{2}} dt\\
	& \to \frac{1}{e^2} \left( \frac{N^{\systemindex}}{M^{\systemindex}} \right)^{M^{\systemindex}} \exp(M^{\systemindex}) \frac{1}{\sqrt{M^{\systemindex}}} \int_{-\infty}^{0} e^{-\frac{t^2}{2}} dt\\
	& \to \frac{1}{e^2} \cdot \infty \cdot \infty \cdot \frac{1}{2} \to \infty. \numberthis \label{probLB2}
	\end{align*}
	Thus the original ${N^{\systemindex} \choose M^{\systemindex}} \mathbb{P} \left(\sum_{\delta} \frac{\mu_{\sindex}}{M^{\systemindex}} < \mu^{\systemindex} \right)$ will not converge to zero, and therefore we rule out this situation. 	
\end{proof}

\viva{Now we conclude that there must be $M^{\systemindex} = C^{\systemindex}_3$ for some $C^{\systemindex}_3$ such that $\mathbb{P}(A_{\systemindex}) \to 0$ holds.}
\begin{lemma}
Let $\mu^r= \bar{\mu} - \frac{1}{r^p},$ and $M^r = C_3^r N^r$ for some $C_3^r \in (0,1)$. Then $\mathbb{P}(A_r) \to 0$ as $\systemindex \to \infty$. \label{convergence-P-o-A_r}
\end{lemma}
\begin{proof}
		Assume $M^{\systemindex} \to \infty$ as $\systemindex \to \infty$. Then $\frac{1}{M^{\systemindex}}\sum_{\delta} \mu_{\sindex}$ weakly converges to a normal random variable $Y^{\systemindex}$ with mean $\bar{\mu}$, variance $\frac{\sigma^2}{M^{\systemindex}}$ as $\systemindex \to \infty, \forall \delta$. Using the left bound for normal distribution,
		\begin{align*}
		\frac{1}{\sqrt{2\pi}} \frac{x}{x^2 + 1} e^{-\frac{x^2}{2}} \leq \mathbb{P}(X>x) \leq \frac{1}{\sqrt{2\pi}} \frac{1}{x} e^{-\frac{x^2}{2}}.
		\end{align*}
		As $\systemindex \to \infty$, we have
		\begin{align*}
		\mathbb{P}(B_{\delta}) 
		&\to \mathbb{P}(Y^{\systemindex} < \mu^{\systemindex}) = \mathbb{P} \left( \frac{\sqrt{M^{\systemindex}}}{\sigma}(Y^{\systemindex} - \bar{\mu})< \frac{\sqrt{M^{\systemindex}}}{\sigma} (\mu^{\systemindex} - \bar{\mu}) \right)\\
		&\leq 1 - \frac{1}{\sqrt{2 \pi}} \frac{\frac{\sqrt{M^{\systemindex}}}{\sigma}(\mu^{\systemindex} - \bar{\mu} )}{\frac{M^{\systemindex}}{\sigma^2}(\mu^{\systemindex} - \bar{\mu} )^2 +1} \exp \left(-\frac{M^{\systemindex}(\mu^{\systemindex} - \bar{\mu})^2}{2\sigma^2}\right), \numberthis \label{estprob1}
		\end{align*}
		then applying stirling's approximation again on the combination factor in \eqref{PA_rUB} gives
		\begin{align*}
		{N^{\systemindex} \choose M^{\systemindex}} = \frac{N^{\systemindex}!}{M^{\systemindex}!(N^{\systemindex} - M^{\systemindex})!} \leq \frac{e (N^{\systemindex})^{(N^{\systemindex} + \frac{1}{2})}}{2 \pi (M^{\systemindex})^{(M^{\systemindex} + \frac{1}{2})} (N^{\systemindex} - M^{\systemindex})^{(N^{\systemindex} - M^{\systemindex} + \frac{1}{2})}}. \numberthis \label{combUB}
		\end{align*}
		Combine \eqref{estprob1} and \eqref{combUB} and we have, as $\systemindex \to \infty$,
		\begin{align*}
		\mathbb{P}(A_{\systemindex})
		& = {N^{\systemindex} \choose M^{\systemindex}} \mathbb{P} \left(\sum_{\gamma} \frac{\mu_{\sindex}}{M^{\systemindex}} < \mu^{\systemindex} \right) \to {N^{\systemindex} \choose M^{\systemindex}} \mathbb{P}(Y^{\systemindex} < \mu^{\systemindex})\\
		& = {N^{\systemindex} \choose M^{\systemindex}} \mathbb{P} \left( \frac{\sqrt{M^{\systemindex}}}{\sigma}(Y^{\systemindex} - \bar{\mu})< \frac{\sqrt{M^{\systemindex}}}{\sigma} (\mu^{\systemindex} - \bar{\mu}) \right)\\
		& \leq \frac{e (N^{\systemindex})^{(N^{\systemindex} + \frac{1}{2})}}{2 \pi (M^{\systemindex})^{(M^{\systemindex} + \frac{1}{2})} (N^{\systemindex} - M^{\systemindex})^{(N^{\systemindex} - M^{\systemindex} + \frac{1}{2})}} \\
		& \qquad \quad \Bigg( 1 - \frac{1}{\sqrt{2 \pi}} \frac{\frac{\sqrt{M^{\systemindex}}}{\sigma}(\mu^{\systemindex} - \bar{\mu} )}{\frac{M^{\systemindex}}{\sigma^2}(\mu^{\systemindex} - \bar{\mu} )^2 +1} \exp \left(-\frac{M^{\systemindex}(\mu^{\systemindex} - \bar{\mu})^2}{2\sigma^2}\right) \Bigg). \numberthis \label{probUB1}
		\end{align*}
		For notational simplicity let $C_2^{\systemindex} = \frac{\mu^{\systemindex} - \bar{\mu}}{ \sigma} = -\frac{1}{\sigma {\systemindex}^p} <0$. Since \eqref{probUB1} is an upper bound of $\mathbb{P} (A_{\systemindex})$, if we can demonstrate \eqref{probUB1} converges to zero as $\systemindex \to \infty$, then $\mathbb{P}(A_{\systemindex}) \to 0$ as $\systemindex \to \infty$. \\
        Assume $C_3^{\systemindex} + C_4^{\systemindex} = 1$ for some $0 < C_4^{\systemindex} < 1$ and $M^{\systemindex} = C_3^{\systemindex} N^{\systemindex}$, then the right hand side of \eqref{probUB1} can be rewritten as 
			\begin{align*}
			& \frac{e(N^{\systemindex})^{(N^{\systemindex} + \frac{1}{2})}}{2 \pi (C_3^{\systemindex} N^{\systemindex})^{(C_3^{\systemindex} N^{\systemindex} + \frac{1}{2} )} (N^{\systemindex} - C_3^{\systemindex} N^{\systemindex})^{(N^{\systemindex} - C_3^{\systemindex} N^{\systemindex} + \frac{1}{2})}}\\
			& \quad - \frac{e(N^{\systemindex})^{(N^{\systemindex} + \frac{1}{2})}}{2 \pi (C_3^{\systemindex} N^{\systemindex})^{(C_3^{\systemindex} N^{\systemindex} + \frac{1}{2} )} (N^{\systemindex} - C_3^{\systemindex} N^{\systemindex})^{(N^{\systemindex} - C_3^{\systemindex} N^{\systemindex} + \frac{1}{2})}} \frac{1}{\sqrt{2 \pi}} \frac{\sqrt{C_3^{\systemindex} N^{\systemindex}} C_2^{\systemindex}}{C_3^{\systemindex} N^{\systemindex} (C_2^{\systemindex})^2 + 1} \exp \left(- \frac{1}{2} C_3^{\systemindex} N^{\systemindex} (C_2^{\systemindex})^2 \right)\\
			& = \frac{e (N^{\systemindex})^{N^{\systemindex}} (N^{\systemindex})^{\frac{1}{2}} }{2 \pi \left((C_3^{\systemindex})^{C_3^{\systemindex}}\right)^{N^{\systemindex}} \left( (N^{\systemindex})^{N^{\systemindex}} \right)^{C_3^{\systemindex}} (C_3^{\systemindex})^{\frac{1}{2}} (N^{\systemindex})^{\frac{1}{2}} \left((C_4^{\systemindex})^{C_4^{\systemindex}}\right)^{N^{\systemindex}}  \left( (N^{\systemindex})^{N^{\systemindex}} \right)^{C_4^{\systemindex}} (N^{\systemindex})^{\frac{1}{2}}}\\
			& \quad -\frac{e}{2 \pi \sqrt{2 \pi}} \frac{(N^{\systemindex})^{N^{\systemindex}} (N^{\systemindex})^{\frac{1}{2}}}{\left((C_3^{\systemindex})^{C_3^{\systemindex}}\right)^{N^{\systemindex}} \left( (N^{\systemindex})^{N^{\systemindex}} \right)^{C_3^{\systemindex}} (C_3^{\systemindex})^{\frac{1}{2}} (N^{\systemindex})^{\frac{1}{2}} \left((C_4^{\systemindex})^{C_4^{\systemindex}}\right)^{N^{\systemindex}}  \left( (N^{\systemindex})^{N^{\systemindex}} \right)^{C_4^{\systemindex}} (N^{\systemindex})^{\frac{1}{2}} }  \\
			& \qquad \frac{(C_3^{\systemindex} N^{\systemindex})^{\frac{1}{2}} C_2^{\systemindex}}{C_3^{\systemindex} N^{\systemindex} (C_2^{\systemindex})^2 + 1} \exp \left( -\frac{1}{2} C_3^{\systemindex} N^{\systemindex} (C_2^{\systemindex})^2 \right)\\
			& = \frac{e}{2 \pi} \frac{1}{((C_3^{\systemindex})^{C_3^{\systemindex}} (C_4^{\systemindex})^{C_4^{\systemindex}})^{N^{\systemindex}} (C_3^{\systemindex} C_4^{\systemindex})^{\frac{1}{2}} (N^{\systemindex})^{\frac{1}{2}}}\\
			& \quad -\frac{e}{2 \pi \sqrt{2 \pi}} \frac{1}{((C_3^{\systemindex})^{C_3^{\systemindex}} (C_4^{\systemindex})^{C_4^{\systemindex}})^{N^{\systemindex}} (C_3^{\systemindex} C_4^{\systemindex})^{\frac{1}{2}} (N^{\systemindex})^{\frac{1}{2}}} \frac{(C_3^{\systemindex} N^{\systemindex})^{\frac{1}{2}} C_2^{\systemindex}}{C_3^{\systemindex} N^{\systemindex} (C_2^{\systemindex})^2 + 1} \exp \left( -\frac{1}{2} C_3^{\systemindex} N^{\systemindex} (C_2^{\systemindex})^2 \right)\\
			& = \frac{e}{\sqrt{2 \pi}} \frac{1}{((C_3^{\systemindex})^{C_3^{\systemindex}} (C_4^{\systemindex})^{C_4^{\systemindex}})^{N^{\systemindex}} (C_3^{\systemindex} C_4^{\systemindex})^{\frac{1}{2}}} \left( \frac{1}{\sqrt{N^{\systemindex}}} - \frac{1}{\sqrt{2 \pi}} \frac{(C_3^{\systemindex})^{\frac{1}{2}}}{(C_3^{\systemindex} N^{\systemindex} C_2^{\systemindex} + \frac{1}{C_2^{\systemindex}}) \exp (\frac{1}{2} C_3^{\systemindex} N^{\systemindex} (C_2^{\systemindex})^2)} \right)\\
			& = \frac{e}{\sqrt{2 \pi}} \frac{1}{((C_3^{\systemindex})^{C_3^{\systemindex}} (C_4^{\systemindex})^{C_4^{\systemindex}})^{N^{\systemindex}} (C_3^{\systemindex} C_4^{\systemindex})^{\frac{1}{2}}} \frac{1}{\exp (\frac{1}{2} C_3^{\systemindex} N^{\systemindex} (C_2^{\systemindex})^2)} \\
			& \quad \left( \frac{\exp (\frac{1}{2} C_3^{\systemindex} N^{\systemindex} (C_2^{\systemindex})^2)}{\sqrt{N^{\systemindex}}} - \frac{1}{\sqrt{2 \pi}} \frac{(C_3^{\systemindex})^{\frac{1}{2}}}{C_3^{\systemindex} N^{\systemindex} C_2^{\systemindex} + \frac{1}{C_2^{\systemindex}}} \right)\\
			& = \frac{e}{\sqrt{2 \pi}} \frac{1}{((C_3^{\systemindex})^{C_3^{\systemindex}} (C_4^{\systemindex})^{C_4^{\systemindex}} e^{\frac{1}{2} C_3^{\systemindex} (C_2^{\systemindex})^2})^{N^{\systemindex}}} \frac{1}{(C_3^{\systemindex} C_4^{\systemindex})^{\frac{1}{2}}} \left( \frac{\exp (\frac{1}{2} C_3^{\systemindex} N^{\systemindex} (C_2^{\systemindex})^2)}{\sqrt{N^{\systemindex}}} - \frac{1}{\sqrt{2 \pi}} \frac{(C_3^{\systemindex})^{\frac{1}{2}}}{C_3^{\systemindex} N^{\systemindex} C_2^{\systemindex} + \frac{1}{C_2^{\systemindex}}} \right). \numberthis \label{probUB2}
			\end{align*}
			First notice that two items in the parentheses both converge to zero, because of the following:
			\begin{align*}
			\frac{\exp (\frac{1}{2} C_3^{\systemindex} N^{\systemindex} (C_2^{\systemindex})^2)}{\sqrt{N^{\systemindex}}} = \frac{\exp({\frac{1}{2} C_3^{\systemindex} \frac{N^{\systemindex}}{\sigma^2 {\systemindex}^{2p}}})}{\sqrt{N^{\systemindex}}}. \numberthis \label{tightness2}
			\end{align*}
			Since $p > \frac{1}{2}$, $\frac{N^{\systemindex}}{\sigma^2 {\systemindex}^{2p}} \to 0$, \eqref{tightness2} converges to zero as $\systemindex \to \infty$; similarly,
			\begin{align*}
			\frac{(C_3^{\systemindex})^{\frac{1}{2}}}{C_3^{\systemindex} N^{\systemindex} C_2^{\systemindex} + \frac{1}{C_2^{\systemindex}}} = \frac{(C_3^{\systemindex})^{\frac{1}{2}}}{C_3^{\systemindex} N^{\systemindex} (-\frac{1}{\sigma {\systemindex}^p}) -\sigma {\systemindex}^p}. \numberthis \label{tightness3}
			\end{align*}
			Again, since $p > \frac{1}{2}$, $\frac{N^{\systemindex}}{{\systemindex}^p} > 0$, $\sigma {\systemindex}^p \to +\infty$ as $\systemindex \to \infty$,  \eqref{tightness3} converges to zero as $\systemindex \to \infty$. Therefore the difference in the parentheses converges to zero as $\systemindex \to \infty$, as long as we can show the left factor in \eqref{probUB2} converges to zero as $\systemindex \to \infty$, we can claim \eqref{probUB2} converges to zero as $\systemindex \to \infty$.\par
			Obviously, if $0 < (C_3^{\systemindex})^{-C_3^{\systemindex}} (C_4^{\systemindex})^{-C_4^{\systemindex}} e^{-\frac{1}{2} C_3^{\systemindex} (C_2^{\systemindex})^2} < 1$, then the factor in \eqref{probUB2} converges to zero as $\systemindex \to \infty$. The first $<$ is trivial, As for the second $<$, 
			\begin{align*}
			& (C_3^{\systemindex})^{-C_3^{\systemindex}} (C_4^{\systemindex})^{-C_4^{\systemindex}} e^{-\frac{1}{2} C_3^{\systemindex} (C_2^{\systemindex})^2} < 1 \Leftrightarrow  \frac{1}{ (e^{\frac{1}{2} (C_2^{\systemindex})^2})^{C_3^{\systemindex}}} < (C_3^{\systemindex})^{C_3^{\systemindex}} (C_4^{\systemindex})^{C_4^{\systemindex}}\\
			& \Leftrightarrow  \frac{1}{e^{\frac{1}{2} (C_2^{\systemindex})^2}} < C_3^{\systemindex} (C_4^{\systemindex})^{\frac{C_4^{\systemindex}}{C_3^{\systemindex}}} = C_3^{\systemindex} (1 - C_3^{\systemindex})^{\frac{1 - C_3^{\systemindex}}{C_3^{\systemindex}}} \\
			& \Leftrightarrow - \frac{1}{2} (C_2^{\systemindex})^2 < \ln C_3^{\systemindex} + \frac{1 - C_3^{\systemindex}}{C_3^{\systemindex}} \ln(1 - C_3^{\systemindex})\\
			& \Leftrightarrow - \frac{(\bar{\mu} - \mu^{\systemindex})^2}{2 \sigma^2} < \ln C_3^{\systemindex} + \frac{1 - C_3^{\systemindex}}{C_3^{\systemindex}} \ln(1 - C_3^{\systemindex})\\
			& \Leftrightarrow -\frac{1}{2 \sigma^2 {\systemindex}^{2p}} <  \ln C_3^{\systemindex} + \frac{1 - C_3^{\systemindex}}{C_3^{\systemindex}} \ln(1 - C_3^{\systemindex}). \numberthis \label{probUBcondition}
			\end{align*}
			The graph of function $f(C_3^{\systemindex}) = \ln C_3^{\systemindex} + \frac{1 - C_3^{\systemindex}}{C_3^{\systemindex}} \ln(1 - C_3^{\systemindex}), C_3^{\systemindex} \in (0,1)$ is shown in Figure \ref{figure1}.
			\begin{figure}
				\centering{\includegraphics[scale=0.5]{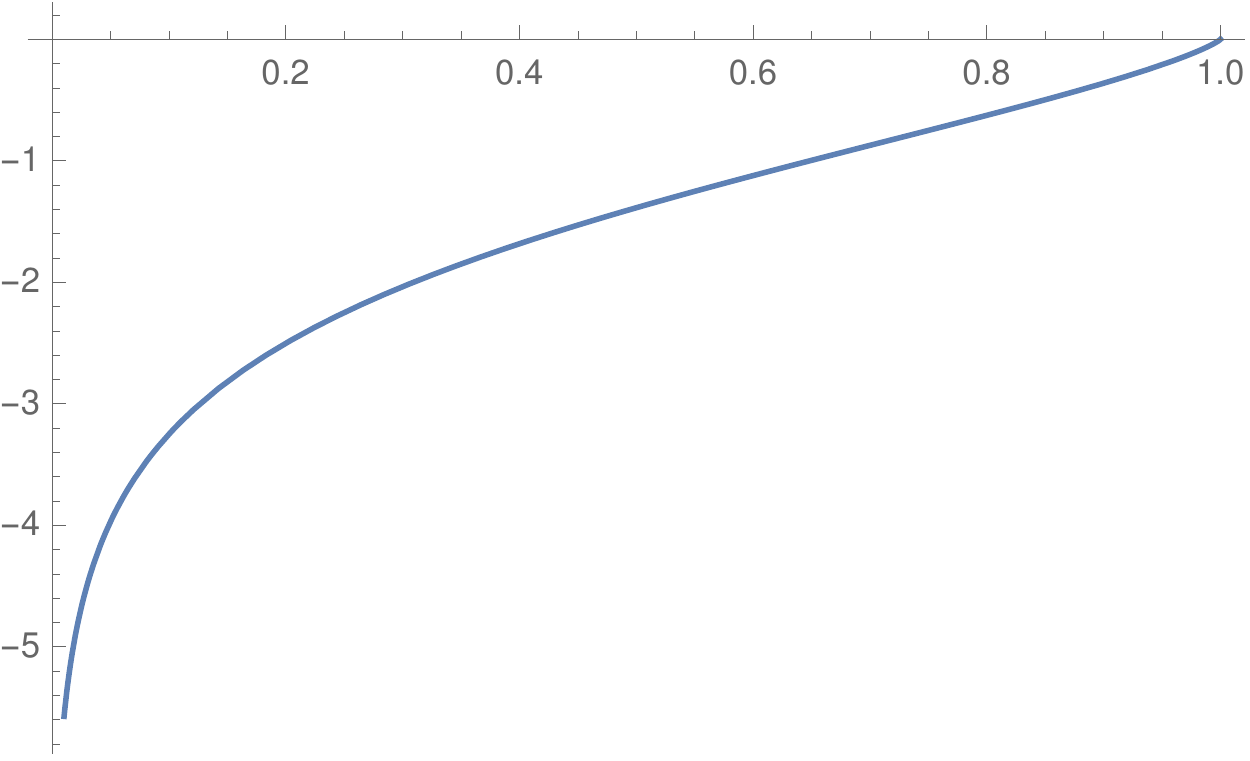}}  
				\caption{$f(C_3^{\systemindex})$ on interval $(0,1)$} \label{figure1}
			\end{figure}
			We can then choose $C_3^{\systemindex}$ as
			\begin{equation}
			C_3^{\systemindex} = \min_{0<x<1} \left\lbrace x: - \frac{1}{2 \sigma^2 {\systemindex}^{2p}} < \ln x + \frac{1 - x}{x} \ln(1 - x) \right\rbrace. \label{choiceofM}
			\end{equation}
			If both \eqref{IDrate} and \eqref{choiceofM} are satisfied, $\mathbb{P}(A_{\systemindex}) \to 0$ as $\systemindex \to \infty$ as required.
\end{proof}

Now we can prove Lemma \ref{lemmaestimatePoW}.
\begin{proof}
	Since our system is under the Halfin-Whitt heavy traffic regime, assume the staffing is determined by 
	\begin{equation}
	N^{\systemindex} =  \offeredload + \theta \sqrt{\offeredload}. \label{staffing1}
	\end{equation}
	Then from Theorem 2.1 in \cite{Atar}, the random drift is
	\begin{equation}
	\beta = \hat{\lambda} - \zeta - \bar{\mu} \nu \label{beta}.
	\end{equation}
	Here we take $\hat{\mu} = 0$ since the purpose of the term $\hat{\mu}_{\sindex}$ is just to make the random service rate more general. By substituting \eqref{initial-condition-convergence} and \eqref{staffing1} into \eqref{beta}, we have
	\begin{align*}
	\beta 
	&= \hat{\lambda} - \zeta - \lim_{\systemindex \to \infty} \bar{\mu} \diffufactor (N^{\systemindex} - \systemindex)  =  \hat{\lambda} - \zeta - \lim_{\systemindex \to \infty} \bar{\mu} \diffufactor \left(\offeredload + \theta \sqrt{\offeredload} - \systemindex \right)\\
	&=  \hat{\lambda} - \zeta - \lim_{\systemindex \to \infty} \bar{\mu} \diffufactor \frac{\lambda^{\systemindex} + \theta \sqrt{\lambda^{\systemindex} \bar{\mu}} - \systemindex \bar{\mu}}{\bar{\mu} }\\
	&=  \hat{\lambda} - \zeta - \lim_{\systemindex \to \infty} \left( \frac{\lambda^{\systemindex} - \systemindex \bar{\mu}}{\sqrt{\systemindex}} + \theta \sqrt{\bar{\mu}} \sqrt{\frac{\lambda^{\systemindex}}{\systemindex}} \right)\\
	& = \hat{\lambda} - \zeta - \hat{\lambda} - \theta \bar{\mu} = - \zeta - \theta \bar{\mu},
	\end{align*}
	by \eqref{arrival-2moment-condition} and the assumption $\lim_{\systemindex \to \infty} \frac{\lambda^{\systemindex}}{\systemindex} =  \lambda = \bar{\mu}$. Denote
	\begin{equation}
	\beta(\theta) = -\zeta - \theta \bar{\mu}. \label{beta2}
	\end{equation}
	Recall that $\zeta$ is a normal random variable with parameters $(0, \int (x- \bar{\mu})^2 d m)$, and $(\sum_{\sindex=1}^{N^{\systemindex}} \mu_{\sindex} - N^{\systemindex} \bar{\mu}) \diffufactor \Rightarrow \zeta$. For the $\systemindex$th system, let $\theta = x^{\systemindex}$ be its extra staffing which is required for stability, and 
	\begin{equation}
	\beta^{\systemindex}(x^{\systemindex}) = -\left(\sum_{\sindex=1}^{N^{\systemindex}(x^{\systemindex})} \mu_{\sindex} - N^{\systemindex}(x^{\systemindex}) \bar{\mu} \right) \diffufactor - x^{\systemindex} \bar{\mu}.\label{betan}
	\end{equation}
	Keep $x^{\systemindex}$ fixed for every $\systemindex$, then $\beta^{\systemindex}(x^{\systemindex}) \Rightarrow \beta(x^{\systemindex})$ as $\systemindex \to \infty$. When $\sum_{\sindex=1}^{N^{\systemindex}(x^{\systemindex})} \mu_{\sindex} = H$, denote $\beta^{\systemindex}(x^{\systemindex})$ as $\beta^{\systemindex}_H(x^{\systemindex})$.\par
	Notice that 
	\begin{equation}
	\pi^{\systemindex} \left( x^{\systemindex} \middle| \sum_{\sindex=1}^{N^{\systemindex}(x^{\systemindex})} \mu_{\sindex} = H \right) = \mathbb{P}\left( \mathrm{Wait}^{\systemindex} >0 \middle|  \beta^{\systemindex}(x^{\systemindex}) = - \left( H - N^{\systemindex}(x^{\systemindex}) \mu \right) \diffufactor - x^{\systemindex} \bar{\mu} \right).
	\end{equation}
	Choose $\omega_H$ such that 
	\begin{equation}
	\zeta_H = \zeta(\omega_H) = \left( H - N^{\systemindex}(x^{\systemindex}) \bar{\mu} \right) \diffufactor, \label{zetaH}
	\end{equation}
    then
	\begin{equation}
	P \left(y^{\systemindex} \right) = \mathbb{P}\left( \xi(\infty) > 0 \middle| \beta(x^{\systemindex}) = y^{\systemindex} = - \zeta_H - x^{\systemindex} \bar{\mu} \right).
	\end{equation}
	Then it remains to show that 
	\begin{equation}
	\lim_{\systemindex \to \infty} \frac{\mathbb{P}\left( \mathrm{Wait}^{\systemindex} >0 \middle|  \beta^{\systemindex}(x^{\systemindex}) = - \left( H - N^{\systemindex}(x^{\systemindex}) \bar{\mu} \right) \diffufactor - x^{\systemindex} \bar{\mu} \right)}{\mathbb{P}\left( \xi(\infty) > 0 \middle| \beta(x^{\systemindex}) = - \zeta_H - x^{\systemindex} \bar{\mu} \right)} = 1, \   \forall H > N^{\systemindex}(x^{\systemindex}) \bar{\mu} - x \bar{\mu} \sqrt{\systemindex}. \label{probofwaitingconve}
	\end{equation} 
	If we can show the family of random variables $\diffuscal{X}^{\systemindex}(\infty)$ is tight, then we can claim $\diffuscal{X}^{\systemindex}(\infty) \Rightarrow \xi(\infty)$ as $\systemindex \to \infty$ when $\beta^{\systemindex} < 0$, thus \eqref{probofwaitingconve} is satisfied.\par
	To show the tightness of $\{\diffuscal{X}^{\systemindex}(\infty)\}$, we only need to show the tightness of $\{X^{\systemindex}(\infty)\}$. (If $\{X^{\systemindex}(\infty)\}$ is tight, then $\forall \epsilon>0, \exists K^{\epsilon}>0$, such that $\mathbb{P}(X^{\systemindex}(\infty) > K^{\epsilon}) < \epsilon, \forall \systemindex$. Then $\mathbb{P}(\diffuscal{X}^{\systemindex}(\infty) > K^{\epsilon}) = \mathbb{P}((X^{\systemindex}(\infty) - \systemindex)/\sqrt{\systemindex} > K^{\epsilon}) = \mathbb{P}(X^{\systemindex}(\infty) > \sqrt{\systemindex} K^{\epsilon} + \systemindex) < \mathbb{P}(X^{\systemindex}(\infty) > K^{\epsilon}) < \epsilon, \forall \systemindex$.)\par
When $\mu^{\systemindex}$ and $M^{\systemindex}$ are defined as in Lemma \ref{homo-tightness} and \ref{convergence-P-o-A_r}, $\{X^{\systemindex}_{hom}(\infty)\}$ are tight, and $\mathbb{P}(A_{\systemindex}) \to 0$ as $\systemindex \to \infty$. Thus the right side of \eqref{tightness1} converges to zero. Let $t \to \infty$ on both sides of \eqref{tightness1}. We have that $\{X^{\systemindex}(\infty)\}$ is tight, hence $\{\diffuscal{X}^{\systemindex}(\infty)\}$ is tight.\par
		Now by the definition of tightness, for every sequence  $\{ \diffuscal{X}^{{\systemindex}_i}(\infty) \}$ in $\{ \diffuscal{X}^{\systemindex}(\infty) \}$, there exists a subsequence  $\{ \diffuscal{X}^{{\systemindex}_{i_j}}(\infty) \}$ weakly converging to a random variable $Y$. Figure \ref{figure2} shows our goal. We already know the left, top and bottom arrows are true, now we need to prove the right one. Assume $Y$ does not have the same distribution as $\xi(\infty)$ when $\beta < 0$. Take an infinite sequence $\{{\systemindex}_k\}$ and $\{t_l\}$, from the left and bottom arrows in Figure \ref{figure2}. The cumulative distribution functions (CDF) of random variables $\diffuscal{X}^{\systemindex_k}(t_l)$ will converge to the CDF of $\xi(\infty)$ as $k \to \infty$ and $l \to \infty$. Next take a subsequence of $\{\systemindex_k\}$ and $\{t_l\}$, i.e.\ $\{\systemindex_{k_h}\}$ and $\{t_{l_g}\}$, then the CDFs of $\diffuscal{X}^{\systemindex_{k_h}}(t_{l_g})$ will converge to the CDF of $Y$ as $h \to \infty$ and $g \to \infty$. If $\xi(\infty)$ and $Y$ are different random variables, then we find two subsequences of random variables' CDFs that converge to different limits, which means such a sequence does not converge, which furthermore contradicts the fact that it is convergent (see left and bottom arrows in Figure \ref{figure2}). Thus $Y$ should equal $\xi(\infty)$ in distribution, so we have proven that $\diffuscal{X}^{\systemindex}(\infty) \Rightarrow \xi(\infty)$ as $\systemindex \to \infty$. Therefore \eqref{probofwaitingconve} is proved, hence \eqref{estimatePoW} is true. The proof is completed.
	
\end{proof}

\subsubsection{Approximation of $G^r$}
Next, we find an approximation for $G^{\systemindex}(\cdot)$, which will be more straightforward.
\viva{\begin{lemma}
Denote $\beta_H(x^{\systemindex}) = -\zeta_H - x^{\systemindex} \mu \ \forall \systemindex$, where $\zeta_H$ is defined in \eqref{zetaH}, and let
\begin{equation}
\hat{G}\left(N^{\systemindex}, \lambda^{\systemindex}(x^{\systemindex}) \middle| \sum_{\sindex=1}^{N^{\systemindex}} \mu_{\sindex} = H \right) =   -\beta_H(x^{\systemindex}) \sqrt{\systemindex} \int_0^{\infty} D^{\systemindex}(s) e^{\beta_H(x^{\systemindex}) \sqrt{\systemindex} s}  ds,
\end{equation} 
then $\hat{G}\left(N^{\systemindex}, \lambda^{\systemindex}(x^{\systemindex}) \middle| \sum_{\sindex=1}^{N^{\systemindex}} \mu_{\sindex} = H \right)$ is a valid approximation of \eqref{G1}, in the sense that
\begin{equation*}
\lim_{\systemindex \to \infty} \frac{\hat{G}\left(N^{\systemindex}, \lambda^{\systemindex}(x^{\systemindex}) \middle| \sum_{\sindex=1}^{N^{\systemindex}} \mu_{\sindex} = H \right)}{G\left(N^{\systemindex}, \lambda^{\systemindex}(x^{\systemindex}) \middle| \sum_{\sindex=1}^{N^{\systemindex}} \mu_{\sindex} = H \right) } =1, \forall H > N^{\systemindex}(x^{\systemindex}) \bar{\mu} - x \bar{\mu} \sqrt{\systemindex}.
\end{equation*}
\end{lemma}
\begin{proof}
Substituting $\beta^{\systemindex}(x^{\systemindex})$ defined in \eqref{betan} to \eqref{G1}, we have
\begin{align*}
G\left(N^{\systemindex}, \lambda^{\systemindex}(x^{\systemindex}) \middle| \sum_{\sindex=1}^{N^{\systemindex}} \mu_{\sindex} = H \right) 
&= (H - \lambda^{\systemindex}) \int_0^{\infty} D^{\systemindex}(s) e^{-(H - \lambda^{\systemindex})s} ds\\
&= -\beta^{\systemindex}_H(x^{\systemindex}) \sqrt{\systemindex} \int_0^{\infty} D^{\systemindex}(s) e^{\beta^{\systemindex}_H(x^{\systemindex}) \sqrt{\systemindex} s}  ds.
\end{align*}
Since $\beta^{\systemindex}(x^{\systemindex}) \Rightarrow \beta(x^{\systemindex})$, the convergence is obtained.
\end{proof}
}

\subsubsection{Verification of approximating function}
Now apply the change of variables \eqref{betan} into the original cost function \eqref{cost2}, it can be rewritten as
\begin{equation}
C^{\systemindex}(x) = F^{\systemindex}(x) + \lambda^{\systemindex} \frac{1}{\mathbb{P}(\beta^{\systemindex}(x) < 0)} L^{\systemindex}(x), \label{cost3}
\end{equation}
where
\begin{equation}
L^{\systemindex}(x) = \int_{-\infty}^{0} \pi^{\systemindex} \left(x \middle| \beta^{\systemindex}(x) = \beta^{\systemindex}_H(x) \right) \left( -\beta^{\systemindex}_H(x) \sqrt{\systemindex} \int_0^{\infty} D^{\systemindex} (s) e^{\beta^{\systemindex}_H s \sqrt{\systemindex}} ds \right) f_{\beta^{\systemindex}(x)}(\beta^{\systemindex}_H(x)) d\beta^{\systemindex}_H(x), \label{L1}
\end{equation}
and $f_{\beta^{\systemindex}(x)}(\cdot)$ is the probability density function of $\beta^{\systemindex}(x)$.\par
Now that we have approximations for both $\pi^{\systemindex}$ and $G^{\systemindex}$, we will have an effective approximation for cost function \eqref{cost3}. More specifically, with the new approximating function
\begin{equation}
\hat{C}^{\systemindex}(x) = F^{\systemindex}(x) + \lambda^{\systemindex} \frac{1}{ \mathbb{P}( \beta(x) < 0)} \hat{L}^{\systemindex}(x), \label{cost4}
\end{equation}
where
\begin{equation}
\hat{L}^{\systemindex}(x) = \int_{-\infty}^{0} P \left(y^{\systemindex} \middle| \beta(x) = \beta_H(x) \right) \left( -\beta_H(x) \sqrt{\systemindex} \int_0^{\infty} D^{\systemindex} (s) e^{\beta_H s \sqrt{\systemindex}} ds \right) f_{\beta(x)}(\beta_H(x)) d\beta_H(x). \label{L2}
\end{equation}
and $f_{\beta(x)}(\cdot)$ is the probability density function of $\beta(x)$, we have the following theorem
\begin{theorem}
	$\hat{C}^{\systemindex}(x)$ is a valid approximation for $C^{\systemindex}(x)$ in the sense that
	\begin{equation}
	\lim_{\systemindex \to \infty} \frac{C^{\systemindex}(x)}{\hat{C}^{\systemindex}(x)} = 1, \text{ for any fixed } x. 
	\end{equation} 
	\label{cost-approximation-theorem}
\end{theorem}
\vspace{-10mm}
\begin{proof}
	Since we already have $\beta^{\systemindex}(x) \Rightarrow \beta(x)$ as $\systemindex \to \infty$ and $F^{\systemindex}(x)$ remains the same, we only need to show
	\begin{equation}
	\lim_{\systemindex \to \infty} \frac{L^{\systemindex}(x)}{\hat{L}^{\systemindex}(x)} = 1.
	\end{equation} 
	
	To make it more clear, let
	\begin{equation}
	Z(\beta^{\systemindex}(x)) = \pi^{\systemindex} \left(x \right) \left( -\beta^{\systemindex}(x) \sqrt{\systemindex} \int_0^{\infty} D^{\systemindex} (s) e^{\beta^{\systemindex} s \sqrt{\systemindex}} ds \right)
	\end{equation}
	\vspace{-5mm}
	and 
	\begin{equation}
	Z(\beta(x)) = P \left(y \right) \left( -\beta(x) \sqrt{\systemindex} \int_0^{\infty} D^{\systemindex} (s) e^{\beta s \sqrt{\systemindex}} ds \right),
	\end{equation}
	where $y = -\left( \sum_{\sindex=1}^{N^{\systemindex}(x)} \mu_{\sindex} - N^{\systemindex}(x) \mu \right) \diffufactor - x \mu$. Then $\lim_{\systemindex \to \infty}\frac{L^{\systemindex}(x)}{\hat{L}^{\systemindex}(x)} = 1$ is actually 
	\begin{equation}
	\lim_{\systemindex \to \infty}\frac{\mathbb{E}\left(Z(\beta^{\systemindex}(x)) \middle| \beta^{\systemindex} < 0\right)}{\mathbb{E}\left(Z(\beta(x)) \middle| \beta < 0\right)} = 1. \label{estimateL}
	\end{equation}
	
	Since $\beta^{\systemindex} \Rightarrow \beta$ is proved, we only need to show $Z(\beta^{\systemindex}(x))$ is uniformly integrable when $\beta^{\systemindex}(x) < 0$, and $Z(\beta^{\systemindex}(x))$ is a continuous function of $\beta^{\systemindex}(x)$. Then by the Continuous Mapping Theorem and Theorem 3.5 in \cite{Billingsley}, one can conclude \eqref{estimateL} is true.\par
	To show $Z(\beta^{\systemindex}(x))$ is uniformly integrable, we only need to show 
	$$g^{\systemindex}(\beta^{\systemindex}(x)) : = -\beta^{\systemindex}(x) \sqrt{\systemindex} \int_0^{\infty} D^{\systemindex} (s) e^{\beta^{\systemindex} s \sqrt{\systemindex}} ds$$
    is integrable when $\beta^{\systemindex}(x) < 0$, since $\pi^{\systemindex}(x)$ are the probability of waiting and thus are bounded by $1$.\par
	It is easy to see the uniform integrability of $g^{\systemindex}(\beta^{\systemindex}(x))$. $D^{\systemindex}(s)$ is assumed to be a function such that $g^{\systemindex}(\beta^{\systemindex}(x))$ is finite. Since $\beta^{\systemindex} < 0$, given the realisation of $\beta^{\systemindex}(x)$, $-\beta^{\systemindex} s \sqrt{\systemindex} \to \infty$, and $e^{\beta^{\systemindex} s \sqrt{\systemindex}} \to 0$, as $\systemindex \to \infty$. Thus, $g^{\systemindex} \to 0$ as $\systemindex \to \infty$. Hence $\exists M > 0$, such that $g^{\systemindex} < M, \forall n$. Therefore using the Dominated Convergence Theorem, 
	\begin{align*}
	\lim_{\systemindex \to \infty} \mathbb{E} \left( g^{\systemindex}(\beta^{\systemindex}(x)) \middle| \beta^{\systemindex}<0 \right) = \mathbb{E}\left(\lim_{\systemindex \to \infty} \left( g^{\systemindex}(\beta^{\systemindex}(x)) \middle| \beta^{\systemindex}<0 \right) \right) = 0,
	\end{align*}
	and the uniform integrability is proved. The continuity of $g^{\systemindex}(\beta^{\systemindex}(x))$ is obvious. We show continuity of $\pi^{\systemindex}(\beta^{\systemindex}(x))$ when $x$ is fixed, ignoring $x$ in the proof for simplicity.\par
	For every $a < 0$, we want to show, for every $\systemindex$, $\lim_{\beta^{\systemindex} \to a} \pi^{\systemindex}(\beta^{\systemindex}) = \pi^{\systemindex}(a)$. From the proof of Lemma \ref{lemmaestimatePoW} we know that for every realisation of $\beta^{\systemindex}$, there always exists a $\omega$ such that $\beta$ has the same value as $\beta^{\systemindex}$. We choose $\beta$ in such manner in the following proof. As 
	$\beta^{\systemindex}\to a (\beta \to a) $,
	\vspace{-4mm}
	\begin{align*}
	\left| \pi^{\systemindex}(\beta^{\systemindex}) - \pi^{\systemindex}(a) \right| 
	& = \left| \pi^{\systemindex}(\beta^{\systemindex}) - \mathbb{P}(\beta) + \mathbb{P}(\beta) - \mathbb{P}(a) + \mathbb{P}(a) - \pi^{\systemindex}(a) \right|\\
	& \leq \left| \pi^{\systemindex}(\beta^{\systemindex}) - \mathbb{P}(\beta) \right| + \left| \mathbb{P}(\beta) - \mathbb{P}(a) \right| + \left| \mathbb{P}(a) - \pi^{\systemindex}(a) \right|\\
	& \leq \epsilon + \epsilon + \epsilon \leq 3\epsilon.
	\end{align*}
	The first and third $\epsilon$s come from convergence of $\pi^{\systemindex}(\beta^{\systemindex})$ to $\mathbb{P}(\beta)$ when $\beta^{\systemindex}$ is given, the second $\epsilon$ is because of the continuity of $\mathbb{P}(\cdot)$. Now we proved continuity of $\pi^{\systemindex}(\beta^{\systemindex})$, hence the function $Z(\beta^{\systemindex}(x))$ is continuous with respect to $\beta^{\systemindex}(x)$.\par
	To this end we showed \eqref{estimateL} and thus \eqref{cost4} is a valid approximation of the cost function \eqref{cost3}.
\end{proof}
\label{staffing-without-abandonment}

\vspace{-4mm}
\section{Staffing many server queues with random service rates and abandonments}
Following the result we obtained in Section 3.1, we add abandonments to the model and design a cost function which we will optimise later.\par
All of the basic settings and notations are the same, except there are abandonments in the queue. Each customer has an associated patience time which are i.i.d.\ exponential random variables with rate $\nu$. A customer abandons the system without getting any service if the waiting time in the queue exceeds the customer's patience. Once her/his service starts, s/he cannot abandon the system.\par
Adapting notations from Section 4.1, let $\diffuscal{X}^r(t)$ be the scaled process of the number of customers in the $r$th system. We have the cost function
\begin{align*}
C^{\systemindex}(x):
&= F^{\systemindex}(x) + d \nu \mathbb{E}_{\beta^{\systemindex}} \left( \mathbb{E}_{\diffuscal{X}^{\systemindex}(\infty)} \left(\diffuscal{X}^{\systemindex}(\infty)^+ \middle| \diffuscal{X}^{\systemindex}(\infty) \geq 0 \right) \mathbb{P}\left(\diffuscal{X}^{\systemindex}(\infty) \geq 0 \right) \right)\\
&= F^{\systemindex}(x) + d \nu \mathbb{E}_{\beta^{\systemindex}} \left( \mathbb{E}_{\diffuscal{X}^{\systemindex}(\infty)} \left(\diffuscal{X}^{\systemindex}(\infty)^+ , \diffuscal{X}^{\systemindex}(\infty) \geq 0 \right) \right)
\numberthis \label{cost-function-abandonment1}
\end{align*}
where $d$ is the cost of every customer abandonment, and $f_{\beta^{\systemindex}}(\cdot)$ is the density function of $\beta^{\systemindex}$. Notice that in this total cost function, we ignore the holding cost in the queue because \viva{ abandonment and holding costs are both linear functions of expected queue length, and there is no need to consider expected queue length twice.}\par
As proved in Section 3.3, $\diffuscal{X}^r(t) \Rightarrow \xi(t)$ where $\xi(t)$ satisfies the equation
\begin{equation}
\xi(t) = \xi(0) + \sigma w(t) + \beta t + \gamma \int_0^t \xi(s)^- ds - \nu \int_0^t \xi(s)^+ ds. \label{limit-diffusion-abandonment}
\end{equation}  
More specifically, \eqref{limit-diffusion-abandonment} can be seen as 
\begin{equation}
\xi(t) = \twopartdef{\xi(0) + \sigma w(t) + \beta t - \nu \int_0^t \xi(s) ds,}{\xi(t) \geq 0}{\xi(0) + \sigma w(t) + \beta t - \gamma \int_0^t \xi(s) ds,}{\xi(t) < 0}.
\end{equation}
Then by Section 4 of \cite{browne}, when $\xi(t) \geq 0$, it is an Ornstein-Uhlenbeck process, thus its steady state conditional on $\xi(\infty) \geq 0$ is normal distributed with density function $f_1(x) = \frac{\frac{\sqrt{2 \nu}}{\sigma} \phi \left( \frac{\sqrt{2 \nu}}{\sigma} \left( x - \frac{\beta}{\nu}\right) \right)}{\Phi \left(\frac{\sqrt{2} \beta}{\sqrt{\nu} \sigma} \right)}$. Similarly, when $\xi(t) < 0$, it is also an O-U process, thus its steady state conditional on $\xi(\infty) < 0$ is a normal random variable with density function $f_2(t) = \frac{\frac{\sqrt{2 \gamma}}{\sigma} \phi \left( \frac{\sqrt{2 \gamma}}{\sigma} \left( x - \frac{\beta}{\gamma}\right) \right)}{\Phi \left(-\frac{\sqrt{2} \beta}{\sqrt{\gamma} \sigma} \right)}$. Let $\varrho =  \mathbb{P}(\xi(\infty) \geq 0)$. Then $\xi(\infty)$ has density function
\begin{equation}
f(x) = \twopartdef{f_1(x) \varrho, }{x \geq 0}{f_2(x) \left(1 - \varrho\right),}{x < 0}.
\end{equation}
To find out $\varrho$, notice that $f(\cdot)$ is continuous because the infinitesimal variance of $\xi(t)$ is constant on the real line. Thus by equating the limits of $f(\cdot)$ from both left and right we get 
\begin{equation}
\varrho = \mathbb{P}(\xi(\infty) \geq 0) = \left( 1 + \sqrt{\frac{\nu}{\gamma}} \frac{\phi \left( - \frac{\sqrt{2} \beta}{\sqrt{\nu} \sigma} \right)}{\phi \left( - \frac{\sqrt{2} \beta}{\sqrt{\gamma} \sigma} \right)} \frac{\Phi \left( - \frac{\sqrt{2} \beta}{\sqrt{\gamma} \sigma} \right)}{\Phi \left( \frac{\sqrt{2} \beta}{\sqrt{\nu} \sigma} \right)} \right)^{-1}.
\end{equation}

To this end, it is intuitive to use the following function as an approximation for the original cost function \eqref{cost-function-abandonment1}:
\begin{align*}
\hat{C}^{\systemindex}(x) 
&= F^{\systemindex}(x) + d \nu \mathbb{E}_{\beta} \left(  \mathbb{E}_{\xi (\infty)} \left( \xi(\infty)^{+} \middle| \xi(\infty) \geq 0 \right) \mathbb{P}\left( \xi(\infty ) \geq 0 \right) \right)\\
&= F^{\systemindex}(x) + d \nu  \mathbb{E}_{\beta} \left( \mathbb{E}_{\xi (\infty)} \left( \xi(\infty)^{+} , \xi(\infty) \geq 0 \right) \right) . \numberthis \label{cost-function-abandonment2}
\end{align*}

We have a similar theorem as \ref{cost-approximation-theorem}.
\begin{theorem}
	\eqref{cost-function-abandonment2} is a valid approximation of \eqref{cost-function-abandonment1} in the sense that
	\begin{equation}
	\lim_{\systemindex \to \infty} \frac{C^{\systemindex}(x)}{\hat{C}^{\systemindex}(x)} = 1, \text{ for any fixed } x. 
	\label{cost-approximation-abandonment-convergence}
	\end{equation}
	\label{cost-approximation-abandonment-theorem}
\end{theorem}
\vspace{-4mm}
To prove this convergence, we will need uniform integrability of the steady state. To achieve this, for each heterogeneous system, we compare it with a homogeneous system with the same number of servers and the service rate being the lower bound of $\mu_{\sindex}$, i.e.\ $p$.   

\begin{lemma}
	Let $D^{\systemindex,p}_{hom}(t)$ be the departure process of the $\systemindex$th homogeneous system described above where all the service rates take their minimum value $p$. Denote the departure process in the $\systemindex$th heterogeneous system as $D^{\systemindex}(t)$. Then $D^{\systemindex,p}_{hom}(t) \leq_{st} D^{\systemindex}(t)$, where `st' means the inequality holds stochastically.  \label{hom-het-stochastic-ordering}
\end{lemma}

\begin{proof}
	Denote the number of total customers in the $\systemindex$th homogeneous system as $X^{\systemindex,p}_{hom}(t)$. We have the system dynamic equations
	\begin{align}
	X^{\systemindex}(t) & = X^{\systemindex}(0) + A^{\systemindex}(t) - D^{\systemindex}(t) - R^{\systemindex}(t) \label{stochastic-ordering-equation4}\\
	X^{\systemindex,p}_{hom}(t) & = X^{\systemindex,p}_{hom}(0) + A^{\systemindex}(t) - D^{\systemindex,p}_{hom}(t) - R^{\systemindex}(t) \label{stochastic-ordering-equation5}
	\end{align} 
	For simplicity let $X^{\systemindex}(0) = X^{\systemindex,p}_{hom}(0)=0$. The arrival process and abandonment process are independent of the service rates and departure process, thus we can take them to be the same for these two processes. Let $S^{\systemindex,q}(t)$ be a generated poisson process with rate $N^{\systemindex} q$, where $q$ is the upper bound of $\mu_{\sindex}$, and $0<\tau_1<\tau_2<\cdots$ be the sequence of its occurrence times, i.e.\ $S^{\systemindex,q}(t) = \sum_{i=1}^{\infty} I(\tau_i < t)$. Let $\{ U_l, l \in \mathbb{N} \}$ be a sequence of independent uniform$(0,1)$ random variables, and $I(A)$ be the indicator function for event $A$ which takes the value $1$ if $A$ occurs and $0$ otherwise. Assume all of the processes equal zero at $t=0$. By splitting the process $S^{\systemindex,q}(t)$, we define the following processes
	\begin{align}
	&\breve{D}^{\systemindex,p,l}_{hom}=\sum_{i=1}^{l} I\left(U_i \leq \frac{(\breve{X}^{\systemindex,p}_{hom}(\tau_i-) \wedge N^{\systemindex})p}{N^{\systemindex}q}\right),\\	
	&\breve{D}^{\systemindex,l}=\sum_{i=1}^{l} I\left(U_i \leq \frac{\sum_{\sindex=1}^{N^{\systemindex}} \mu_{\sindex} \breve{B}_{\sindex}(\tau_i-)}{N^{\systemindex} q}\right),\\
	& S^{\systemindex,q}(t) = l = \sum_{i=1}^{\infty}I(\tau_i < t), \\
	&\breve{D}^{\systemindex,p}_{hom}(t)= \breve{D}^{\systemindex,p,l}_{hom}, \breve{D}^{\systemindex}(t)= \breve{D}^{\systemindex,l}, \forall t \in [\tau_{l}, \tau_{l+1}),\\
	&\breve{X}^{\systemindex,p}_{hom}(t) = \breve{X}^{\systemindex,p}_{hom}(0) + A^{\systemindex}(t) - \breve{D}^{\systemindex,p}_{hom}(t) - R^{\systemindex}(t),\\
	&\breve{X}^{\systemindex}(t) = \breve{X}^{\systemindex}(0) + A^{\systemindex}(t) - \breve{D}^{\systemindex}(t) - R^{\systemindex}(t),
	\end{align}
	where $\breve{B}_{\sindex}(t)$ is determined by the process $A^{\systemindex}(t)$, the sequence $\tau_1, \tau_2, \dots$, and the selection scheme defined as follows: if, for some $i \in \{1,2,\dots, S^{\systemindex,q}(t)\}$, $I\left(U_i \leq \frac{\sum_{\sindex = 1}^{N^{\systemindex}} \mu_{\sindex} \breve{B}_{\sindex}(\tau_i-)}{N^{\systemindex} q}\right)=1$, then we have that the potential departure occurring at time $\tau_i$ in process $S^{r,q}(t)$ is accepted as the real departure for process $\breve{D}^{\systemindex}(t)$. Assume there are $m$ busy servers just before this departure occurs (time $\tau_i-$). Then after it is accepted as the real departure, one of the $m$ servers will be freed, which leads to our selection scheme. Let $\eta_i$ be a uniformly distributed random variable on $(0,1)$. If $\frac{\sum_{a=0}^{j} \mu_{\sindex}}{\sum_{\sindex=1}^{m} \mu_{\sindex}} \leq \eta_i < \frac{\sum_{\sindex=0}^{j+1} \mu_{\sindex}}{\sum_{\sindex=1}^{m} \mu_{\sindex}},\ j=0,1,\dots,m-1$, then server $j+1$ will be freed at time $\tau_i$. Here we let $\mu_0=0$.\par
	Under such definition, the process $(\mu^{(i)}{D}^{\systemindex,p}_{hom}(t),\breve{X}^{\systemindex,p}_{hom}(t))$ and $(\breve{D}^{\systemindex}(t),\breve{X}^{\systemindex}(t))$ are stochastically equivalent to $(D^{\systemindex,p}_{hom}(t),X^{\systemindex,p}_{hom}(t))$ and $(D^{r}(t),X^{r}(t))$ respectively.\par
	We prove this by contradiction. Define
	\begin{equation}
	l^* = \min \{l: \breve{D}^{\systemindex,p,l}_{hom} > \breve{D}^{\systemindex,l} \}. \label{stochastic-ordering-equation3}
	\end{equation}
	Then we have 
	$$I\left(U_{l^*} \leq \frac{\sum_{\sindex=1}^{N^{\systemindex}} \mu_{\sindex} \breve{B}_{\sindex}(\tau_{l^*}-)}{N^{\systemindex} q}\right)=0 \text{ and } I\left(U_{l^*} \leq \frac{(\breve{X}^{\systemindex,p}_{hom}(\tau_{l^*}-) \wedge N^{\systemindex})p}{N^{\systemindex}q}\right) = 1 ,$$
	i.e.\
	$$\frac{\sum_{\sindex=1}^{N^{\systemindex}} \mu_{\sindex} \breve{B}_{\sindex}(\tau_{l^*}-)}{N^{\systemindex} q} < U_{l^*} \leq \frac{(\breve{X}^{\systemindex,p}_{hom}(\tau_{l^*}-) \wedge N^{\systemindex})p}{N^{\systemindex}q}.$$
	Thus
	\begin{equation}
	\sum_{\sindex=1}^{N^{\systemindex}} \mu_{\sindex} \breve{B}_{\sindex}(\tau_{l^*}-) < (\breve{X}^{\systemindex,p}_{hom}(\tau_{l^*}-) \wedge N^{\systemindex})p.
	\label{stochastic-ordering-equation1}
	\end{equation}
	Notice that $\breve{D}^{\systemindex,p,l^*}_{hom} > \breve{D}^{\systemindex,l^*}$ implies $\breve{D}^{\systemindex,p,l^*-1}_{hom} = \breve{D}^{\systemindex,l^*-1}$. Since arrival processes and abandonment processes are identical, we have $\breve{X}^{\systemindex,p}_{hom}(\tau_{l^*-1}) = \breve{X}^{\systemindex}(\tau_{l^*-1})$. Also, during the time $[\tau_{l^*-1} , \tau_{l^*})$, there are no departures, thus, by equations \eqref{stochastic-ordering-equation4} and \eqref{stochastic-ordering-equation5},
	\begin{align*}
	&\breve{X}^{\systemindex,p}_{hom}(\tau_{l^*}-) = \breve{X}^{\systemindex,p}_{hom}(\tau_{l^*-1}) + A^{\systemindex}(\tau_{l^*}-) - A^{\systemindex}(\tau_{l^*-1}) - (R^{\systemindex}(\tau_{l^*}-) - R^{\systemindex}(\tau_{l^* - 1})), \mbox{ and} \\
	&\breve{X}^{\systemindex}(\tau_{l^*}-) = \breve{X}^{\systemindex}(\tau_{l^*-1}) + A^{\systemindex}(\tau_{l^*}-) - A^{\systemindex}(\tau_{l^*-1}) - (R^{\systemindex}(\tau_{l^*}-) - R^{\systemindex}(\tau_{l^* - 1})).
	\end{align*}
	This shows us that $\breve{X}^{\systemindex,p}_{hom}(\tau_{l^*}-) = \breve{X}^{\systemindex}(\tau_{l^*}-).$ Substituting this into \eqref{stochastic-ordering-equation1} gives us 
	\begin{equation}
	\sum_{\sindex=1}^{N^{\systemindex}} \mu_{\sindex} \breve{B}_{\sindex}(\tau_{l^*}-) < (\breve{X}^{\systemindex}(\tau_{l^*}-) \wedge N^{\systemindex})p.
	\label{stochastic-ordering-equation2}
	\end{equation} 
	If $\breve{X}^{\systemindex}(\tau_{l^*}-) < N^{\systemindex}$, then $\sum_{\sindex=1}^{N^{\systemindex}} \breve{B}_{\sindex}(\tau_{l^*}-) = \breve{X}^{\systemindex}(\tau_{l^*}-)$, and \eqref{stochastic-ordering-equation2} becomes $\sum_{\sindex=1}^{N^{\systemindex}} \mu_{\sindex} \breve{B}_{\sindex}(\tau_{l^*}-) < \breve{X}^{\systemindex}(\tau_{l^*}-)p$, but since $\mu_{\sindex} \geq p$, this doesn't hold. If $\breve{X}^{\systemindex}(\tau_{l^*}-) \geq N^{\systemindex}$, then $\sum_{\sindex=1}^{N^{\systemindex}} \breve{B}_{\sindex}(\tau_{l^*}-) = N^{\systemindex}$, and \eqref{stochastic-ordering-equation2} is $\sum_{\sindex=1}^{N^{\systemindex}} \mu_{\sindex} \breve{B}_{\sindex}(\tau_{l^*}-) < N^{\systemindex}p$, which is also not true because of $\mu_{\sindex} \geq p$. \par
	Thereby we have found a contradiction to the assumption \eqref{stochastic-ordering-equation3}. Hence, we conclude that such an $n^*$ does not exist and $\breve{D}^{\systemindex,p}_{hom}(t) \leq \breve{D}^{\systemindex}(t)$ for every $t \geq 0$. That means $D^{\systemindex,p}_{hom}(t) \leq_{st} D^{\systemindex}(t)$.
\end{proof}
\vspace{-3mm}
Now we are ready to prove the theorem. 
\vspace{-3mm}
\begin{proof}[Proof of Theorem \ref{cost-approximation-abandonment-theorem}]
	Since $F^{\systemindex}$ does not change in both cost functions, we only need to show 
	\begin{align*}
	& \mathbb{E}_{\beta^{\systemindex}} \left( \mathbb{E}_{\diffuscal{X}^{\systemindex}(\infty)} \left(\diffuscal{X}^{\systemindex}(\infty)^+ , \diffuscal{X}^{\systemindex}(\infty) \geq 0 \right) \right)\\
	& \quad \to \mathbb{E}_{\beta} \left( \mathbb{E}_{\xi (\infty)} \left( \xi(\infty)^{+} , \xi(\infty) \geq 0 \right) \right), \mbox{ as } \systemindex \to \infty. \numberthis \label{cost-approximation-abandonment-equivalent-convergence}
	\end{align*}
	
	Consider a sequence of homogeneous systems with abandonment and service rates being the lower bound of $\mu_{\sindex}$, i.e.\ $p$. All other settings are the same as in the heterogeneous systems.  By Lemma \ref{hom-het-stochastic-ordering} and equations \eqref{stochastic-ordering-equation4} and \eqref{stochastic-ordering-equation5}, we know that $X^{\systemindex}(t) \leq_{st} X^{\systemindex,p}_{hom}(t)$, $\forall t \geq 0$. Since $X^{\systemindex}(t) \Rightarrow X^{\systemindex}(\infty)$ and $X^{\systemindex,p}_{hom}(t) \Rightarrow X^{\systemindex,p}_{hom}(\infty)$ as $t \to \infty$, it is also true that $X^{\systemindex}(\infty) \leq_{st} X^{\systemindex,p}_{hom}(\infty)$. 
	And by equations (3.3) and (3.4) in \cite{mandelbaum2004palm}, 
	\begin{align*}
	\mathbb{P}(X^{\systemindex,p}_{hom}(\infty) = j) = \frac{(\lambda^{\systemindex}/p)^{N^{\systemindex}}}{N^{\systemindex}!} \pi_0 \prod_{k=N^{\systemindex} + 1}^{j} \left( \frac{\lambda^{\systemindex}}{N^{\systemindex} p + (k - N^{\systemindex}) \nu} \right), j \geq N^{\systemindex} + 1,
	\end{align*}
	where
	\begin{equation}
	\pi_0 = \left( \sum_{j=0}^{N^{\systemindex}} \frac{(\lambda^{\systemindex}/p)^{j}}{j!} + \sum_{j = N^{\systemindex} + 1}^{\infty} \prod_{k=N^{\systemindex} + 1}^{j} \left( \frac{\lambda^{\systemindex}}{N^{\systemindex} p + (k - N^{\systemindex}) \nu} \right) \frac{(\lambda^{\systemindex}/p)^{N^{\systemindex}}}{N^{\systemindex}!} \right)^{-1}.
	\end{equation}
	Thus $\forall \systemindex$,
	\begin{align*}
	&\mathbb{E} \left( (X^{\systemindex,p}_{hom}(\infty) - N^{\systemindex}), X^{\systemindex}_{hom}(\infty) \geq N^{\systemindex} \right) \\
	& = \sum_{j = N^{\systemindex} + 1}^{\infty} \left(\frac{(\lambda^{\systemindex}/p)^{N^{\systemindex}}}{N^{\systemindex}!} \pi_0 \prod_{k=N^{\systemindex} + 1}^{j} \left( \frac{\lambda^{\systemindex}}{N^{\systemindex} p + (k - N^{\systemindex}) \nu} \right) \right) (j - N^{\systemindex})\\
	& = \frac{(\lambda^{\systemindex}/p)^{N^{\systemindex}}}{N^{\systemindex}!} \pi_0 \sum_{j = N^{\systemindex} + 1}^{\infty} (j - N^{\systemindex}) \prod_{k=N^{\systemindex} + 1}^{j} \left( \frac{\lambda^{\systemindex}}{N^{\systemindex} p + (k - N^{\systemindex}) \nu} \right)\\
	& = \frac{(\lambda^{\systemindex}/p)^{N^{\systemindex}}}{N^{\systemindex}!} \pi_0 \sum_{j = N^{\systemindex} + 1}^{\infty} (j - N^{\systemindex}) \frac{(\lambda^{\systemindex})^{j-N^{\systemindex}}}{(N^{\systemindex}p + \nu)(N^{\systemindex}p + 2\nu) \cdots (N^{\systemindex}p + (j-N^{\systemindex})\nu)}.
	\end{align*}
Let $m_1 = \min(p, \nu)$. Then the equation above can be enlarged
\begin{align*}
	&\mathbb{E} \left( (X^{\systemindex,p}_{hom}(\infty) - N^{\systemindex}), X^{\systemindex}_{hom}(\infty) \geq N^{\systemindex} \right) \\
	& \leq \frac{(\lambda^{\systemindex}/p)^{N^{\systemindex}}}{N^{\systemindex}!} \pi_0 \sum_{j = N^{\systemindex} + 1}^{\infty} (j - N^{\systemindex}) \frac{(\lambda^{\systemindex})^{j-N^{\systemindex}}}{m_1^{j-N^{\systemindex}} (N^{\systemindex} + 1)(N^{\systemindex} + 2) \cdots j }\\
	& = \frac{(\lambda^{\systemindex}/p)^{N^{\systemindex}}}{N^{\systemindex}!} \pi_0 \sum_{i = 1}^{\infty} \left( \frac{\lambda^{\systemindex}}{m_1} \right)^{i} \frac{i}{(N^{\systemindex} + 1)(N^{\systemindex} + 2) \cdots (N^{\systemindex} + i)}\\
	& = \frac{(\lambda^{\systemindex}/p)^{N^{\systemindex}}}{N^{\systemindex}!} \pi_0 \sum_{i = 1}^{\infty} \left( \frac{\lambda^{\systemindex}}{m_1} \right)^{i} \frac{1}{i!} \frac{i! N^{\systemindex}!}{(N^{\systemindex} + i)!} i \leq \frac{(\lambda^{\systemindex}/p)^{N^{\systemindex}}}{N^{\systemindex}!} \pi_0 \sum_{i = 1}^{\infty} \left( \frac{\lambda^{\systemindex}}{m_1} \right)^{i} \frac{1}{i!} i\\
	& = \frac{(\lambda^{\systemindex}/p)^{N^{\systemindex}}}{N^{\systemindex}!} \pi_0 \frac{\lambda^{\systemindex}}{m_1} \sum_{i = 1}^{\infty} \left( \frac{\lambda^{\systemindex}}{m_1} \right)^{i-1} \frac{1}{(i-1)!} \\
	& = \frac{(\lambda^{\systemindex}/p)^{N^{\systemindex}}}{N^{\systemindex}!} \pi_0 \frac{\lambda^{\systemindex}}{m_1} \sum_{i = 0}^{\infty} \left( \frac{\lambda^{\systemindex}}{m_1} \right)^{i} \frac{1}{i!} = \frac{(\lambda^{\systemindex}/p)^{N^{\systemindex}}}{N^{\systemindex}!} \pi_0 \frac{\lambda^{\systemindex}}{m_1} e^{-\frac{\lambda^{\systemindex}}{m_1}}. \numberthis \label{hom-expected-queue-length}
	\end{align*}
	From \cite{mandelbaum2004palm}, we know $\pi_0$ converges, thus $\eqref{hom-expected-queue-length} < \infty$, which means the expected queue length of the $\systemindex$th homogeneous system is bounded. Since $X^{\systemindex}(\infty) \leq_{st} X^{\systemindex,p}_{hom}(\infty)$, we have 
	\begin{equation}
	\mathbb{E}((X^{\systemindex}(\infty) - N^{\systemindex}), X^{\systemindex}(\infty) \geq N^{\systemindex}) \leq \mathbb{E} ((X^{\systemindex,p}_{hom}(\infty) - N^{\systemindex}), X^{\systemindex,p}_{hom}(\infty) \geq N^{\systemindex}) < \infty. 
	\end{equation}
	Scaling the inequality on both sides, we have
	\begin{equation}
	\mathbb{E}(\diffufactor (X^{\systemindex}(\infty) - N^{\systemindex}), X^{\systemindex}(\infty) \geq N^{\systemindex}) \leq \mathbb{E} (\diffufactor(X^{\systemindex,p}_{hom}(\infty) - N^{\systemindex}), X^{\systemindex,p}_{hom}(\infty) \geq N^{\systemindex}) < \infty,
	\end{equation}
	i.e.\
	\begin{equation}
	\mathbb{E}(\diffuscal{X}^{\systemindex}(\infty), \diffuscal{X}^{\systemindex}(\infty) \geq 0) \leq \mathbb{E} (\diffuscal{X}^{\systemindex,p}_{hom}(\infty), \diffuscal{X}^{\systemindex,p}_{hom}(\infty) \geq 0) < \infty.
	\label{hom-het-expected-queue-length-comparison}
	\end{equation}
	Notice that the expected value on the left is a function of the random variable $\beta^{\systemindex}$, thus it itself is also a random variable. \eqref{hom-het-expected-queue-length-comparison} implies that in the $\systemindex$th heterogeneous systems, the (scaled) expected queue length is always bounded no matter what values the service rates take, which further implies that the (scaled) queue length is uniformly integrable. \par
	Using the same reasoning process for Figure \ref{figure2}, but with abandonment in the systems, we can get $\diffuscal{X}^{\systemindex}(\infty) \Rightarrow \xi(\infty)$. Hence, by Theorem 3.5 in \cite{Billingsley}, \eqref{cost-approximation-abandonment-equivalent-convergence} is proved.
\end{proof}
To this end, we have proven the validity of the approximation \eqref{cost-function-abandonment2} of the cost function \eqref{cost-function-abandonment1}.

\section{Impact of service rate variation on abandonment rate}
In this section we want to see how the variance of the service rate influence the queue length, and thus the abandonment cost. We analyse this by considering the expected (scaled) queue length. The analytical result seems rather intractable, so instead we show their numerical results and explain how it reflects such influence. We mainly focus on systems under the LISF policy. To have a better idea of how variance plays its role in a system, we also include the numerical results for systems under the FSF policy.\par
According to \eqref{cost-function-abandonment2}, the approximating cost function for systems with abandonment is
\begin{align*}
\hat{C}^{\systemindex}(x) 
=& F^{\systemindex}(x) + d \nu \mathbb{E}_{\beta} \left(  \mathbb{E}_{\xi (\infty)} \left( \xi(\infty)^{+} \middle| \xi(\infty) \geq 0 \right) \mathbb{P}\left( \xi(\infty ) \geq 0 \right) \right)\\
=& F^{\systemindex}(x) + d \nu  \mathbb{E}_{\beta} \left( \mathbb{E}_{\xi (\infty)} \left( \xi(\infty)^{+} , \xi(\infty) \geq 0 \right) \right)\\
= & F^{\systemindex}(x) + d \nu \int_{\mathbb{R}} \int_{0}^{\infty} x f_1(x) \varrho dx d \mathbb{P}(\beta) \\
= & F^{\systemindex}(x) + d \nu \int_{-\infty}^{\infty} \int_{0}^{\infty} x \frac{\frac{\sqrt{2 \nu}}{\sigma} \phi \left( \frac{\sqrt{2 \nu}}{\sigma} \left( x - \frac{\beta}{\nu}\right) \right)}{\Phi \left(\frac{\sqrt{2} \beta}{\sqrt{\nu} \sigma} \right)} f_{\beta}(\beta) \\
& \qquad \qquad \qquad \qquad \quad \left( 1 + \sqrt{\frac{\nu}{\gamma}} \frac{\phi \left( - \frac{\sqrt{2} \beta}{\sqrt{\nu} \sigma} \right)}{\phi \left( - \frac{\sqrt{2} \beta}{\sqrt{\gamma} \sigma} \right)} \frac{\Phi \left( - \frac{\sqrt{2} \beta}{\sqrt{\gamma} \sigma} \right)}{\Phi \left( \frac{\sqrt{2} \beta}{\sqrt{\nu} \sigma} \right)} \right)^{-1} dx d\beta. \numberthis \label{cost-function-abandonment-detail-1}
\end{align*}
We want to see how the variance of the service rate influences the steady state and, thus by \eqref{cost-function-abandonment-detail-1}, the abandonment cost. \eqref{cost-function-abandonment-detail-1} contains a complicated integral which may not have a closed form. To simplify the problem, we consider a special distribution of service rates. Let the random service rates be uniformly distributed on $(\bar{\mu}-\epsilon, \bar{\mu}+\epsilon)$, $\epsilon>0$. Then $\xi$ has a random drift $\beta$, where $\beta = -\zeta-\theta \bar{\mu}$, and $\zeta \sim N(0,\frac{\epsilon^2}{3})$, and $\gamma = \bar{\mu} + \frac{\epsilon^2}{3 \bar{\mu}}$. Then, \eqref{cost-function-abandonment-detail-1} becomes
\begin{align*}
& F^{\systemindex}(x) + d \nu \int_{-\infty}^{\infty} \int_{0}^{\infty} x \frac{\frac{\sqrt{2 \nu}}{\sigma} \phi \left( \frac{\sqrt{2 \nu}}{\sigma} \left( x - \frac{\beta}{\nu}\right) \right)}{\Phi \left(\frac{\sqrt{2} \beta}{\sqrt{\nu} \sigma} \right)} \frac{\sqrt{3}}{\epsilon} \phi \left( \frac{\sqrt{3}(\beta+\theta \bar{\mu})}{\epsilon} \right) \\ 
& \qquad \qquad \qquad \qquad \quad  \left( 1 + \sqrt{\frac{\nu}{\gamma}} \frac{\phi \left( - \frac{\sqrt{2} \beta}{\sqrt{\nu} \sigma} \right)}{\phi \left( - \frac{\sqrt{2} \beta}{\sqrt{\gamma} \sigma} \right)} \frac{\Phi \left( - \frac{\sqrt{2} \beta}{\sqrt{\gamma} \sigma} \right)}{\Phi \left( \frac{\sqrt{2} \beta}{\sqrt{\nu} \sigma} \right)} \right)^{-1} dx d\beta\\
&= F^{\systemindex}(x) + d \nu \int_{-\infty}^{\infty} \int_{0}^{\infty} x \frac{\frac{\sqrt{2 \nu}}{\sigma} \phi \left( \frac{\sqrt{2 \nu}}{\sigma} \left( x - \frac{\beta}{\nu}\right) \right)}{\Phi \left(\frac{\sqrt{2} \beta}{\sqrt{\nu} \sigma} \right)} \frac{1}{\epsilon \sqrt{\frac{2\pi}{3}}} \exp\left(-\frac{3(\beta+\theta \bar{\mu})^2}{2 \epsilon^2}\right) \\ 
& \qquad \qquad \qquad \qquad \quad \left( 1 + \sqrt{\frac{\nu}{\bar{\mu} + \frac{\epsilon^2}{3 \bar{\mu}}}} \frac{\phi \left( - \frac{\sqrt{2} \beta}{\sqrt{\nu} \sigma} \right)}{\phi \left( - \frac{\sqrt{2} \beta}{\sqrt{\bar{\mu} + \frac{\epsilon^2}{3 \bar{\mu}}} \sigma} \right)} \frac{\Phi \left( - \frac{\sqrt{2} \beta}{\sqrt{\bar{\mu} + \frac{\epsilon^2}{3 \bar{\mu}}} \sigma} \right)}{\Phi \left( \frac{\sqrt{2} \beta}{\sqrt{\nu} \sigma} \right)} \right)^{-1} dx d\beta.
\numberthis \label{cost-function-abandonment-detail-2}
\end{align*}
Since the only part depending on the service rate variance is the double integral in \eqref{cost-function-abandonment-detail-2}, which is actually the expected (scaled) queue length, we let
\begin{align*}
& QL(\epsilon) = \int_{-\infty}^{\infty} \int_{0}^{\infty} x \frac{\frac{\sqrt{2 \nu}}{\sigma} \phi \left( \frac{\sqrt{2 \nu}}{\sigma} \left( x - \frac{\beta}{\nu}\right) \right)}{\Phi \left(\frac{\sqrt{2} \beta}{\sqrt{\nu} \sigma} \right)} \frac{1}{\epsilon \sqrt{\frac{2\pi}{3}}} \exp\left(-\frac{3(\beta+\theta \bar{\mu})^2}{2 \epsilon^2}\right)\\ 
& \qquad \qquad \qquad \qquad \quad \left( 1 + \sqrt{\frac{\nu}{\bar{\mu} + \frac{\epsilon^2}{3 \bar{\mu}}}} \frac{\phi \left( - \frac{\sqrt{2} \beta}{\sqrt{\nu} \sigma} \right)}{\phi \left( - \frac{\sqrt{2} \beta}{\sqrt{\bar{\mu} + \frac{\epsilon^2}{3 \bar{\mu}}} \sigma} \right)} \frac{\Phi \left( - \frac{\sqrt{2} \beta}{\sqrt{\bar{\mu} + \frac{\epsilon^2}{3 \bar{\mu}}} \sigma} \right)}{\Phi \left( \frac{\sqrt{2} \beta}{\sqrt{\nu} \sigma} \right)} \right)^{-1} dx d\beta.
\numberthis
\label{Q-length-expectation}
\end{align*}
After simplification, \eqref{Q-length-expectation} is still hard to tackle, thus we employ a numerical integral. In Figure \ref{LISFQL}, we show the function $QL(\epsilon)$ vs $\epsilon$. From the graph, $QL(\epsilon)$ is increasing. This implies that when all the other conditions remain the same, the total cost will grow as the service rate variance grows.\par
In contrast, we consider the same function for the FSF policy. The only difference in the limiting diffusion for FSF is $\gamma=\mu_{\min}$, which is $\bar{\mu}-\epsilon$ in the above-mentioned uniform distribution. Keeping other parameters unchanged, we plot its graph in Figure \ref{FSFQL}. The result is surprisingly counter-intuitive. It shows that the expected (scaled) queue length will decrease as the variance grows. One can explain such a situation as follows: the FSF policy always routes customers to the fastest available servers, thus in the long run, only the slowest server will have the chance to be idle. When the service rate variance increase, the minimum service rate will decrease, thus when the arrival rate remains the same, the `lost' capacities due to idleness will also decrease. This implies that the total service rates that are indeed utilised will increase, thus the queue length decreases. 
\begin{figure}
	\centering{\includegraphics[scale=0.6]{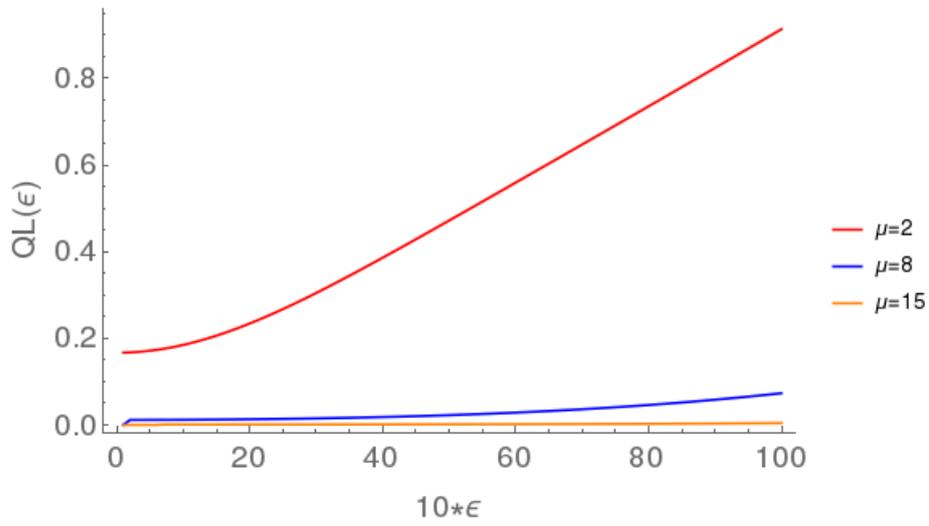}}  
	\caption{$QL(\epsilon)$ vs $\epsilon$ for the LISF policy, when $\sigma=4, \theta=2, \nu =2$.} \label{LISFQL}
\end{figure}
\begin{figure}
	\centering{\includegraphics[scale=0.6]{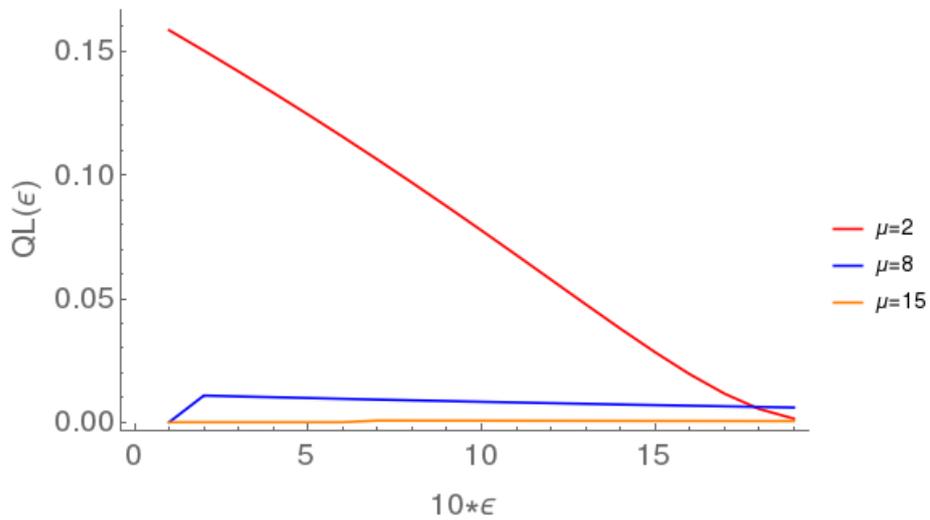}}  
	\caption{$QL(\epsilon)$ vs $\epsilon$ for the FSF policy, when $\sigma=4, \theta=2, \nu =2$.} \label{FSFQL}
\end{figure}

\section{Fairness among severs under different routing policies}
From \cite{Atar} and our analysis above, for a heterogeneous system, the longest idle server first (LISF) policy expresses a form of fairness, because when several servers are free, the one selected for the next incoming job is the one that has been idle for the longest time. On the other hand, policy faster server first (FSF) always routes customer to the fastest idle servers. Such difference in the routing scheme shows that LISF is more fair for servers than FSF, which provokes a question: can we quantify the fairness level for different routing policies? In this section we will answer this question, and demonstrate how it can be used in real systems.
\vspace{-5mm}
\subsection{Fairness measure}
We introduce a new concept called ``fairness measure''. We still consider a sequence of many server queues with i.i.d.\ servers. For simplicity, we remove the randomness on the number of servers and assume the $N$th system has exactly $N$ servers. Denote $\Upsilon$ as the support of random variables $\{\mu_{\sindex}\}$. To analyse fairness among servers, it is intuitive to consider their idle times. As before, denote $I_{\sindex}(t)$ as the idleness process for server $\sindex$, i.e.\ $I_{\sindex}(t)$ equals $1$ if server $\sindex$ is idle at $t$, and is equal to $0$ if it is busy. For a routing policy $\pi$, its fairness measure is defined as $\eta_{\pi}: \Upsilon \to [0,1]$, such that $\forall A \subset \Upsilon$, and $\forall \ T>0$, 
\begin{equation}
\sup_{0 \leq t \leq T} \frac{1}{\sqrt{N}}\left\vert \sum_{\sindex=1}^{N} \mathds{1}_{\mu_{\sindex}}(A) I_{\sindex}(t) - \eta_{\pi}(A)\sum_{\sindex=1}^{N} I_{\sindex}(t)  \right\vert \to 0 \mbox{ in probability, } \label{fairness-measure}
\end{equation}
\vspace{3mm}
as $N \to \infty$, where $\mathds{1}_{\mu_{\sindex}}(\cdot)$ is the indicator function
\begin{equation}
\mathds{1}_{\mu_{\sindex}}(A) = \twopartdef{1}{\mu_{\sindex} \in A}{0}{\mu_{\sindex} \notin A}.
\end{equation}
From the definition, it is easy to see that $\eta_{\pi}$ is a probability measure on $\Upsilon$. One can understand \eqref{fairness-measure} that the number of idle servers whose service rates are in $A$ converges u.o.c to the product of the fairness measure of set $A$ and the total number of idle servers. \par
We need to be aware that although \eqref{fairness-measure} reflects fairness of a policy to some degree, it does not hold in every situation. There should be some limitations on the policies such that \eqref{fairness-measure} is true, e.g. the policy should not depend on the total number of idle servers. Denote the set of all the eligible policies as $\Pi$. We have the following assumption,
\begin{assumption}
	For any $\pi \in \Pi$, \eqref{fairness-measure} holds. \label{assumption-fairness-measure}
\end{assumption}

\subsection{Application of fairness measure on LISF and FSF policies}
To see how Assumption \ref{assumption-fairness-measure} is used, we consider the diffusion limits proved in \cite{Atar}. In Section 2 of Chapter 2, we showed that after some manipulations on the process $\diffuscal{X}(t)$ (diffusively scaled process of total number of customers in the system), it has expression
\begin{equation}
\diffuscal{X}(t) = \diffuscal{X}(0) + W(t) + bt + F(t),
\end{equation}
with $F(t) = \frac{1}{\sqrt{N}} \int_{0}^{t} \sum_{\sindex = 1}^{N} \mu_{\sindex} I_{\sindex}(s)ds$ being the only item that is troublesome to deal with and also the only item that depends on policies.\par
Now we show the convergence of $F(t)$ under different policies. By Assumption \ref{assumption-fairness-measure}, 
\begin{align*}
F(t) &= \frac{1}{\sqrt{N}} \int_{0}^{t} \sum_{\sindex = 1}^{N} \mu_{\sindex} I_{\sindex}(s)ds \to \frac{1}{\sqrt{N}} \int_{0}^{t} \sum_{\sindex = 1}^{N} \mu_{\sindex} \mathds{1}_{\mu_{\sindex}}(\Upsilon) I_{\sindex}(s) ds\\
& \to \frac{1}{\sqrt{N}} \int_{0}^{t} \left( \int_{\Upsilon} \mu_{1}d\eta_{\pi}(\mu_1)\right) \sum_{\sindex = 1}^{N} I_{\sindex}(s) ds,   \numberthis \label{fairness-convergence1}
\end{align*} 
in probability u.o.c as $N \to \infty$.
\begin{itemize}
	\item
	LISF\\
	For any $A \in \Upsilon$, define $\eta_{LISF}(A) = \frac{\int_{A} \mu_1 dm}{\int \mu_1 dm}$, then  
	\begin{equation}
	\int_{\Upsilon} \mu_{1}d\eta_{LISF}(\mu_1) = \int_{\Upsilon} \mu_{1} \frac{\mu_1 dm}{\int \mu_1 dm} = \frac{\int \mu_1^2 dm}{\int \mu_1 dm},
	\end{equation}
	and \eqref{fairness-convergence1} becomes
	\begin{align}
	F(t) \to \frac{\int \mu_1^2 dm}{\int \mu_1 dm} \frac{1}{\sqrt{N}} \int_{0}^{t} \sum_{\sindex = 1}^{N} I_{\sindex}(s) ds, \label{fairness-convergence2}
	\end{align}
	in probability u.o.c as $N \to \infty$. $\frac{\int \mu_1^2 dm}{\int \mu_1 dm}$ is actually $\gamma$ in Theorem 2.1 in \cite{Atar}, and \eqref{fairness-convergence2} matches with the result in that paper that $F(t) \to \gamma \int_{0}^{t} \diffuscal{I}(s) ds$ in probability u.o.c as $N \to \infty$.
	\item
	FSF\\
	Similarly, for any $A \in \Upsilon$, define $\eta_{FSF}(A) = \mathds{1}_{\mu_{\min}}(A)$. Then
	\begin{equation}
	\int_{\Upsilon} \mu_{1}d\eta_{FSF}(\mu_1) = \int_{\Upsilon} \mu_{1} d\mathds{1}_{\mu_{\min}}(\mu_1) = \mu_{\min},
	\end{equation}
	and \eqref{fairness-convergence1} becomes
	\begin{equation}
	F(t) \to \mu_{\min} \frac{1}{\sqrt{N}} \int_{0}^{t} \sum_{\sindex = 1}^{N} I_{\sindex}(s) ds, \label{fairness-convergence3}
	\end{equation}
	in probability u.o.c. as $N \to \infty$.
	Such a form also matches with the result of Theorem 2.2 in \cite{Atar}.
\end{itemize}

\viva{With fairness measure, we rephrase results in \cite{Atar} in a more general way. The problem of proving diffusion limit of a particular system is reduced to finding the fairness measure of the policy that is used in the system. Thus it is possible to invent a standard method for proving diffusion limits, which is more insightful than the proof in \cite{Atar}.}

\chapter{State Space Collapse for Many Server Queues and Queueing Networks with Parameter Uncertainty}

\section{Introduction}

In this chapter, we discuss the state space collapse phenomenon. Throughout this chapter, we adapt the framework developed by \cite{Dai}. They consider a queueing network with multi-class customers and several server pools. In each pool, servers have the same capacities and capabilities. Customer arrivals are exogenous and independent from service processes. For such systems, exact analysis provides limited insight into the general properties of performances. One general way to overcome this is to use diffusion approximations. Similar to Chapter 3, the central part of the diffusion approximation is some heavy traffic limit theorems that state that a certain diffusively scaled performance processes converges to a diffusion in heavy traffic. Since the system processes in such networks are multidimensional, it becomes difficult to deal with when the system size grows large. Here is where state space collapse (SSC) plays a role. The SSC result reveals that under some conditions, the dimensions of system processes can be significantly reduced in the heavy traffic limit, while the essential information of the systems is still maintained. \cite{Dai} gain the SSC result by using what they call an SSC function. They show that under some assumptions on the networks, the SSC function evaluated at diffusively scaled processes converges to zero as the system grows large.\par
For administrative and economic reasons, servers can be categorized and allocated such that, within pools, servers have the same capabilities, i.e.\ the set of customer classes that one server in a pool is able to serve is the same as the set of customer class that any other server in the same pool can serve. This way of pooling reflects a kind of heterogeneity among servers. \cite{Dai} also assume that within each pool, servers not only have the same capabilities, but also have the same capacities, i.e.\ when two servers in the same pool serve the same class of customers, their service rates are the same. However, in reality, it is more common that some differences are present among servers who are capable of doing the same tasks. Inside each pool, even though servers have the same capabilities, their skill levels can still be different, thus it will make more sense if we consider their rates to be different and random rather than identical.\par
In our work, service rates within pools are assumed to be i.i.d.\ random variables, and we also restrict our service times to being exponentially distributed. We analyse such networks, and demonstrate that with randomness within pools the SSC result still holds. We will show that the SSC function in such systems still converges to zero in the limit.\par
This chapter is organised as follows. In Section 4.2, we introduce our model and define \textit{parallel random server systems}. In Section 4.3, we give our main result, the SSC for networks with random service rates. Then in Section 4.4, we provide essential proofs that are unique to our results. Proofs that are the same as in \cite{Dai} are relegated to Appendix \ref{Appendix-HDSSCfunction}. In Section 4.5, to show how the SSC can be used in queueing system analysis, we use the SSC method to show the diffusion limits proved in Chapter 3.

\section{Notation and model descriptions}
Our basic settings are similar to those in \cite{Dai},but slightly different. Besides the randomness among servers, we do not include abandonments in our systems, and we also do not differentiate between the arrival streams and customer classes. Every arrival stream forms one class of customers. More specifically, we consider a system with parallel server pools and several customer classes. A server pool consists of several servers whose capabilities are the same, and their capacities are i.i.d.\ random variables (see more detailed definitions below). Customers of one class arrive in the system at a certain rate. Each class of customers have their own queue. Upon their arrival they will be routed to a capable server \viva{(idle and possesses the skill to serve this class of customer)} if there is at least one; if all the capable servers are occupied, they will wait in their queues. Each customer is served by one of the servers. Once the service of a customer is completed by one of the servers, the customer leaves the system. And once a customer starts his service, he cannot abandon the service. For convenience, we refer to these systems as \textit{parallel random server systems}.
\subsection{Notation}
We need to define some notation for convenience in presentation. Throughout this chapter, unless stated otherwise, for a vector $x=(x_1, \dots, x_n) \in \mathbb{R}^n$, its norm is defined as $|x|=\max_{\{i=1,\dots,n\}}|x_i|$. For an $m \times n$ matrix $M$, its norm is $|M|= \max_{\{ i = 1,\dots,m \}}|M_i|$, where $\{M_i, i =1,\dots,m\}$ are the row vectors of $M$.\par
In Chapter 2, we defined that for any function $x(t) \in \mathbb{D}^d$ and any $T>0$, $||x(t)||_T = \sup_{0 \leq t \leq T}|x(t)|$. Now consider a sequence $x^r(t)$, we say $x^r \to x$ uniformly on a compact set (u.o.c) if $||x^r(t) - x(t)||_T \to 0$ as $r \to \infty$ for any $T>0$.
\subsection{Dynamics of the queueing networks}	
We use $I$ to denote the number of server pools, and $J$ to denote the number of customer classes. For notational convenience, we define $\mathcal{I}=\{1,\dots,I\}, \mathcal{J} = \{1,\dots,J\}$. And denote the number of servers in pool $i$ by $N_i$ for $i \in \mathcal{I}$ and set $N = (N_1, \dots, N_I)$. The total number of servers in the system is denoted by $|N|$. Class $j$ customers arrive the system according to a Poisson process with rate $\lambda_j$. We assume that the set of pools that can handle class $j$ customers is fixed and denoted by $\mathcal{I}(j)$. Similarly, the set of customer classes that pool $i$ can handle is fixed and denoted by $\mathcal{J}(i)$. \par
Upon arrival, each customer of class $j$ is routed to a server if there is an available server in one of the pools in $\mathcal{I}(j)$. Otherwise, the customer joins the queue of class $j$, waiting to be served later. Assume the service time of a class $j$ customer by the $\sindex$th server in pool $i$ is exponentially distributed with rate $\mu_{ij\sindex}$, where $j \in \mathcal{J}(i), \sindex = 1,2, \dots, N_i$. Then the service rates $\{\mu_{ij \sindex}, \sindex = 1,2,\dots,N_i\}$ are i.i.d.\ random variables. Denote their expectation value as $\bar{\mu}_{ij}=E \mu_{ij \sindex}$. We also assume that $\mu_{ij \sindex} \in [p_{ij},q_{ij}]$, $0<p_{ij}<q_{ij}$. \par
The object of study in this paper is a stochastic process $\mathbb{X} = (A, A_q, A_s, C, Q, Z, T, D)$. Assume all of the components are right continuous with left limits. We provide definitions of individual processes below.
\begin{itemize}
	\item  $A=(A_j; j \in \mathcal{J})$, $A_j(t)$ is the total number of class $j$ arrivals by time $t$.
	\item  $A_q=(A_{qj};j \in \mathcal{J})$, $A_{qj}(t)$ denotes the total number of class $j$ customers who are delayed and have to wait in the queue before their service starts.
	\item  $A_s = (A_{sij}; i \in \mathcal{I}, j \in \mathcal{J})$, $A_{sij}(t)$ is the total number of class $j$ customers who are routed to a server and start service in pool $i$ immediately after their arrival by time $t$.
	\item  $C=(C_{ij}; i \in \mathcal{I}, j \in \mathcal{J})$,  $C_{ij}(t)$ is the total number of class $j$ customers who are delayed in the queue and whose service started in pool $i$ before time $t$.
	\item  $Q=(Q_j; j \in \mathcal{J})$, $Q_j(t)$ is the total number of class $j$ customers in queue at time $t$.
	\item  $Z=(Z_{ij}; i \in \mathcal{I}, j \in \mathcal{J})$, $Z_{ij}(t)$ is the total number of servers in pool $i$ who are busy with serving class $j$ customers at time $t$.
	\item   $T=(T_{ij} ; i \in \mathcal{I}, j \in \mathcal{J})$, $T_{ij}(t)$ denotes the total time spent by servers in pool $i$ in serving class $j$ customers by time $t$.
	\item  $D= (D_{ij} ; i \in \mathcal{I}, j \in \mathcal{J})$, $D_{ij}(t)$ denotes the total number of class $j$ customers whose service are completed by servers in pool $i$ by time $t$.
\end{itemize}
Since there is heterogeneity among servers, we need to deal with each server individually, thus we also need the following notations: 
\begin{itemize}
	\item  $B_{ij \sindex}(t), i \in \mathcal{I}, j \in \mathcal{J}(i), \sindex = 1,2,\dots,N_i$: busy server indicator function. If the $\sindex$th server in pool $i$ is busy with a class $j$ customer at time $t$, $B_{ij \sindex}(t)=1$, otherwise $B_{ij \sindex}(t)=0$.
	\item  $T_{ij \sindex}(t),  i \in \mathcal{I}, j \in \mathcal{J}(i), \sindex = 1,2,\dots,N_i$: total time spent by the $\sindex$th server in pool $i$ in serving class $j$ customers by time $t$.
	\item  $D_{ij \sindex}(t), i \in \mathcal{I}, j \in \mathcal{J}(i), \sindex = 1,2,\dots,N_i$: total number of class $j$ customers whose service are completed by the $\sindex$th server in pool $i$ by time $t$.
\end{itemize}
Notice $B_{ij\sindex}, Z_{ij}, T_{ij\sindex}, T_{ij}, D_{ij}, D_{ij \sindex}$ have such relation: 
$$Z_{ij}(t) = \sum_{\sindex=1}^{N_i} B_{ij \sindex}(t), D_{ij}(t) = \sum_{\sindex=1}^{N_i} D_{ij \sindex}(t), \text{ and } T_{ij}(t) = \sum_{\sindex = 1}^{N_i} T_{ij \sindex}(t)$$
for all $i \in \mathcal{I}, j \in \mathcal{J}$.\par
The main goal of this chapter is to study the SSC results of the above-mentioned queueing networks in the diffusion limit manner. Therefore, we analyse a sequence of systems indexed by $r$ such that the arrival rate grows to infinity as $r \to \infty$. The number of servers also grows to infinity to meet the growing demand. We append $``r"$ to the processes that are associated with the $r$th system, e.g. $Q^r_j(t)$ is used to denote the number of class $j$ customers in the queue in the $r$th system at time $t$. The arrival rate in the $r$th system is given by $\lambda^r=(\lambda^r_j, j \in \mathcal{J})$, and we assume that 
\begin{equation}
\lambda^r_j \to \infty,
\end{equation}
as $r \to \infty$.\par
Let $\{S_{ij \sindex}, i \in \mathcal{I}, j \in \mathcal{J}, \sindex = 1,2,\dots, N_i\}$ be i.i.d.\ standard Poisson processes, each having right-continuous sample paths. Combining the settings in \cite{Atar} and \cite{Dai}, the processes $D_{ij\sindex}$ are assumed to satisfy
\begin{equation}
D^r_{ij\sindex}(t) = S_{ij \sindex}(\mu_{ij \sindex} T^r_{ij \sindex}(t)),\ \ \ i \in \mathcal{I}, j \in \mathcal{J}, \sindex = 1,2, \dots, N_i, \label{depadef}
\end{equation}
where
\begin{equation}
T^r_{ij \sindex}(t) = \int_0^t B^r_{ij \sindex}(s) ds,\ \ \  i \in \mathcal{I}, j \in \mathcal{J}, \sindex = 1,2, \dots, N_i. \label{timeshift}
\end{equation}
The process $\mathbb{X}^r$ depends on the control policy used in the system. To emphasize the dependence on the control policy $\pi$ used, we use $\mathbb{X}_{\pi}$ to denote the process. Clearly, each element of $A^r, A^r_q, A^r_s, C^r, T^r, D^r$ is a nondecreasing process, and each element of $Q^r$ and $Z^r$ is nonnegative. Furthermore, the process $\mathbb{X}^r_{\pi}$ satisfies the following dynamic equations for all $t \geq 0$.
\begin{equation}
A^r_j(t) = A^r_{qj}(t) + \sum_{i \in \mathcal{I}(j)}A^r_{sij}(t), \mbox{ for all } j \in \mathcal{J}, \label{originprocess1}
\end{equation}
\begin{equation}
Q^r_j(t) = Q^r_j(0) + A^r_{qj}(t) - \sum_{i \in \mathcal{I}(j)} C^r_{ij}(t), \mbox{ for all } j \in \mathcal{J}, \label{originprocess2}
\end{equation}
\begin{equation}
Z^r_{ij}(t) = Z^r_{ij}(0) + A^r_{sij}(t) + C^r_{ij}(t) -\sum_{\sindex=1}^{N^r_i} D^r_{ij\sindex}(t), \mbox{ for all } i \in \mathcal{I}, j \in \mathcal{J}, \label{originprocess3}
\end{equation}
\begin{equation}
\sum_{j \in \mathcal{J}(i)} Z^r_{ij}(t) \leq N^r_i, \mbox{ for all } i \in \mathcal{I}, \label{originprocess4}
\end{equation}
\begin{equation}
Q^r_j(t) \left( \sum_{i \in \mathcal{I}(j)} \left( N^r_i - \sum_{j^{\prime} \in \mathcal{J}(i)} Z^r_{ij^{\prime}}(t) \right) \right) = 0, \mbox{ for all } j \in \mathcal{J}, \label{originprocess5}
\end{equation}
\begin{equation}
\int_0^t \sum_{i \in \mathcal{I}(j)} \left( N^r_i - \sum_{j^{\prime} \in \mathcal{J}(i)} Z^r_{ij^{\prime}}(s-)\right) dA^r_{qj}(s) = 0, \mbox{ for all } j \in \mathcal{J}, \label{originprocess6}
\end{equation}
\begin{equation}
\text{Equations} \text{ associated with the control policy } \pi. \label{originprocess7}
\end{equation}
Equations \eqref{originprocess5} and \eqref{originprocess6} are based on the assumed non-idling property of a control policy. Equation \eqref{originprocess5} implies that there can be customers in the queue only when all of the servers that can serve that class of customers are busy. Equation \eqref{originprocess6} implies that an arriving customer is delayed in the queue only if there is no idle server that can serve that customer at the time of his arrival. Equation \eqref{originprocess7} indicates the scheduling decisions to be made according to the selected scheduling policies.
\vspace{-2mm} \par
Let $\{S_{ij \sindex}, i \in \mathcal{I}, j \in \mathcal{J} \}$ be the Poisson processes defined before, and $\{v_{ij \sindex}(l); l=1,2,\dots\}$ be the corresponding sequence of i.i.d.\ exponential random variables. Since $S_{ij \sindex}$ is a Poisson process, $v_{ij \sindex}(l)$ has exponential distribution with rate $\mu_{ij \sindex}$. We define $V_{ij \sindex}: \mathbb{N} \to \mathbb{R}$ by
\begin{equation}
V_{ij\sindex}(m) = \sum_{l=1}^m \frac{v_{ij\sindex}(l)}{\mu_{ij\sindex}} ,\ \ m \in \mathbb{N},
\end{equation}
where, by convention, empty sums are set to be zero. The term $V_{ij\sindex}(m)$ is the total service requirement of the first $m$ class $j$ customers who are served by the $\sindex$th server in pool $i$, and $V_{ij \sindex}$ is known as the cumulative service time process. By the duality of $S_{ij \sindex}$ and $V_{ij \sindex}$, we have
\begin{equation}
S_{ij \sindex}(\mu_{ij \sindex} t) = \max \left\lbrace m: V_{ij \sindex}(m) \leq t \right\rbrace, \ \ t \geq 0.
\end{equation}
It follows from \eqref{depadef} that
\begin{equation}
V_{ij \sindex}(D^r_{ij \sindex}(t)) \leq T^r_{ij \sindex}(t) \leq V_{ij \sindex}(D^r_{ij \sindex}(t) + 1). \label{srvtinequa}
\end{equation}

Next, we give the details of the arrival processes. Let $\chi(t)$ be a delayed renewal process with rate $1$. Let
\begin{equation}
A^r(t) = \chi(\lambda^r t)
\end{equation}
Let $\{ u(l): l=1,2,\dots \}$ be the sequence of interarrival times that are associated with the process $\chi$. Note that they are independent and identically distributed. We define $U: \mathbb{N} \to \mathbb{R}$ by
\begin{equation}
U(m) = \sum_{l=1}^r u(l),\ \ m \in \mathbb{N},
\end{equation}
and so
\begin{equation}
\chi(t) = \max\{ m: U(m) \leq t \}.
\end{equation}
We require that the interarrival times of the arrival processes satisfy the following condition, which is similar to condition (3.4) in \cite{Bramson}:
\begin{equation}
\mathbb{E}\left(u(2)^{2+\epsilon}\right) < \infty,\ \ \text{for some } \epsilon > 0.\label{momentcon}
\end{equation}
Condition \eqref{momentcon} is automatically satisfied by the service times because they are assumed to be exponentially distributed. For the rest of the paper, we assume that the primitive processes of the system satisfy \eqref{momentcon}. We also assume that $Q^r(0),Z^r(0),\chi$ and $S$ are independent.

We require that the number of servers in the $r$th system is selected so that
\begin{align}
&\lim_{r \to \infty} \frac{|N^r|}{r} = 1,  \label{No.servers} \\ 
&\lim_{r \to \infty} \frac{N^r_i}{|N^r|} = \beta_i, \mbox{ for all } i \in \mathcal{I} \mbox{ and for some } \beta_i \in (0,1), \label{server-pool-percertage}\\
&\lim_{r \to \infty} \frac{\lambda^r_j}{|N^r|} = \lambda_j, \ \ \text{for some } 0<\lambda_j<\infty. \label{arriallimits}
\end{align}
We shall denote $\lambda=(\lambda_j, j \in \mathcal{J})$.

\section{Main results}

In this section, we state our main result as Theorem \ref{SSC1}, \viva{which is an SSC result for \textit{parallel random server systems}. We extend the SSC result for systems from identical servers in each pool which is proved by \cite{Dai}, to include random service rates in each pool. The proof framework of our results is similar to the one in \cite{Dai}. Our main contribution is when showing almost Lipschitz condition for hydrodynamically scaled departure processes, the direct way in \cite{Dai} is no longer valid, so instead we come up with a new coupling method. For more detailed differences between proofs of \cite{Dai} and our results, see discussion at the beginning of Section 4.4.} \par
Before stating the theorem, we need some preliminary definitions and assumptions. \par
\viva{We only consider systems under the heavy traffic condition. For a multi-class network with several server pools, it is not trivial to define the heavy traffic condition. Instead, we use the static planning problem (SPP) for the networks as a bridge to heavy traffic conditions.}

\subsection{The static planning problem}
\viva{The SPP is introduced in \cite{Dai}. We will modify the original problem such that it suits our models. The objective of an SPP is to minimise server utilisations in the network. \cite{Dai} use identical service rates to define utilisations, while we consider their expected value $\bar{\mu}_{ij}$ instead.}\par
Let $x=(x_{ij}, i \in \mathcal{I}, j \in \mathcal{J}(i))$, where $x_{ij}$ is the long term proportion of pool $i$ servers' working time in serving class $j$ customers. We define the static planning problem 
\begin{equation}
\begin{aligned}
\min & \quad \rho \\
\mbox{s.t.} & \sum_{i \in \mathcal{I}(j)} \beta_i \bar{\mu}_{ij} x_{ij} = \lambda_j, \mbox{ for all } j \in \mathcal{J},\\
& \sum_{j \in \mathcal{J}(i)}x_{ij} \leq \rho, \mbox{ for all } i \in \mathcal{I},\\
& x_{ij} \geq 0 \mbox{ for all } i \in \mathcal{I}, j \in \mathcal{J}.
\end{aligned} \label{static-planning-optimisation1}
\end{equation}
Denote $(\rho^*, x^*)$ as the optimal solution of the above static planning problem. This means the average utilisation of the busiest pool will reach its minimum value $\rho^*$ with the allocation $x^*$. If $\rho^* > 1$, it can easily be shown that the queue length grows without bound; thus in our analysis it is always assumed that $\rho^* \leq 1$.\par
Since we need a heavy traffic condition, we now consider a sequence of the following optimisation problem
\begin{equation}
\begin{aligned}
\min & \quad \rho^r\\
\mbox{s.t.} & \sum_{i \in \mathcal{I}(j)} N^r_i \bar{\mu}_{ij} x^r_{ij} = \lambda^r_j, \mbox{ for all } j \in \mathcal{J},\\
& \sum_{j \in \mathcal{J}(i)} x^r_{ij} \leq \rho^r, \mbox{ for all } i \in \mathcal{I},\\
& x^r_{ij} \geq 0 \mbox{ for all } i \in \mathcal{I}, j \in \mathcal{J}.
\end{aligned} \label{static-planning-optimisation2}
\end{equation}
Let $(\rho^r, x^{r,*})$ be an optimal solution of \eqref{static-planning-optimisation2}. These optimisation problems will be an important part in the assumptions of heavy traffic below.
\vspace{-2mm}
\subsection{Assumptions}
\viva{Before starting the analysis of the state space collapse phenomenon, we need to clarify that such phenomenon will not happen in all networks. We need some constraints on the systems to guarantee SSC. We have the following two assumptions for that reason.}\par
Using the definitions in the static planning problem, we have the assumption about heavy traffic conditions and control policies.
\vspace{-2mm}
\begin{assumption}
	For each static optimal solution $(\rho^*, x^*)$ of the SPP \eqref{static-planning-optimisation1}, we have $\rho^*=1$ and $\sum_{j \in \mathcal{J}(i)} x^*_{ij} = 1$. Moreover, for any sequence of optimal solutions $\{ x^{r,*}\}$ of \eqref{static-planning-optimisation2}, we have
	\begin{equation*}
	x^{r,*} \to x^*,
	\end{equation*}
	as $r \to \infty$ for some optimal solution of \eqref{static-planning-optimisation1}.
	\label{heavytraffic}
\end{assumption}

In this work, we only consider control policies that will not cause the system to explode, i.e.\ the queue length does not grow to infinity. Thus we need the following assumption.

\begin{assumption}
	For a control policy $\pi$,
	\begin{equation}
	\frac{Z^r_{ij}(t)}{|N^r|} \to z \ \ \ \text{u.o.c.  a.s.}
	\end{equation}
	as $r \to \infty$ if $\ Z^r_{ij}(0)/|N^r| \to (0,z)$ a.s. as $r \to \infty$, where $z=(z_{ij}, i \in \mathcal{I}, j \in \mathcal{J})$, and $z_{ij} = \beta_i x^*_{ij}$ for an optimal solution $(\rho^*,x^*)$ of the static planning problem \eqref{static-planning-optimisation1}. \label{steadycondition}
\end{assumption}

We do not include the constraints on $Q^r(\cdot)$ because, unlike in \cite{Dai}, we do not have abandonments in our system. 

\subsection{Fluid limit and verification of the control policy assumption}
Under a control policy, when Assumption \ref{steadycondition} is satisfied, the fluid limits exist and do not explode, even though they are critically loaded. We assume that
\begin{equation}
Z^r(0)/|N^r| \to z \text{  a.s.} \label{diffusioninitialcondition}
\end{equation}
as $r \to \infty$, where $z$ is given as in Assumption \ref{steadycondition}. Under Assumption \ref{steadycondition}, condition \eqref{diffusioninitialcondition} implies that 
\begin{equation*}
Z^r(\cdot)/|N^r| \to z \text{  u.o.c.  a.s.}
\end{equation*}
as $r \to \infty$, for $t \geq 0$.\par
To make Assumption \ref{steadycondition} easier to check, we use the fluid limit concept. The fluid scaling is defined as
\begin{equation}
\fluidscal{\mathbb{X}}^r(t) = \frac{\mathbb{X}^r(t)}{|N^r|}.
\end{equation}
The following definitions and notations are from \cite{Dai}, but we repeat them here for completeness. \par
Let $\mathscr{A} \in \Omega$ be such that $\{ \fluidscal{Q}^{\systemindex}(0) \}$ is bounded and the following Functional Strong Law of Large Numbers holds:
\begin{align*}
\frac{A^{\systemindex}_j(|N^{\systemindex}| \cdot)}{|N^{\systemindex}|} \to a_j(\cdot), \frac{\sum_{\sindex=1}^{N^{\systemindex}_i}S_{ij\sindex}(|N^{\systemindex}|\cdot)}{|N^{\systemindex}|} \to \alpha_{ij}(\cdot), \mbox{ u.o.c }  \numberthis \label{FSLLN}
\end{align*}
as $\systemindex \to \infty$, where $\alpha_{ij}(t) = \bar{\mu}_{ij}t, a_j(t)=t$. Note that we can take $\mathbb{P}(\mathscr{A})=1$.\par
We call $\fluidscal{\mathbb{X}}^r(t)$ the fluid scaled process. $\fluidscal{\mathbb{X}}$ is called a fluid limit of $\{\mathbb{X}^r\}$ if there exists an $\omega \in \mathscr{A}$ and a sequence $\{r_l\}$ with $r_l \to \infty$ as $l \to \infty$, such that $\fluidscal{\mathbb{X}}^{r_l}(\cdot, \omega)$ converges u.o.c. to $\fluidscal{\mathbb{X}}$ as $l \to \infty$, where $\mathscr{A}$ is taken from Theorem B.1 in \cite{Dai}. The following theorem is analogous to Theorem B.1 in \cite{Dai}, with the main difference being in $\fluidscal{T}$ and $\fluidscal{I}$. These two processes are considered individually rather than aggregately due to the servers' heterogeneity.
\begin{theorem}
	Let $\{\mathbb{X}^r_{\pi}\}$ be a sequence of $\pi$-parallel random server systems processes. Assume that \eqref{No.servers} and \eqref{arriallimits} hold and $\{\fluidscal{Q}^r(0)\}$ is bounded a.s. as $r \to \infty$. Then $\{\fluidscal{\mathbb{X}}^r_{\pi}\}$ is a.s.  precompact(i.e.\, every sequence has a convergent subsequence) in the Skorohod space $\mathbb{D}^d[0, \infty)$ endowed with the u.o.c. topology. Thus, the fluid limits exist, and each fluid limit, $\fluidscal{\mathbb{X}}_{\pi}$, of $\{\fluidscal{\mathbb{X}}^r_{\pi}\}$ satisfies the following equations for all $t \geq 0$:
	\begin{equation}
	\lambda_j t = \fluidscal{A}_{qj}(t) + \sum_{i \in \mathcal{I}(j)}\fluidscal{A}_{sij}(t), \mbox{ for all } j \in \mathcal{J}, \label{fluidequ1}
	\end{equation}
	\begin{equation}
	\fluidscal{Q}_j(t) = \fluidscal{Q}_j(0) + \fluidscal{A}_{qj}(t) - \sum_{i \in \mathcal{I}(j)}\fluidscal{C}_{ij}(t), \mbox{ for all } j \in \mathcal{J}, \label{fluidequ2}
	\end{equation}
	\begin{equation}
	\fluidscal{Z}_{ij}(t) = \fluidscal{Z}_{ij}(0) + \fluidscal{A}_{sij}(t) + \fluidscal{C}_{ij}(t) - \bar{\mu}_{ij} \fluidscal{T}_{ij}(t), \mbox{ for all } i \in \mathcal{I}, \mbox{and } j \in \mathcal{I}(i),\label{fluidequ3}
	\end{equation}
	\begin{equation}
	\fluidscal{T}_{ij}(t) = \int_0^t \fluidscal{Z}_{ij}(s) ds, \mbox{ for all } i \in \mathcal{I}, \mbox{and } j \in \mathcal{I}(i),\label{fluidequ4}
	\end{equation}
	\begin{equation}
	\fluidscal{I}_{i}(t) = \beta_i t - \sum_{j \in \mathcal{J}(i)}\fluidscal{T}_{ij}(t),  \mbox{ for all } i \in \mathcal{I}, \label{fluidequ6}
	\end{equation}
	\begin{equation}
	\fluidscal{Q}_j(t) \left( \beta_i - \sum_{j^{\prime} \in \mathcal{J}(i)}\fluidscal{Z}_{ij^{\prime}}(t) \right) = 0, \mbox{ for all } j \in \mathcal{J}, \label{fluidequ7}
	\end{equation}
	\begin{equation}
	\int_0^t \sum_{j \in \mathcal{J}(i)}\fluidscal{Q}_{j}(s) d\fluidscal{I}_i(s) = 0, \mbox{ for all } i \in \mathcal{I}, \label{fluidequ8}
	\end{equation}
	\begin{equation}
	\int_0^t \sum_{i \in \mathcal{I}(j)}\left( \beta_i - \sum_{j^{\prime} \in \mathcal{J}(i)}\fluidscal{Z}_{ij^{\prime}}(s) \right) d\fluidscal{A}_{qj}(s) = 0, \mbox{ for all } j \in \mathcal{J}, \label{fluidequ9}
	\end{equation}
	\begin{equation}
	\fluidscal{A}, \fluidscal{A}_q, \fluidscal{A}_s, \fluidscal{T}, \text{ and } \fluidscal{C} \text{ are nondecreasing}, \label{fluidequ10}
	\end{equation}
	\begin{equation}
	\fluidscal{Q}(t) \geq 0,\ \fluidscal{Z}_{ij}(t) \geq 0, \mbox{ and } \sum_{j \in \mathcal{J}(i)}\fluidscal{Z}_{ij}(t) \leq 1 \mbox{ for all } i \in \mathcal{I}, \mbox{and } j \in \mathcal{J}(i).\label{fluidequ11}
	\end{equation}
	\label{fluidlimit}
\end{theorem}
The proof of this theorem is in Appendix \ref{Appendix-fluidlimit}.

The vector $(q,z)$ is called a steady state of the fluid limits if for any fluid limit $\fluidscal{\mathbb{X}}$, $\fluidscal{Q}(0) = q$ and $\fluidscal{Z}(0) = z$ implies $\fluidscal{Q}(t) = q$ and $\fluidscal{Z}(t) = z$ for all $t \geq 0$.\par
We denote the set of all of the steady states of the fluid limits of $\{\mathbb{X}^r\}$ by $\mathscr{M}$. The following result is analogous to Lemma B.1 in \cite{Dai}, and it is an equivalent condition to Assumption \ref{steadycondition}.
\begin{lemma}
	Let $\{\mathbb{X}^r_{\pi}\}$ be a sequence of $\pi$-parallel random server systems processes that satisfies conditions of Theorem \ref{fluidlimit} and Assumption \ref{heavytraffic}. A control policy $\pi$ satisfies Assumption \ref{steadycondition} if $(0,z) \in \mathscr{M}$, where $z_{ij}=\beta_i x^*_{ij}$.
\end{lemma}
The proof is the same as Lemma B.1 in \cite{Dai} and is trivial so we omit it here.

In general, diffusion limits are introduced as refinements of the fluid limits. Under condition \eqref{diffusioninitialcondition} and Assumption \ref{steadycondition}, we define the diffusive scaling as follows:

\begin{equation}
\diffuscal{Q}^r(t) = \frac{Q^r(t)}{\sqrt{|N^r|}} \ \ \text{  and  }\ \  \diffuscal{B}_{ij \sindex}^r(t) = \frac{B_{ij \sindex}^r(t)}{\sqrt{|N^r|}}, \ \ \text{ for }\ \  t \geq 0. \label{diffusionscaling1}
\end{equation}
and denote
\begin{equation}
\diffuscal{Z}^r_{ij}(t)  = \sum_{\sindex=1}^{N^r_i} \left( \diffuscal{B}_{ij \sindex}^r(t) - \frac{x^*_{ij}}{\sqrt{N^r}} \right) = \frac{Z^r_{ij}(t) - x^*_{ij}N^r_i}{\sqrt{|N^r|}}. \label{diffusionscaling2}
\end{equation}

Now we introduce the hydrodynamic model equations, and we borrow the SSC function from \cite{Dai}. This function is the key point in our main theorem.
\subsection{Hydrodynamic model equations}
Consider the process $\tilde{\mathbb{X}}_{\pi} = (\tilde{A}, \tilde{A}_q, \tilde{A}_s, \tilde{Q}, \tilde{B}, \tilde{Z}, \tilde{C})$ and the following set of equations:
\begin{equation}
\lambda_j t = \tilde{A}_{qj}(t) + \sum_{i \in \mathcal{I}(j)}\tilde{A}_{sij}(t), \mbox{ for all } j \in \mathcal{J}, \label{HDfirst}
\end{equation}
\begin{equation}
\tilde{Q}_j(t) = \tilde{Q}_j(0) + \tilde{A}_{qj}(t) - \sum_{i \in \mathcal{I}(j)}\tilde{C}_{ij}(t), \mbox{ for all } j \in \mathcal{J}, \label{HDmodel2}
\end{equation}
\begin{equation}
\tilde{A}_{qj}, \tilde{A}_{sij}, \tilde{C}_{ij} \text{ are nondecreasing for all } i \in \mathcal{I}, j \in \mathcal{J},\label{HDmodel3}
\end{equation}
\begin{equation}
\tilde{Z}_{ij}(t) = \tilde{Z}_{ij}(0) + \tilde{A}_{sij}(t) + \tilde{C}_{ij}(t) - \bar{\mu}_{ij}\tilde{T}_{ij}(t), \mbox{ for all } i \in \mathcal{I}, \mbox{ and } j \in \mathcal{J}(i), \label{HDmodel4}
\end{equation}
\begin{equation}
\tilde{T}_{ij}(t) = \int_0^t z_{ij} ds = z_{ij}t, \mbox{ for all } i \in \mathcal{I}, \mbox{ and } j \in \mathcal{J}(i), \label{HDmodel8}
\end{equation}
\begin{equation}
\tilde{Q}_j(t) \geq 0, \mbox{ for all } j \in \mathcal{J}, \text{ and } \sum_{j \in \mathcal{J}(i)} \tilde{Z}_{ij}(t) \leq 0, \mbox{ for all } i \in \mathcal{I}, \label{HDmodel5}
\end{equation}
\begin{equation}
\tilde{Q}_j(t) \left( \sum_{i \in \mathcal{I}(j)} \sum_{j^{\prime} \in \mathcal{J}(i)}\tilde{Z}_{ij^{\prime}}(t) \right) = 0, \mbox{ for all } j \in \mathcal{J}, \label{HDmodel6}
\end{equation}
\begin{equation}
\int_0^t \left( \sum_{i \in \mathcal{I}(j)} \sum_{j^{\prime} \in \mathcal{J}(i)} \tilde{Z}_{ij^{\prime}}(s)\right) d\tilde{A}_{qj}(s) = 0, \mbox{ for all } j \in \mathcal{J}, \label{HDmodel7}
\end{equation}
\begin{equation}
\text{Additional equations associated with the control policy } \pi, \label{HDlast}
\end{equation}
where $\lambda_j$ is defined as in \eqref{arriallimits}. Equations \eqref{HDfirst}-\eqref{HDlast} are called the \textit{hydrodynamic model equations}, and they define the \textit{hydrodynamic model} of the system under policy $\pi$. Any process $\tilde{\mathbb{X}}_{\pi}$ satisfying \eqref{HDfirst}-\eqref{HDlast} for all $t \geq 0$ is called a hydrodynamic model solution.

Hydrodynamic model solutions are deterministic and absolutely continuous, hence almost everywhere differentiable. Absolute continuity follows from the following result.

\begin{proposition}
	Any process $\tilde{\mathbb{X}}_{\pi}$ satisfying \eqref{HDfirst}-\eqref{HDlast} for all $t \geq 0$ is Lipschitz continuous.
\end{proposition}

It will be proved in Proposition \ref{mainbridge} that the hydrodynamic model equations are satisfied by hydrodynamic limits under certain general assumptions; these limits are obtained from the hydrodynamically scaled sequences.

\subsection{SSC in the diffusion limits}
Similar to Section 4.2 in \cite{Dai}, we define the state space collapse function. Let $g: \mathbb{R}^{J+d_z} \to \mathbb{R}^+$, where $d_z= \sum_{j \in \mathcal{J}} |\mathcal{I}(j)|$, be a nonnegative function that satisfies the following \textit{homogeneity} condition:
\begin{equation}
g(\alpha x) =  \alpha^c g(x),\label{homogeneity}
\end{equation}
for some $c>0$, for all $x \in \mathbb{R}^{J+d_z}$, and for all $0 \leq \alpha \leq 1$. We call $g$ a SSC-function. Nonnegativity assumption is made for notational convenience, and one can always consider $|g|$ in order to have a nonnegative function if $g$ can take negative values. We make the following assumption about the SSC function.

\begin{assumption}
	The function $g: \mathbb{R}^{J+d_z} \to \mathbb{R}^+$ satisfies \eqref{homogeneity} and is continuous on $\mathbb{R}^{J+d_z}$. \label{continuity}
\end{assumption}

As the machinery to state an SSC result has been set, we are ready to state the conditions on the hydrodynamic model solutions that imply that an SSC result holds in the diffusion limit. The following assumption is analogous to \cite[Assumption 4.2]{Dai}.

\begin{assumption}
	Let $g$ be a function that satisfies Assumption \ref{continuity}. There exists a function $H(t)$ with $H(t) \to 0$ as $t \to \infty$ such that
	\begin{equation}
	g(\tilde{Q}(t),\tilde{Z}(t)) \leq H(t) \text{ for all } t \geq 0 \label{HDSSCconditionequ}
	\end{equation}
	for each hydrodynamic model solution $\tilde{\mathbb{X}}_{\pi}$ satisfying $|(\tilde{Q}(0), \tilde{Z}(0))| \leq 1$. Furthermore, for each hydrodynamic model solution $\tilde{\mathbb{X}}_{\pi}$ with $g(\tilde{Q}(0), \tilde{Z}(0))=0$ and $|(\tilde{Q}(0), \tilde{Z}(0))| \leq 1$, $g(\tilde{Q}(t),\tilde{Z}(t))=0$ for $t \geq 0$.
	\label{SSCcondition}
\end{assumption}

We are ready to state the main result of this chapter.

\begin{theorem}
	Let $\{\mathbb{X}^r_{\pi}\}$ be a sequence of $\pi$-parallel random server systems processes. Suppose that Assumption \ref{heavytraffic} and Assumption \ref{steadycondition} hold, $g$ satisfies Assumption \ref{continuity}, the hydrodynamic model of the system satisfies Assumption \ref{SSCcondition}, and 
	\begin{equation}
	g(\diffuscal{Q}^r(0), \diffuscal{Z}^r(0)) \to 0 \text{  in probability} \label{initialcondition}
	\end{equation}
	as $r \to \infty$. Then, for each $T>0$,
	\begin{equation}
	\frac{||g(\diffuscal{Q}^r(t), \diffuscal{Z}^r(t))||_T}{( ||\diffuscal{Z}^r(t)||_T \vee 1)^c} \to 0 \text{  in probability} 
	\label{multiplicativeSSC}
	\end{equation}
	as $r \to \infty$, where $c>0$ is given as in \eqref{homogeneity}.
\label{SSC1}
\end{theorem}
\viva{Theorem \ref{SSC1} tells us that, with a well-defined SSC function $g$ for parallel random server systems under a specific control policy, one can show that function $g$ evaluated at $\diffuscal{Q}^{\systemindex}$ and $\diffuscal{Z}^{\systemindex}$ over a compact set converges to zero in some way. This implies $\diffuscal{Q}^{\systemindex}$ and $\diffuscal{Z}^{\systemindex}$ can be represented by a lower dimensional process using the relation between $\diffuscal{Q}^{\systemindex}$ and $\diffuscal{Z}^{\systemindex}$ in function $g$, which means the process states collapse from a higher dimension to to lower dimension. This SSC result can be then applied to analysing diffusion limits of such systems, which is much simpler now since the dimensions of diffusions are reduced.}
\label{Section-SSC-in-diffusion-limit}

\section{SSC framework}
\viva{In this section we explain in detail how the SSC result is obtained. Since we use the framework developed by \cite{Dai}, some of the steps remain the same as in that paper. We will explain the process with an emphasis in our contributions.}\par
\viva{The steps are as follows:
\begin{enumerate}
\item
Define the \textit{hydrodynamic scaling} $\mathbb{X}^{r,m}(\cdot)$ for the original system process $\mathbb{X}^r(\cdot)$.
\item
Show that the SSC function $g(\cdot)$ evaluated at the hydrodynamically scaled process $\mathbb{X}^{r,m}(t)$ is bounded by some function $H(t)$ plus an arbitrary $\epsilon$. The function $H(t)$ has the property that $H(t) \to 0$ as $t \to \infty$. This step works as a bridge to prove the final SSC result in step 3.
\item
Using the mathematical relation between the hydrodynamically scaled process $\mathbb{X}^{r,m}(\cdot)$ and diffusively scaled process $\diffuscal{\mathbb{X}}$, together with the boundness result in step 2, we can show that the SSC function $g(\cdot)$ evaluated at the diffusively scaled process converges to zero in probability, i.e.\ \eqref{multiplicativeSSC} is true.
\end{enumerate}
}
\viva{In Step 1, the hydrodynamic scaling is slightly different from \cite{Dai}. They consider servers in each pool aggregately, thus they define the scaling for the process of the number of busy servers in each pool. We need to treat each server individually since their service rates are random, therefore, apart from the process scaled above, we also need to define hydrodynamic scaling for the process $B_{ijk}(t)$ of each server.}\par
\viva{Step 3 is the same as in \cite{Dai}.}\par
\viva{Our main contributions are in Step 2. To get a better idea of this, we need to look at Step 2 in more detail. Step 2 can be further decomposed into 5 smaller steps:
\begin{enumerate}[label=\roman*)]
\item
Show that hydrodynamically scaled process $\mathbb{X}^{r,m}(\cdot)$ is almost Lipschitz.
\item
Define the \textit{hydrodynamic limit} $\tilde{\mathbb{X}}(\cdot)$ such that it is right continuous with left limits, and has the Lipschitz condition.
\item
Using \romannum{1}) and \romannum{2}), one can show that $\mathbb{X}^{r,m}(\cdot)$ converges to $\tilde{\mathbb{X}}(\cdot)$ u.o.c. as $r \to \infty$.
\item
Show that $\tilde{\mathbb{X}}(\cdot)$ satisfies the hydrodynamic model equations \eqref{HDfirst} to \eqref{HDlast}.
\item
Finally, by Assumptions \ref{continuity} and \ref{SSCcondition}, and using results in \romannum{3}) and \romannum{4}), show that the boundness result of step 2. 
\end{enumerate}
	}
\viva{The difficult part is in Step \romannum{1}) - more specifically, the almost Lipschitz condition of departure processes. In \cite{Dai}, to show this they use Proposition 4.3 of \cite{Bramson}, which is an application of Chebyshev's inequality on renewal processes. They use this proposition on departure processes of each pool when the service rates are given. However, in our systems, service rates within pools are random and unknown, thus we can not use this proposition directly. We formulate a coupling method to achieve this. We assume the random service rates in each pool for all kind of customers are bounded, then generate a Poisson process with rate being the product of the number of servers in each pool and the upper bound of the service rates. Then we couple our real departure processes by splitting this Poisson process. We only need to show the almost Lipschitz condition for the generated Poisson process, then the same condition also holds for the coupled real departure processes since the time points when departures happen in the coupled processes are subsets of the generated Poisson process.}\par
\viva{Another main difference in our proof is in Step \romannum{4}). When we show that the hydrodynamic limits $\tilde{\mathbb{X}}(\cdot)$ satisfy the hydrodynamic model equations, we need to use the Law of Large Numbers because of the randomness among servers.}\par
\viva{The rest of this section is organized as follows: in Section \ref{Section-hydro-scaling}, we define the hydrodynamic scaling and prove \romannum{1}) of step 2. Then in Section \ref{Section-hydro-limits}, we use the hydrodynamic scaling to define the hydrodynamic limits, and show \romannum{3}) and \romannum{4}), hence provd step 2. In Section \ref{Section-proof-SSC-in-diffusion-limit}, we translate the boundness inequality in step 2 to a similar inequality for the diffusion-scaled processes. We finally show that this latter inequality implies the desired SSC result in the diffusion limit.}

\subsection{Hydrodynamic scaling and bounds}
We begin by defining the \textit{hydrodynamic scaling}. We divide the interval $[0,T]$ into $T\sqrt{|N^r|}$ intervals of length $\frac{1}{\sqrt{|N^r|}}$ and analyse the processes in each intervals. We index the intervals by $m$. For a nonnegative integer $m$, let
\begin{equation}
x_{r,m} = \left\vert  Z^r\left(\frac{m}{\sqrt{|N^r|}}\right)  - |N^r| \right\vert^2 \vee |N^r|,
\end{equation}
Note that the square root of the first two terms of $x_{r,m}$ gives the deviations of these processes from their fluid limits.

We define the hydrodynamic scaling by shifting and scaling the processes of $\mathbb{X}^r$ as follows. For a process $X^r$ associated with the $r$th process, we denote the hydrodynamic scaled version by $X^{r,m}$. For $A^r,A^r_s, A^r_q,C^r, D^r, T^r$, the hydrodynamic scaling is defined for $t \in [0,L]$, for some $L>0$ by
\begin{equation}
X^{r,m}(t) = \frac{1}{\sqrt{x_{r,m}}} \left( X^r \left( \frac{\sqrt{x_{r,m}}t}{|N^r|} + \frac{m}{\sqrt{|N^r|}} \right) - X^r \left( \frac{m}{\sqrt{|N^r|}} \right) \right).\label{generalscaling}
\end{equation} 
The hydrodynamic scaled versions of $Q^r$ and $B^r$ and $Z^r$ are defined as follows:
\begin{align}
Q^{r,m}(t) &= \frac{1}{\sqrt{x_{r,m}}} \left( Q^r \left( \frac{\sqrt{x_{r,m}}t}{|N^r|} + \frac{m}{\sqrt{|N^r|}} \right) \right), \label{Qscaling}\\ 
B^{r,m}_{ij \sindex}(t) &= \frac{1}{\sqrt{x_{r,m}}} \left( B^r_{ij \sindex} \left( \frac{\sqrt{x_{r,m}}t}{|N^r|} + \frac{m}{\sqrt{|N^r|}} \right) - x^*_{ij}\right), \label{Bscaling}\\
Z^{r,m}_{ij}(t) & = \frac{1}{\sqrt{x_{r,m}}} \left( Z^r_{ij}\left( \frac{\sqrt{x_{r,m}}t}{|N^r|} + \frac{m}{\sqrt{|N^r|}} \right) - N^r_j x^*_{ij} \right). \label{Zscaling}
\end{align}
Note that $Z^{r,m}_{ij}(t) = \sum_{\sindex=1}^{N^r_i}  B_{ij \sindex}^{r,m}(t) $ and $D^{r,m}_{ij}(t) = \sum_{\sindex=1}^{N^r_i} D_{ij \sindex}^{r,m}(t)$.\par
Observe that $x_{r,m}$ must be in the order of $|N^r|$ for $Q^r$ and $Z_{ij}^r$ to have meaningful diffusion limits. Also, if $x_{r,m}$ is in the order of $|N^r|$, then $Q^{r,m}(\cdot)$ and $Z_{ij}^{r,m}(\cdot)$ are very similar to the diffusion scaling. This reveals the relationship between the hydrodynamic and diffusion scaling that will be used to translate a SSC result from hydrodynamic limits to diffusion limits.

For notational convenience, with a slight abuse of notation, we set
\begin{align*}
V^{r,m}_{ij \sindex}(D^{r,m}_{ij \sindex}(t),b) 
& = \frac{1}{\sqrt{x_{r,m}}} \Bigg( V_{ij \sindex} \left( D^r_{ij \sindex} \left( \frac{\sqrt{x_{r,m}}t}{|N^r|} + \frac{m}{\sqrt{|N^r|}} \right) + b_1 \right)\\
& \qquad \qquad \quad - V_{ij \sindex} \left(D^r_{ij \sindex} \left( \frac{m}{\sqrt{|N^r|}} \right) +b_2  \right) \Bigg),\numberthis \label{srvscaling}
\end{align*}
and for $b = (b_1, b_2) \in \mathbb{R}^2$. By \eqref{srvtinequa},
\begin{equation}
V^{r,m}_{ij \sindex} (D^{r,m}_{ij \sindex}(t), (0,1)) \leq T_{ij \sindex}^{r,m} \leq V^{r,m}_{ij \sindex} (D^{r,m}_{ij \sindex}(t), (1,0)).
\end{equation}
Let $\mathbb{X}^{r,m} = (A^{r,m},A_s^{r,m},A_q^{r,m},Q^{r,m},B^{r,m},Z^{r,m},C^{r,m},T^{r,m},D^{r,m})$. We refer to $\mathbb{X}^{r,m}$ as the hydrodynamic scaled process. From the definition of $x_{r,m}$ we have that 
$$|\mathbb{X}^{r,m}(0)| \leq 1.$$

It can easily be checked that $\mathbb{X}^{r,m}$ satisfies the following equations for all $t \geq 0$:  
\begin{equation}
A^{r,m}_j(t) = \sum_{i \in \mathcal{I}(j)}A^{r,m}_{sij}(t) + A^{r,m}_{qj}(t), \mbox{ for all } j \in \mathcal{J}, \label{process1}
\end{equation}
\begin{equation}
Q^{r,m}_j(t) = Q^{r,m}_j(0) + A^{r,m}_{qj}(t) - \sum_{i \in \mathcal{I}(j)}C^{r,m}_{ij}(t), \mbox{ for all } j \in \mathcal{J}, \label{process2}
\end{equation}
\begin{equation}
Z^{r,m}_{ij}(t) = Z^{r,m}_{ij}(0) + A^{r,m}_{sij}(t) + C^{r,m}_{ij}(t) - D^{r,m}_{ij}(t) , \mbox{ for all } i \in \mathcal{I}, \mbox{ and } j \in \mathcal{J}(i), \label{process3}
\end{equation}
\begin{equation}
D^{r,m}_{ij}(t) = \sum_{\sindex=1}^{N^r_i} D^{r,m}_{ij \sindex}(t), \mbox{ for all } i \in \mathcal{I}, \mbox{ and } j \in \mathcal{J}(i), \label{process8}
\end{equation}
\begin{align*}
D^{r,m}_{ij \sindex}(t) = &\frac{S_{ij \sindex}(\mu_{ij \sindex} (\sqrt{x_{r,m}} T^{r,m}_{ij \sindex} (t) + T^r_{ij \sindex}(m/\sqrt{|N^r|}))) - S_{ij \sindex}(\mu_{ij \sindex} ( T^r_{ij \sindex}(m/\sqrt{|N^r|})))}{\sqrt{x_{r,m}}}, \\
& \mbox{ for all } i \in \mathcal{I}, \mbox{ and } j \in \mathcal{J}(i), \mbox{ and } \sindex = 1,2,\dots, N^r_i, \numberthis \label{process4}
\end{align*}
\begin{equation}
T^{r,m}_{ij \sindex} (t) = \frac{x^*_{ij}}{|N^r|}t + \frac{\sqrt{x_{r,m}}}{|N^r|} \int_0^t B^{r,m}_{ij \sindex}(s) ds, \mbox{ for all } i \in \mathcal{I}, \mbox{ and } j \in \mathcal{J}(i), \mbox{ and } \sindex = 1,2,\dots, N^r_i, \label{process5}
\end{equation}
\begin{equation}
Q^{r,m}_j(t) \left( \sum_{i \in \mathcal{I}(j)} \sum_{j^{\prime} \in \mathcal{J}(i)}Z^{r,m}_{ij^{\prime}}(t) \right) = 0, \mbox{ for all } j \in \mathcal{J}, \label{process6}
\end{equation}
\begin{equation}
\int_0^t \left( \sum_{i \in \mathcal{I}(j)} \sum_{j^{\prime} \in \mathcal{J}(i)} Z^{r,m}_{ij^{\prime}}(s-) \right) dA^{r,m}_{qj}(s) = 0, \mbox{ for all } j \in \mathcal{J}.\label{process7}
\end{equation}

Now we have a similar result to Proposition 5.1 in \cite{Dai} for random service rates systems. 

\begin{proposition}
	Let $\{\mathbb{X}^r_{\pi}\}$ be a sequence of $\pi$-parallel random server systems processes. Assume \eqref{server-pool-percertage} and \eqref{arriallimits} hold, and $\pi$ satisfies Assumption \ref{steadycondition}. Fix $\epsilon>0, L>0$ and $T>0$. Then, for large enough $r$, there exists $N>0$ such that
	\begin{equation}
	P\left\lbrace \max_{m<\sqrt{|N^r|}T} \left\Vert A^{r,m}(t) - \frac{\lambda^r}{|N^r|}t \right\Vert_{L} > \epsilon \right\rbrace \leq \epsilon, \label{prop1equ1}
	\end{equation}
	\begin{equation}
	P\left\lbrace \max_{m<\sqrt{|N^r|}T} \sup_{0 \leq t_1, t_2 \leq L} |D^{r,m}(t_2) - D^{r,m}(t_1)| > N|t_2 - t_1| + \epsilon \right\rbrace \leq \epsilon, \text{  and } \label{prop1equ2}
	\end{equation}
	\begin{align*}
	&P\left\lbrace \max_{m<\sqrt{|N^r|}T} \left\Vert V_{ij \sindex}^{r,m}(D_{ij \sindex}^{r,m}(t),b) - \frac{1}{\mu_{ij \sindex}}D^{r,m}_{ij \sindex}(t) \right\Vert_{L} > \epsilon \right\rbrace \leq \epsilon \\
	&\qquad \text{  for all } i \in \mathcal{I}, j \in \mathcal{J}(i), \text{ and } k =1,2,\dots,N^{\systemindex}_i, \numberthis \label{prop1equ3}
	\end{align*}
	where $b = (1,0) \text{ or } (0,1)$. \label{prop5.1}
\end{proposition}

The proof of \eqref{prop1equ1} is the same as in \cite{Dai}, so we only need to show \eqref{prop1equ2} and \eqref{prop1equ3}.

\begin{proof}[Proof of \eqref{prop1equ2}]
	\viva{For \eqref{prop1equ2}, \cite{Dai} proved the analogous equation (77) in that paper when servers are identical inside each pool, thus they can show it directly using properties of Poisson processes.} However, in our case, there is randomness among servers for each $D^r_{ij}$, which will cause troubles for a direct proof. To resolve this, we make a detour by comparing the departure processes of such systems to the ones of a homogeneous system. \viva{The departure processes of homogeneous systems can be proved to exhibit the almost Lipschitz condition in the direct way as shown in \cite{Dai}. Then we use a coupling method to construct stochastically equivalent departure processes of the heterogeneous systems. Finally we compare the equivalent departure processes with the homogeneous departure processes, and obtain the almost Lipschitz condition for the former processes, and we conclude the  original departure processes also have the almost Lipschitz property.}\par
	To this end, let $S^{r,q}_{ij}(t)$ be a Poisson process with rate $N_i q_{ij}$ for all $i \in \mathcal{I}$ and $j \in \mathcal{J}(i)$. \viva{First we show the almost Lipschitz condition for the sequence of processes $S^{r,q}_{ij}(t), i \in \mathcal{I}, j \in \mathcal{J}(i).$}\\ 
	Notice that we only need to investigate the process when $m=0$, then multiply the error bounds. To see this, first define the hydrodynamic scaling for $S^{r,q}_{ij}(t)$: 
	\begin{equation}
	S^{r,q,m}_{ij}(t) = \frac{1}{\sqrt{x_{r,m}}} \left(S^{r,q}_{ij} \left( \frac{\sqrt{x_{r,m}}t}{|N^r|} + \frac{m}{\sqrt{|N^r|}} \right) - S^{r,q}_{ij} \left( \frac{m}{\sqrt{|N^r|}}\right) \right)
	\end{equation}
	where
	\begin{equation}
	x_{r,m} = \left\vert Z^r\left(\frac{m}{\sqrt{|N^r|}}\right) - |N^r| \right\vert^2 \vee |N^r|. \label{scalingparm}
	\end{equation}
	Here $Q^r(\cdot), \mbox{ and } Z^r(\cdot)$ are the queue length and number of busy servers respectively.
	Then, similar to the proof in C.2.1 in \cite{Dai}, we have
	\begin{align*}
	&P \Bigg\lbrace \max_{m < \sqrt{|N^r|}T} \Bigg\Vert S^{r,q,m}_{ij}(t) - \frac{N^r_i q_{ij}}{|N^r|}t \Bigg\Vert_L >  \frac{N^r_i}{|N^r|}q_{ij} \epsilon L \Bigg\rbrace\\
	&= P \Bigg\lbrace \max_{m < \sqrt{|N^r|}T} \Bigg\Vert \frac{1}{\sqrt{x_{r,m}}} \left(S^{r,q}_{ij} \left( \frac{\sqrt{x_{r,m}}t}{|N^r|} + \frac{m}{\sqrt{|N^r|}} \right) - S^{r,q}_{ij} \left( \frac{m}{\sqrt{|N^r|}}\right) \right) \\
	& \qquad \qquad \qquad \qquad - \frac{N^r_i q_{ij}}{|N^r|}t \Bigg\Vert_L  >  \frac{N^r_i}{|N^r|} q_{ij} \epsilon L \Bigg\rbrace\\
	&= P \Bigg\lbrace \max_{m < \sqrt{|N^r|}T} \Bigg\Vert \frac{1}{\sqrt{x_{r,m}}} \left(S \left( \frac{N^r_i}{|N^r|} q_{ij}  \sqrt{x_{r,m}}t + \frac{N^r_i}{\sqrt{|N^r|}} q_{ij}  m \right) - S \left( \frac{N^r_i}{\sqrt{|N^r|}} q_{ij} m\right) \right) \\
	& \qquad \qquad \qquad \qquad - \frac{N^r_i q_{ij}}{|N^r|}t \Bigg\Vert_L >  \frac{N^r_iq_{ij} }{|N^r|} \epsilon L \Bigg\rbrace\\
	&= P \Bigg\lbrace \max_{m < \sqrt{|N^r|}T} \Bigg\Vert \left(S \left( t + \frac{N^r_i}{|N^r|} q_{ij}\sqrt{|N^r|}m \right) - S \left( \frac{N^r_i}{|N^r|} q_{ij}\sqrt{|N^r|}m\right) \right) - t \Bigg\Vert_{\frac{N^r_i}{|N^r|} q_{ij}\sqrt{x_{r,m}}L} \\
	& \qquad \quad >   \frac{N^r_iq_{ij} }{|N^r|} \epsilon L \sqrt{x_{r,m}} \Bigg\rbrace.
	\end{align*}
	\vspace{-5mm}
	Since $S(t)$ is a homogeneous Poisson process, the probability above is equal to
	\begin{align*}
	&P \left\lbrace \max_{m < \sqrt{|N^r|}T} \left\Vert \left(S \left( t \right) - S \left( 0 \right) \right) - t \right\Vert_{\frac{N^r_i}{|N^r|} q_{ij}\sqrt{x_{r,m}}L}  >  \frac{N^r_i }{|N^r|} q_{ij} \epsilon L \sqrt{x_{r,m}} \right\rbrace\\
	& \leq \sum_{m<\sqrt{|N^r|}T} P \left\lbrace \left\Vert S \left( t \right) - t \right\Vert_{\frac{N^r_i}{|N^r|} q_{ij}\sqrt{x_{r,m}}L}  >  \frac{N^r_i }{|N^r|} q_{ij} \epsilon L \sqrt{x_{r,m}} \right\rbrace, \numberthis \label{PPNqequ1}
	\end{align*}
	because $x_{r,m} \geq |N^r|$ by \eqref{scalingparm}, and $\frac{N^r_i}{N^r} \to \beta_i \in (0,1)$, and by Proposition 4.3 of \cite{Bramson}, for given $\epsilon$, any $m$ and large enough $r$,
	\begin{equation}
	P \left\lbrace \left\Vert S \left( t \right) - t \right\Vert_{\frac{N^r_i}{|N^r|} q_{ij}\sqrt{x_{r,m}}L}  >  \frac{N^r_i }{|N^r|}q_{ij} \epsilon L \sqrt{x_{r,m}} \right\rbrace \leq \frac{\epsilon}{\frac{N^r_i }{|N^r|}q_{ij}\sqrt{x_{r,m}}L} \leq \frac{\epsilon}{2 \beta_i q_{ij}\sqrt{|N^r|}L}. \label{convergetomean}
	\end{equation}
	Thus we only need to consider the process when $m=0$, then multiply by the error bound $\lceil \sqrt{|N^r|}T \rceil$. Using this result
	\begin{align*}
	&P\left\lbrace \sup_{0 \leq t_1 \leq t_2 \leq L} \left\vert \left( S^{r,q,0}_{ij}\left(t_2\right) - \frac{N^r_i q_{ij}}{|N^r|}t_2 \right) - \left(S^{r,q,0}_{ij}\left(t_1\right) - \frac{N^r_i q_{ij}}{|N^r|}t_1\right) \right\vert > \epsilon \right\rbrace\\
	&\leq P \left\lbrace  \left\Vert S^{r,q,0}_{ij}(t_2) - \frac{N^r_i q_{ij}}{|N^r|}t_2 \right\Vert_L + \left\Vert S^{r,q,0}_{ij}(t_1) - \frac{N^r_i q_{ij}}{|N^r|}t_1 \right\Vert_L  > \epsilon \right\rbrace\\
	&\leq P \left\lbrace \left\Vert S^{r,q,0}_{ij}(t_2) - \frac{N^r_i q_{ij}}{|N^r|}t_2 \right\Vert_L > \frac{\epsilon}{2} \right\rbrace + P \left\lbrace \left\Vert S^{r,q,0}_{ij}(t_1) - \frac{N^r_i q_{ij}}{|N^r|}t_1 \right\Vert_L > \frac{\epsilon}{2} \right\rbrace\\
	&=P \left\lbrace \left\Vert S \left( t_2 \right) - t_2 \right\Vert_{\frac{N^r_i }{|N^r|} q_{ij}\sqrt{x_{r,0}}L}  > \frac{N^r_i }{|N^r|} q_{ij} \epsilon L \sqrt{x_{r,0}} \right\rbrace \\
	& \quad + P \left\lbrace \left\Vert S \left( t_1 \right) - t_1 \right\Vert_{\frac{N^r_i }{|N^r|}q_{ij}\sqrt{x_{r,0}}L}  > \frac{N^r_i }{|N^r|} q_{ij} \epsilon L \sqrt{x_{r,0}} \right\rbrace\\
	&\leq \frac{\epsilon}{\frac{N^r_i }{|N^r|} q_{ij}\sqrt{x_{r,0}}L} \leq \frac{\epsilon}{2 \beta_i q_{ij}\sqrt{|N^r|}L}.
	\end{align*}
	Therefore, reselecting $\epsilon$, we get
	\begin{equation}
	P\left\lbrace \sup_{0 \leq t_1 \leq t_2 \leq L} \left\vert S^{r,q,0}_{ij}\left(t_2\right)  - S^{r,q,0}_{ij}\left(t_1\right) \right\vert > \frac{N^r_i }{|N^r|}q_{ij}|t_2 - t_1| + \epsilon \right\rbrace \leq \frac{\epsilon}{\sqrt{|N^r|}}.
	\end{equation}
	Notice that $\frac{N^r_i }{|N^r|}$ depends on $r$ in the Lipschitz coefficient. To get rid of this one can simply use the fact that, for large enough $r$, $\beta_i \in \left[\frac{N^r_i }{|N^r|}-\delta, \frac{N^r_i }{|N^r|}+\delta \right]$ for some really small $\delta > 0$, thus
	\begin{align*}
	P & \left\lbrace \sup_{0 \leq t_1 \leq t_2 \leq L} \left\vert S^{r,q,0}_{ij}\left(t_2\right)  - S^{r,q,0}_{ij}\left(t_1\right) \right\vert > \beta_i q_{ij}|t_2 - t_1| + \epsilon \right\rbrace\\
	& \leq \left\lbrace \sup_{0 \leq t_1 \leq t_2 \leq L} \left\vert S^{r,q,0}_{ij}\left(t_2\right)  - S^{r,q,0}_{ij}\left(t_1\right) \right\vert > \left(\frac{N^r_i }{|N^r|}- \delta \right) q_{ij}|t_2 - t_1| + \epsilon \right\rbrace \\
	& \leq \left\lbrace \sup_{0 \leq t_1 \leq t_2 \leq L} \left\vert S^{r,q,0}_{ij}\left(t_2\right)  - S^{r,q,0}_{ij}\left(t_1\right) \right\vert > \frac{N^r_i }{|N^r|} q_{ij}|t_2 - t_1| - \delta q_{ij}|t_2 - t_1| + \epsilon \right\rbrace.
	\end{align*} 
	Since $|t_2 - t_1| < L$, the term $\delta q_{ij}|t_2 - t_1|$ can be chosen to be less than $\epsilon$; therefore, after a reselection of $\epsilon$, we have
	\begin{equation}
	P \left\lbrace \sup_{0 \leq t_1 \leq t_2 \leq L} \left\vert S^{r,q,0}_{ij}\left(t_2\right)  - S^{r,q,0}_{ij}\left(t_1\right) \right\vert > \beta_i q_{ij}|t_2 - t_1| + \epsilon \right\rbrace \leq \frac{\epsilon}{\sqrt{|N^r|}}. \label{Lipschitz-potential-departure}
	\end{equation}
	
	\viva{Inequality \eqref{Lipschitz-potential-departure} is sufficient for our comparison. }Now we consider the departure processes with random service rates. We want to show \eqref{prop1equ2}. To overcome the difficulty in proving it directly, we compare $D^r_{ij}$ with $S^{r,q}_{ij}$. In order to do this, we find a stochastically equivalent process $\breve{D}^r_{ij}$ of $D^r_{ij}$, which is the main difference between our proof and the one of \cite{Dai}. We will do this in the following way.\par
	For pool $i$, let $J(i)$ be the cardinality of $\mathcal{J}(i)$, i.e.\ total number of customer classes that servers in pool $i$ are capable to serve. Then let $\{S^{r,q}_{ij_a}(t), a = 1,2,\dots, J(i)\}$ be a sequence of generated Poisson processes with rate $N^r_i q_{ij_a}$ respectively. And let $0<\tau^a_1<\tau^a_2<\cdots$ be the sequence of occurrence times of the process $S^{r,q}_{ij_a}$, i.e.\ $S^{r,q}_{ij_a}(t) = \sum_{n=1}^{\infty}I(\tau^a_n < t)$. Let $\{ U_n, n \in \mathbb{N} \}$ be a sequence of independent uniform $(0,1)$ random variables, and $I(A)$ be the indicator function for event $A$ which takes value $1$ if $A$ occurs and $0$ otherwise. Assume all of the processes are equal to zero at $t=0$. By splitting the processes $\{S^{r,q}_{ij_a}(t)\}$, we define the following processes
	\begin{align*}
	&\breve{D}^{r,n}_{ij_1}=\sum_{l=1}^{n_1} I\left( U_l \leq \frac{ \sum_{\sindex=1}^{N^r_i} \mu_{ij_1 \sindex} \breve{B}_{ij_1 \sindex}(\tau^1_l-)}{\sum_{a=1}^{J(i)}N^r_i q_{ij_a}} \right), \numberthis\\
	&\breve{D}^{r,n}_{ij_a}=\sum_{l=1}^{n_a} I\left( \frac{\sum_{c=1}^{a-1} N^r_i q_{ij_c}}{\sum_{a=1}^{J(i)}N^r_i q_{ij_a}} \leq U_l \leq \frac{\sum_{c=1}^{a-1} N^r_i q_{ij_c} + \sum_{\sindex=1}^{N^r_i} \mu_{ij_a \sindex} \breve{B}_{ij_a \sindex}(\tau^a_l-)}{\sum_{a=1}^{J(i)}N^r_i q_{ij_a}} \right),\\
	& \qquad a=2,3,\dots,J(i), \numberthis \label{stochastic-equivalent-departure}\\
	&S^{r,q}_{ij_a}(t) = n_a = \sum_{l=1}^{\infty}I(\tau^a_l < t), \ a=1,2,\dots,J(i), \numberthis \label{potential-departure}\\
	&\breve{D}^{r}_{ij_a}(t)= \breve{D}^{r,n}_{ij_a}, \forall t \in [\tau^a_n, \tau^a_{n+1}), \ a=1,2,\dots,J(i), \numberthis
	\end{align*}
	where $\breve{B}_{ij \sindex}(t)$ is determined by the arrival processes $A^r_i(t)$, the routing policies, the sequence $\tau^a_1, \tau^a_2, \dots$, and the selection scheme defined as follows: we use class $j_1$ customer as an example,  and other types of customer departure processes will be similar. If, for some $l \in \{1,2,\dots, S^{r,q}_{ij_1}(t)\}$, $I\left(U_l \leq \frac{\sum_{\sindex=1}^{N^r_i} \mu_{ij_1 \sindex} \breve{B}_{ij_1 \sindex}(\tau^1_l-)}{N^r_i q_{ij_1}}\right)=1$, then the potential departure occurring at time $\tau^1_l$ in process $S^{r,q}_{ij_1}(t)$ is accepted as the real departure for process $\breve{D}^{r}_{ij_1}(t)$. Assume there are $m$ busy servers just before this departure occurs (time $\tau^1_l-$). Then after it is accepted as a real departure, one of the $m$ servers will be freed, which leads to our selection scheme. Generate a uniformly distributed random variable $\varsigma$ on $(0,1)$. If $\frac{\sum_{\sindex=0}^{h} \mu_{ij_1 \sindex}}{\sum_{\sindex=1}^{m} \mu_{ij_1 \sindex}} \leq \varsigma < \frac{\sum_{\sindex=0}^{h+1} \mu_{ij_1 \sindex}}{\sum_{\sindex=1}^{m} \mu_{ij_1 \sindex}},\ h=0,1,\dots,m-1$, then the $h+1$th server will be freed at time $\tau^1_l$. Here we let $\mu_{ij_10}=0$.
	\vspace{-3mm}\par
 	Under such a definition, the process $\breve{D}^{r}_{ij_a}(t)$ is stochastically equivalent to $D^{r}_{ij_a}(t)$ for all $a = 1,\dots, J(i)$, thus in order to show
	\begin{equation*}
	P\left\lbrace \max_{m<\sqrt{|N^r|}T} \sup_{0 \leq t_1, t_2 \leq L} |D^{r,m}_{ij_a}(t_2) - D^{r,m}_{ij_a}(t_1)| > N|t_2 - t_1| + \epsilon \right\rbrace \leq \epsilon,
	\end{equation*}
	it is equivalent to show
	\begin{equation}
	P\left\lbrace \max_{m<\sqrt{|N^r|}T} \sup_{0 \leq t_1, t_2 \leq L} |\breve{D}^{r,m}_{ij_a}(t_2) - \breve{D}^{r,m}_{ij_a}(t_1)| > N|t_2 - t_1| + \epsilon \right\rbrace \leq \epsilon. \label{depainequ}
	\end{equation}
	From the construction of process $\breve{D}^r_{ij_a}$, it is easy to see that for every $m$,
	\begin{align*}
	&\left\vert \breve{D}^{r,m}_{ij_a}(t_2) - \breve{D}^{r,m}_{ij_a}(t_1) \right\vert \\
	&= \frac{1}{\sqrt{x_{r,m}}} \left\vert \left( \breve{D}^r_{ij_a} \left(\frac{\sqrt{x_{r,m}}t_2}{|N^r|} + \frac{m}{\sqrt{|N^r|}} \right) - \breve{D}^r_{ij_a} \left( \frac{m}{\sqrt{|N^r|}}\right) \right) \right. \\ 
	&\quad \left.- \left( \breve{D}^r_{ij_a} \left(\frac{\sqrt{x_{r,m}}t_1}{|N^r|} + \frac{m}{\sqrt{|N^r|}} \right) - \breve{D}^r_{ij_a} \left( \frac{m}{\sqrt{|N^r|}}\right) \right) \right\vert\\
	&= \frac{1}{\sqrt{x_{r,m}}} \left\vert \breve{D}^r_{ij_a} \left(\frac{\sqrt{x_{r,m}}t_2}{|N^r|} + \frac{m}{\sqrt{|N^r|}} \right) - \breve{D}^r_{ij_a} \left(\frac{\sqrt{x_{r,m}}t_1}{|N^r|} + \frac{m}{\sqrt{|N^r|}} \right) \right\vert\\
	&\leq \frac{1}{\sqrt{x_{r,m}}} \left\vert S^{r,q}_{ij_a} \left(\frac{\sqrt{x_{r,m}}t_2}{|N^r|} + \frac{m}{\sqrt{|N^r|}} \right) - S^{r,q}_{ij_a} \left(\frac{\sqrt{x_{r,m}}t_1}{|N^r|} + \frac{m}{\sqrt{|N^r|}} \right) \right\vert\\
	&= \left\vert S^{r,q,m}_{ij_a}(t_2) - S^{r,q,m}_{ij_a}(t_1) \right\vert. \numberthis \label{comparison-real-potential-departure}
	\end{align*}
	The inequality is from the fact that the event points of $\breve{D}^r_{ij_a}(t)$ are chosen from the event points of $S^{r,q}_{ij_a}(t)$ (equations \eqref{stochastic-equivalent-departure} and \eqref{potential-departure}), thus the departure difference between time $t_2$ and $t_1$ is a subset of the difference of $S^{r,q}_{ij_a}(t)$ during the same time. Then we also only need to consider the process when $m=0$. From \eqref{Lipschitz-potential-departure} and \eqref{comparison-real-potential-departure}, we have
	\begin{equation}
	P\left\lbrace \sup_{0 \leq t_1 \leq t_2 \leq L} \left\vert \breve{D}^{r,0}_{ij_a}\left(t_2\right)  - \breve{D}^{r,0}_{ij_a}\left(t_1\right) \right\vert > \beta_i q_{ij_a}|t_2 - t_1| + \epsilon \right\rbrace \leq \frac{\epsilon}{\sqrt{|N^r|}}.
	\end{equation}
	\vspace{1mm}
	Note that we can conduct the same analysis for each pool, thus we can get rid of the $a$ here. \\
	\vspace{6mm}
	Let $N = \max_{i \in \mathcal{I}, j \in \mathcal{J}}\{ \beta_i q_{ij} \}$. Then, for all $i \in \mathcal{I}$ and $j \in \mathcal{J}$,
	\begin{equation}
	P\left\lbrace \sup_{0 \leq t_1 \leq t_2 \leq L} \left\vert \breve{D}^{r,0}_{ij}\left(t_2\right)  - \breve{D}^{r,0}_{ij}\left(t_1\right) \right\vert > N |t_2 - t_1| + \epsilon \right\rbrace \leq \frac{\epsilon}{\sqrt{|N^r|}}.
	\label{mequ0}
	\end{equation}
	Multiplying the error bound $\lceil \sqrt{|N^r|}T \rceil$ and enlarging $\epsilon$ appropriately we obtain \eqref{prop1equ2}.
\end{proof}

\viva{Before proving \eqref{prop1equ3}, we need a lemma taken from \cite[Lemma 5.1]{Bramson}.}
\begin{lemma}
	Let $v^{r,T,\max}_{ij\sindex} = \max\{ |v_{ij\sindex}(l)|:= V_{ij \sindex}(l-1) \leq |N^{\systemindex}| T \}$ for all $i \in \mathcal{I}, j \in \mathcal{J}(i), \sindex=1,2,\dots,N^{\systemindex}_i$. Then, for given $T$,
	\begin{equation}
	v^{r,T,\max}_{ij \sindex}/\sqrt{|N^{\systemindex}|} \to 0 \mbox{ in probability as } \systemindex \to \infty, \mbox{ for all } i \in \mathcal{I}, j \in \mathcal{J}(i), \sindex =1,2,\dots, N^{\systemindex}_i.
	\end{equation}
	\label{lemma-first-residual-time}
\end{lemma}
\vspace{-3mm}
The proof is in Appendix \ref{Appendix-Proof-of-prop1equ3}.

Using this proposition, one can show that $\mathbb{X}^{r,m}$ is almost Lipschitz \viva{in probability}, as described in the next proposition. In this section and for the remainder of this paper, $N$ without a superscript is reused to denote a general constant.

\vspace{2mm}

\begin{proposition}
	Let $\{\mathbb{X}^r_{\pi}\}$ be a sequence of $\pi$-parallel random server systems processes. Assume that Assumption \ref{heavytraffic} and Assumption \ref{steadycondition} hold. Fix $\epsilon>0$, $L>0$, and $T>0$. Then for large enough $r$,
	\begin{equation}
	P \left\lbrace \max_{m<\sqrt{|N^r|}T} \sup_{0 \leq t_1 \leq t_2 \leq L} \left\vert \mathbb{X}^{r,m}(t_2) - \mathbb{X}^{r,m}(t_1) \right\vert > N |t_2 - t_1| + \epsilon \right\rbrace \leq \epsilon,
	\end{equation}
	where $N<\infty$ and only depends on $\lambda$. \label{prop5.3}
\end{proposition}
The proof is similar to that of Proposition 5.3 in \cite{Dai}. We include it into Appendix \ref{Appendix-prop5.3}.

For convenience, we assume for the rest of the paper that $N \geq 1$ and $L \geq 1$.
Let
\begin{equation}
\mathscr{K}^r_0 =  \left\lbrace \max_{m < \sqrt{|N^r|}T} \sup_{0 \leq t_1 \leq t_2 \leq L} |\mathbb{X}^{r,m}(t_1) - \mathbb{X}^{r,m}(t_2)| \leq N|t_1-t_2| + \epsilon(r) \right\rbrace, \label{almostLP}
\end{equation}
where $N,L,$ and $T$ are fixed as before and $\epsilon(r)$ with $\epsilon \to 0$ as $r \to \infty$ is a sequence of real numbers. Similarly, we can replace $\epsilon$ in \eqref{prop1equ1}, \eqref{prop1equ2}, and \eqref{prop1equ3} by $\epsilon(r)$. We denote these new inequalities obtained from \eqref{prop1equ1}, \eqref{prop1equ2}, and \eqref{prop1equ3} by \eqref{prop1equ1}$'$, \eqref{prop1equ2}$'$, and \eqref{prop1equ3}$'$. Let $\mathscr{K}^r$ denote the intersection of $\mathscr{K}^r_0$ with the complements of the events in \eqref{prop1equ1}$'$, \eqref{prop1equ2}$'$, and \eqref{prop1equ3}$'$. As in \cite{Dai}, when $\epsilon(r) \to 0$ sufficiently slowly as $r \to \infty$, one can show that $P(\mathscr{K}^r) \to 1$ as $r \to \infty$.
\vspace{2mm} \par
We summarize the above discussion in the following corollary for future reference, which is similar to \cite[Corollary 5.1]{Dai}.

\begin{corollary}
	Let $\{\mathbb{X}^r_{\pi}\}$ be a sequence of $\pi$-parallel random server system processes. Assume that Assumption \ref{heavytraffic} and Assumption \ref{steadycondition} hold. Fix $L>0$, and $T>0$ and choose $\epsilon(r)$ as above. Then for $\mathscr{K}^r$ defined as above
	\begin{equation}
	\lim_{r \to \infty} P(\mathscr{K}^r) = 1. \label{goodpoints1}
	\end{equation}
\end{corollary}

\label{Section-hydro-scaling}

\subsection{Hydrodynamic limits}
In this section, we define the hydrodynamic limits. First, we define a set of functions that contains all of the hydrodynamic limits. The following definitions are similar to those in \cite{Dai}, and the notation is adapted from that paper.

Fix $L>0$. Let $\tilde{E}$ be the set of right continuous functions with left limits, $x:[0,L] \to \mathbb{R}^d$. Let $E'$ denote those $x \in \tilde{E}$ that satisfy
\begin{equation}
|x(0)| \leq 1
\end{equation}
and
\begin{equation}
|x(t_2) - x(t_1)| \leq N|t_1 - t_2| \text{ for all } t_1, t_2 \in [0,L],
\end{equation}
where the constant $N$ is chosen as in Proposition \ref{prop5.3}. We set
\begin{equation}
E^r = \{ \mathbb{X}^{r,m}, m < \sqrt{|N^r|}T, \omega \in \mathscr{K}^r \}
\end{equation}
and
\begin{equation}
\mathscr{E} =  \{ E^r: r \in \mathbb{N} \}
\end{equation}
where $T$ is fixed, and $\mathscr{K}^r$ is defined as in the previous section.

We define a \textit{hydrodynamic limit} $x$ of $\mathscr{E}$ to be a point $ x \in \tilde{E}$ such that for all $\epsilon > 0$ and $r_0 \in \mathbb{N}$, there exist $r \geq r_0$ and $y \in E^r$, with $\left\Vert x(\cdot) - y(\cdot) \right\Vert_L < \epsilon$.

Because
\begin{equation}
|\mathbb{X}^{r,m}(0)| \leq 1
\end{equation}
for all $m < \sqrt{|N^r|}T$ and $r \in \mathbb{N}$, the following result is a corollary in \cite{Dai}, and is similar to Corollary 5.2 in that paper. It shows that the hydrodynamic limits are ``rich'' in the sense that, for $r$ large enough, every hydrodynamic scaled process is close to a hydrodynamic limit. One can use the almost Lipschitz property of processes $\mathbb{X}^{r,m}$ to show this-see \cite[Lemma 4.2]{Bramson}.

\begin{corollary}
	Let $\{\mathbb{X}^r_{\pi}\}$ be a sequence of $\pi$-parallel random server systems processes. Assume that Assumption \ref{heavytraffic} holds and $\pi$ satisfies \ref{steadycondition}. Let $\tilde{E}, E^r$, and $\mathscr{E}$ be as specified above. Fix $\epsilon > 0, L >0$ and $T>0$, and choose $r$ large enough. Then, for $\omega \in \mathscr{K}^r$ and any $m < \sqrt{|N^r|}T$, there exists a hydrodynamic limit $\tilde{\mathbb{X}}(\cdot) \in E'$, such that
	\begin{equation}
	\left\Vert \mathbb{X}^{r,m}(\cdot) - \tilde{\mathbb{X}}(\cdot) \right\Vert_L \leq \epsilon.
	\end{equation}
	\label{convergencecorollary}
\end{corollary}
The next result is mainly needed to translate the condition on the hydrodynamic model solutions to hydrodynamic limits given in Assumption \ref{SSCcondition}. It also reveals the origin of hydrodynamic model equations.

\begin{proposition}
	Let $\{\mathbb{X}^r_{\pi}\}$ be a sequence of $\pi$-parallel random server system processes. Assume that Assumption \ref{heavytraffic} holds and $\pi$ satisfies Assumption \ref{steadycondition}. Choose $L > 0$ and let $\tilde{\mathbb{X}}_{\pi}$ be a hydrodynamic limit of $\mathscr{E}$ over $[0,L]$. $\tilde{\mathbb{X}}_{\pi}$ satisfies the hydrodynamic model equations \eqref{HDfirst}-\eqref{HDlast} on $[0,L]$. \label{mainbridge}
\end{proposition}

To prove \ref{mainbridge}, we need a lemma which appears to be the same as Lemma C.1 in \cite{Dai}, but it is for random rates systems.
\begin{lemma}
	Let $\{ \mathbb{X}^r \}$ be a sequence of $\pi$-parallel random server systems processes. Assume that Assumption \ref{heavytraffic} holds and $\pi$ satisfies Assumption \ref{steadycondition}. Fix $\epsilon >0, L>0, $ and $T>0$. Then, for large enough $r$ and $\omega \in \mathscr{A}$,
	\begin{equation}
	\max_{m < \sqrt{|N^r|}T} \frac{\sqrt{x_{r,m}}}{|N^r|} \int_0^L |Z^{r,m}_{ij}(s)| ds < \epsilon, \forall \ i \in \mathcal{I}, \mbox{ and } j \in \mathcal{J}(i).
	\end{equation}
	\label{lemma-bounds}
\end{lemma}

\begin{proof}
	The proof of this lemma is the same as in \cite{Dai}, except we do not consider $Q^{r,m}$. Recalling that $Z^{r,m}_{ij}(t) = \sum_{\sindex=1}^{N^r_i} B^{r,m}_{ij \sindex}(t)$, the proof becomes obvious.
\end{proof}
Now we prove the proposition.
\begin{proof}[The proof of Proposition \ref{mainbridge}]
	Proof is similar to that in \cite[Proposition 5.4]{Dai}. Assume that Assumption \ref{heavytraffic} holds, and  $\pi$ satisfies Assumption \ref{steadycondition}. $g$ satisfies Assumption \ref{continuity}. Fix $\omega \in \mathscr{K}^r$ and let $\mathbb{X}^{r,m}$ be given as in \eqref{generalscaling}-\eqref{Bscaling}. By \eqref{prop1equ1}$'$, we have, for large enough $r$, that
	\begin{equation}
	\left\Vert A^{r,m}(t) - \frac{\lambda^r}{|N^r|}t \right\Vert_L \leq \epsilon(r). \label{HDconvergence1}
	\end{equation}
	Using \eqref{process5} and Lemma \ref{lemma-bounds} gives
	\begin{equation}
	\left\Vert \sum_{\sindex=1}^{N^r_i} \left( T^{r,m}_{ij \sindex} (t) - \frac{x^*_{ij}}{|N^r|}t \right) \right\Vert_L \leq \epsilon(r). \label{HDconvergence3}
	\end{equation}
	Now select any hydrodynamic limit $\tilde{\mathbb{X}}$ of $\mathscr{E}$. By Corollary \ref{convergencecorollary}, for given $\delta > 0$, choose $(r,m)$ so that, $\epsilon(r) \leq \delta$,
	\begin{equation}
	\left\Vert \tilde{\mathbb{X}}(t) - \mathbb{X}^{r,m}(t,\omega) \right\Vert_L \leq \delta,\label{HDconvergence4}
	\end{equation}
	and\\
	\begin{equation}
	\left\vert \frac{\lambda^r}{|N^r|} - \lambda \right\vert \leq \delta.\label{HDconvergence5}
	\end{equation}
	It follows from \eqref{HDconvergence1} and \eqref{HDconvergence5} that
	\begin{equation}
	\left\Vert \tilde{A}(t) - \lambda t \right\Vert_L \leq (2+L)\delta.\label{HDconvergence6}
	\end{equation}
	
	\viva{Until now the proof remains the same as in \cite[Proposition 5.4]{Dai}. The next inequality is an analogy of inequality (C25) in \cite{Dai}. However, we can not derive (C25) directly in our case because we do not have a constant $\mu_{jk}$ here. Since we have i.i.d.\ random service rates $\{\mu_{ij \sindex}\}, i \in \mathcal{I}, j \in \mathcal{J}(i), \sindex=1,2,\dots,N_i$, using the Law of Large Numbers and properties of Poisson processes, we have} 
	\vspace{3mm}
	\begin{align*}
	&\left\Vert \sum_{\sindex=1}^{N^r_i} \left( T^{r,m}_{ij \sindex}(t) - \frac{D^{r,m}_{ij \sindex}(t)}{\mu_{ij \sindex}} \right) \right\Vert_L\\
	&= \frac{1}{\sqrt{x_{r,m}}} \left\Vert \sum_{\sindex=1}^{N^r_i}  \frac{1}{\mu_{ij \sindex}} \left( S_{ij \sindex} \left( \mu_{ij \sindex} T^r_{ij \sindex} \left( \frac{m}{\sqrt{|N^r|}} \right) \right)  - S_{ij \sindex}\left(\mu_{ij \sindex} T^r_{ij \sindex} \left( \frac{\sqrt{x_{r,m}}t}{|N^r|} + \frac{m}{\sqrt{|N^r|}} \right) \right) \right) \right.\\
	& \quad \left. - \sum_{\sindex=1}^{N^r_i} \frac{1}{\mu_{ij \sindex}} \left( \mu_{ij \sindex} T^r_{ij \sindex} \left( \frac{m}{\sqrt{|N^r|}} \right) - \mu_{ij \sindex} T^r_{ij \sindex} \left( \frac{\sqrt{x_{r,m}}t}{|N^r|} + \frac{m}{\sqrt{|N^r|}} \right) \right) \right\Vert_L\\
	& \leq \frac{1}{\sqrt{x_{r,m}}} \left\Vert \sum_{\sindex=1}^{N^r_i} \left( S_{ij \sindex} \left( \mu_{ij \sindex} T^r_{ij \sindex} \left( \frac{m}{\sqrt{|N^r|}} \right) \right)  - S_{ij \sindex}\left(\mu_{ij \sindex} T^r_{ij \sindex} \left( \frac{\sqrt{x_{r,m}}t}{|N^r|} + \frac{m}{\sqrt{|N^r|}} \right) \right) \right) \right.\\
	& \quad \left. - \sum_{\sindex=1}^{N^r_i} \left( \mu_{ij \sindex} T^r_{ij \sindex} \left( \frac{m}{\sqrt{|N^r|}} \right) - \mu_{ij \sindex} T^r_{ij \sindex} \left( \frac{\sqrt{x_{r,m}}t}{|N^r|} + \frac{m}{\sqrt{|N^r|}} \right) \right) \right\Vert_L\\
	& = \frac{1}{\sqrt{x_{r,m}}} \left\Vert \sum_{\sindex=1}^{N^r_i} \left( S_{ij \sindex} \left( \mu_{ij \sindex} \int_0^{\frac{m}{\sqrt{|N^r|}}} B_{ij \sindex}(s) ds \right)  - S_{ij \sindex}\left(\mu_{ij \sindex} \int_0^{\frac{\sqrt{x_{r,m}}t}{|N^r|} + \frac{m}{\sqrt{|N^r|}}} B_{ij \sindex}(s) ds \right) \right) \right.\\
	& \quad \left. - \sum_{\sindex=1}^{N^r_i} \left( \mu_{ij \sindex} \int_0^{\frac{m}{\sqrt{|N^r|}}} B_{ij \sindex}(s) ds - \mu_{ij \sindex} \int_0^{\frac{\sqrt{x_{r,m}}t}{|N^r|} + \frac{m}{\sqrt{|N^r|}}} B_{ij \sindex}(s) ds \right) \right\Vert_L\\
	& \overset{d}{=} \frac{1}{\sqrt{x_{r,m}}} \Bigg\Vert \sum_{\sindex=1}^{N^r_i} \left( S_{ij \sindex} \left( \mu_{ij \sindex} \int_{\frac{m}{\sqrt{|N^r|}}}^{\frac{\sqrt{x_{r,m}}t}{|N^r|} + \frac{m}{\sqrt{|N^r|}}} B_{ij \sindex}(s) ds \right) \right)\\
	& \quad - \sum_{\sindex=1}^{N^r_i} \left( \mu_{ij \sindex} \int_{\frac{m}{\sqrt{|N^r|}}}^{\frac{\sqrt{x_{r,m}}t}{|N^r|} + \frac{m}{\sqrt{|N^r|}}} B_{ij \sindex}(s) ds \right) \Bigg\Vert_L\\
	& \leq \frac{1}{\sqrt{x_{r,m}}} \left\Vert \sum_{\sindex=1}^{N^r_i} \left( S_{ij \sindex} \left( \mu_{ij \sindex} \frac{\sqrt{x_{r,m}}t}{|N^r|} \right) \right) - \sum_{\sindex=1}^{N^r_i} \left( \mu_{ij \sindex} \frac{\sqrt{x_{r,m}}t}{|N^r|} \right) \right\Vert_L\\
	& \overset{d}{=} \frac{1}{\sqrt{x_{r,m}}} \left\Vert S \left( \sum_{\sindex=1}^{N^r_i} \mu_{ij \sindex} \frac{\sqrt{x_{r,m}}t}{|N^r|} \right) - \left( \sum_{\sindex=1}^{N^r_i} \mu_{ij \sindex} \frac{\sqrt{x_{r,m}}t}{|N^r|} \right) \right\Vert_L\\
	& \leq \frac{1}{\sqrt{x_{r,m}}} \left\Vert S \left( \sum_{\sindex=1}^{N^r_i} q \frac{\sqrt{x_{r,m}}t}{|N^r|} \right) - \left( \sum_{\sindex=1}^{N^r_i} q \frac{\sqrt{x_{r,m}}t}{|N^r|} \right) \right\Vert_L\\
	& = \frac{1}{\sqrt{x_{r,m}}} \left\Vert S \left( q \sqrt{x_{r,m}}t \right) - \left( q \sqrt{x_{r,m}}t \right) \right\Vert_L \numberthis \label{sumofindividual}
	\end{align*}
	where $\overset{d}{=}$ means equal in distribution and $S$ is a standard Poisson process. The first $\overset{d}{=}$ is because every $S_{ij \sindex}$ is a stationary process. The second $\overset{d}{=}$ comes from the fact that the superposition of several Poisson processes is also a Poisson process with rate being the sum of the rates of the original processes. From \eqref{convergetomean}, for $r$ large enough, $ \eqref{sumofindividual} < \epsilon(r) $, thus
	\begin{align}
	\left\Vert \sum_{\sindex=1}^{N^r_i} \left( T^{r,m}_{ij \sindex}(t) - \frac{D^{r,m}_{ij \sindex}(t)}{\mu_{ij \sindex}} \right) \right\Vert_L < \epsilon(r). \label{HDconvergence13}
	\end{align}
	
	Therefore, from \eqref{HDconvergence13}, \eqref{HDconvergence3}, and \eqref{HDconvergence4}, 
	\begin{align*}
	\left\Vert  \left( \tilde{D}_{ij}(t) - \sum_{\sindex=1}^{N^r_i} \frac{\mu_{ij \sindex} t}{|N^r|} \right) \right\Vert_L 
	&\leq \left\Vert  \left( \tilde{D}_{ij}(t) - D^{r,m}_{ij}(t) \right) \right\Vert_L + \left\Vert  \sum_{\sindex=1}^{N^r_i} \left( D_{ij \sindex}^{r,m}(t) - \mu_{ij \sindex} T^{r,m}_{ij \sindex}(t) \right) \right\Vert_L\\
	& \quad + \left\Vert \sum_{\sindex=1}^{N^r_i} \left( \mu_{ij \sindex} T^{r,m}_{ij \sindex}(t) - \frac{\mu_{ij \sindex} x^*_{ij}}{|N^r|} t \right) \right\Vert_L\\
	&\leq \delta(1+2q_{ij}), \numberthis \label{HDconvergence7}
	\end{align*}
	and from the Law of Large Numbers
	\begin{equation}
	\left\Vert \sum_{\sindex = 1}^{N^r_i} \frac{\mu_{ij \sindex} x^*_{ij}}{|N^r|} t - \bar{\mu}_{ij}z_{ij}t \right\Vert_L \leq \delta, \label{HDconvergence8}
	\end{equation}
	for large enough $r$.
	Thus, from \eqref{HDconvergence7} and \eqref{HDconvergence8}, and reselecting $\delta$, one has
	\begin{equation}
	\left\Vert \tilde{D}_{ij}(t) - \bar{\mu}_{ij}z_{ij}t \right\Vert_L \leq \delta. \label{HDconvergence9}
	\end{equation}
	By combining \eqref{HDconvergence4}, \eqref{HDconvergence6}, \eqref{process1}, and \eqref{process2}, we get
	\begin{equation}
	\left\Vert \lambda_j t - \tilde{A}_{qj}(t) - \sum_{i \in \mathcal{I}(j)}\tilde{A}_{sij}(t) \right\Vert_L \leq (2+L)\delta, \label{HDconvergence10} \text{ and}
	\end{equation}
	\begin{equation}
	\left\Vert \tilde{Q}_j(t) - \tilde{Q}_j(0) -\tilde{A}_{aj}(t) + \sum_{i \in \mathcal{I}(j)}\tilde{C}_{ij}(t) \right\Vert_L \leq 4\delta.\label{HDconvergence11}
	\end{equation}
	By combining \eqref{HDconvergence4} with \eqref{HDconvergence9} and \eqref{process3}, we get
	\begin{equation}
	\left\Vert \tilde{Z}_{ij}(t) -\tilde{Z}_{ij}(0) - \tilde{A}_{sij}(t) - \tilde{C}_{ij}(t) + \bar{\mu}_{ij}z_{ij}t \right\Vert_L \leq 5\delta.\label{HDconvergence12}
	\end{equation}
	Equations \eqref{HDconvergence10}-\eqref{HDconvergence12} show that the hydrodynamic limits satisfy \eqref{HDfirst},\eqref{HDmodel2}, and \eqref{HDmodel4}. 
	
	That the hydrodynamic limits satisfy \eqref{HDmodel6} and \eqref{HDmodel7} is proved similarly to the fact that the fluid limits satisfy the fluid analogs of those equations. Hence, we only illustrate the proof of \eqref{HDmodel6}.
	
	Fix a hydrodynamic limit $\tilde{\mathbb{X}}$. By the definition of a hydrodynamic limit, there exists a sequence $(r_l, m_l, \omega_l)$, with $\omega_l \in \mathcal{K}$ for all $l \geq 0$, such that 
	\begin{equation}
	\mathbb{X}^{r_l, m_l} (\cdot, \omega_l) \to \tilde{\mathbb{X}}(\cdot) \text{  u.o.c. as } l \to \infty. \label{HDconvergence14}
	\end{equation}
	Fix $t>0$. If for any $j \in \mathcal{J}$, $\tilde{Q}_j(t) = 0$, \eqref{HDmodel6} holds trivially. Now we assume that $\tilde{Q}_j(t) > a$ for some $a>0$. By \eqref{HDconvergence14}, there exists an $l_0$ such that
	\begin{equation*}
	Q^{r_l,m_l}_j(t,\omega_l) > a/2  \text{  for all  } l>l_0.
	\end{equation*}
	This implies, by \eqref{process6}, that
	\begin{equation*}
	\sum_{\sindex=1}^{N^r_i} B^{r_l,m_l}_{ij \sindex}(t,\omega_l) = 0,
	\end{equation*}
	hence,
	\begin{equation}
	Q^{r_l,m_l}_j(t,\omega_l) \sum_{\sindex=1}^{N^r_i} B^{r_l,m_l}_{ij \sindex}(t,\omega_l) = 0. \label{HDconvergence15}
	\end{equation}
	Convergence in \eqref{HDconvergence14} implies that
	\begin{equation*}
	Q^{r_l,m_l}_j(t,\omega_l) \sum_{\sindex=1}^{N^r_i} B^{r_l,m_l}_{ij \sindex}(t,\omega_l) \to \tilde{Q}_j(t) \left( \sum_{\sindex=1}^{N^r_i}\tilde{B}_{ij \sindex}(t) \right) \text{  as  } l \to \infty.
	\end{equation*}
	This gives \eqref{HDmodel6} by \eqref{HDconvergence15}.
\end{proof}

Observe that, by \eqref{homogeneity} and the definitions of hydrodynamic and diffusion scalings,
\begin{equation}
|g(Q^{r,0}(0),Z^{r,0}(0))| \leq |g(\diffuscal{Q}^{r}(0),\diffuscal{Z}^{r}(0))|. \label{initialbridge}
\end{equation} 
If condition \eqref{initialcondition} holds, then \eqref{initialbridge} implies that $g(Q^{r,0}(0),Z^{r,0}(0)) \to 0$ in probability as $r \to \infty$. Therefore, we can choose $\epsilon(r)$ with $\epsilon(r) \to 0$ as $r \to \infty$ such that, for $\mathscr{L}^r = \mathscr{K}^r \cap \mathscr{G}^r$, where
\begin{equation}
\mathscr{G}^r = \{ |g(Q^{r,0}(0),Z^{r,0}(0))| \leq \epsilon(r)\},
\end{equation}
we have
\begin{equation}
\lim_{r \to \infty} P(\mathscr{L}^r) = 1. \label{goodpoints2}
\end{equation}
We set
\begin{equation}
E^r_g = \{ \mathbb{X}^{r,0}(\cdot, \omega), \omega \in \mathscr{L}^r \},
\end{equation}
and 
\begin{equation}
\mathscr{E}_g = \{ E^r_g, r \in \mathbb{N} \}.
\end{equation}

The following proposition is similar to \cite[Proposition 5.5]{Dai}. It connects Assumption \ref{SSCcondition} with Corollary \ref{convergencecorollary}, and shows an inequality similar to \eqref{HDSSCconditionequ} holds for the hydrodynamically scaled process $\mathbb{X}^{r,m}(\cdot)$.

\begin{proposition}
	Let $\{\mathbb{X}^r_{\pi}\}$ be a sequence of $\pi$-parallel random server systems processes. Assume that Assumption \ref{heavytraffic} and \ref{steadycondition} hold, $g$ satisfies Assumption \ref{continuity}, and the hydrodynamic model of the system satisfies Assumption \ref{SSCcondition}. Fix $\epsilon > 0, L>0, $ and $T>0$, and assume that $r$ is large. Then, for $\omega \in \mathscr{K}^r$,
	\begin{equation}
	g(Q^{r,m}(t), Z^{r,m}(t)) \leq H(t) + \epsilon \label{HDfuncconverge1}
	\end{equation}
	for all $t \in [0,L]$, and $m < \sqrt{|N^r|}T$, with $H(\cdot)$ as given in Assumption \ref{SSCcondition}.
	Furthermore, for $\omega \in \mathscr{L}^r$
	\begin{equation}
	||g(Q^{r,0}(t), Z^{r,0}(t))||_L \leq \epsilon.\label{HDfuncconverge2}
	\end{equation}
	If, in addition, condition \eqref{initialcondition} holds, then \eqref{goodpoints2} holds. 
	\label{HDSSCfunction}
\end{proposition}
\viva{The proof is the same as in \cite{Dai}. We state the proof modified for our notation in Appendix \ref{Appendix-HDSSCfunction}.}

\label{Section-hydro-limits}

\subsection{SSC in the diffusion limits}

In this section we change the scaling from hydrodynamic to diffusion to prove Theorem \ref{SSC1}. Once the scaling is changed, a few complications need to be dealt with regarding the change in the range of the time variable. In this process, all of the steps including lemmas and proofs remain the same as in \cite{Dai}. We include only final results in this section and leave the proofs in the Appendix. \par
We begin by changing the scaling. One can check by employing \eqref{diffusionscaling1}, \eqref{diffusionscaling2}, and \eqref{Qscaling}, \eqref{Bscaling} that
\begin{equation}
\begin{aligned}
Q^{r,m}(t) &= \sqrt{\frac{|N^r|}{x_{r,m}}} \diffuscal{Q}^r \left( \frac{\sqrt{x_{r,m}}t}{|N^r|} + \frac{m}{\sqrt{|N^r|}} \right) = \frac{1}{y_{r,m}} \diffuscal{Q}^r \left( \frac{1}{\sqrt{|N^r|}} (y_{r,m}t + m) \right)  \text{ and}\\
Z^{r,m}(t) &= \sqrt{\frac{|N^r|}{x_{r,m}}} \diffuscal{Z}^r \left( \frac{\sqrt{x_{r,m}}t}{|N^r|} + \frac{m}{\sqrt{|N^r|}} \right) = \frac{1}{y_{r,m}} \diffuscal{Z}^r \left( \frac{1}{\sqrt{|N^r|}} (y_{r,m}t + m) \right),
\end{aligned} \label{relationbetweenHDandD}
\end{equation}
where
\begin{equation}
y_{r,m} = \sqrt{\frac{x_{r,m}}{|N^r|}} = \left\vert \diffuscal{Z}^r\left( \frac{m}{\sqrt{|N^r|}} \right) \right\vert \vee 1. \label{diffusionfactor}
\end{equation} 

By changing the scaling in Proposition \ref{HDSSCfunction} as above, we can rephrase \eqref{HDfuncconverge1} and \eqref{HDfuncconverge2}. Also, the domain of the time scales will change and the domain $0 \leq t \leq L$ for the arguments on the left-hand side of \eqref{relationbetweenHDandD} will correspond to 
\begin{equation}
\frac{m}{\sqrt{|N^r|}} \leq t \leq \frac{1}{\sqrt{|N^r|}} (y_{r,m}L + m) \label{diffusiontrange}
\end{equation}
for the arguments on the right.\par
The next proposition uses the connection between hydrodynamic scaling $\mathbb{X}^{r,m}$ and diffusion scaling $\diffuscal{\mathbb{X}}^r$ in \eqref{relationbetweenHDandD} to translate the inequality \eqref{HDfuncconverge1} into another version for diffusively scaled processes.

\begin{proposition}
	Let $\{\mathbb{X}^r_{\pi}\}$ be a sequence of $\pi$-parallel random server system processes. Assume that Assumption \ref{heavytraffic} and Assumption \ref{steadycondition} hold, $g$ satisfies Assumption \ref{continuity}, and the hydrodynamic model of the system satisfies Assumption \ref{SSCcondition}. Fix $\epsilon > 0, L > 0$, and $T > 0$, and assume that $r$ is large. Then, for $\omega \in \mathscr{K}^r$ and for $H(\cdot)$ given as in Assumption \ref{SSCcondition},
	\begin{equation}
	g(\diffuscal{Q}^r(t), \diffuscal{Z}^r(t)) \leq y^c_{r,m} H \left( \frac{1}{y_{r,m}} (\sqrt{|N^r|}t - m) \right) + \epsilon y^c_{r,m} \label{diffusion1}
	\end{equation}
	for all $t \in [0,T]$ and $m$ satisfying \eqref{diffusiontrange}. Also
	\begin{equation}
	\left\Vert g(\diffuscal{Q}^r(t), \diffuscal{Z}^r(t)) \right\Vert_{Ly{r,0}/\sqrt{|N^r|}} \leq \epsilon y^c_{r,0} \label{diffusion2}
	\end{equation}
	for $\omega \in \mathscr{L}^r$. 
	\label{diffusionlemma1}
\end{proposition}

At this point, if we can show that $(\sqrt{|N^r|}t - m)/y_{r,m}$ is large enough, we can conclude that the results in Theorem \ref{SSC1} hold by using the convergence property of $H(t)$, as given in Assumption \ref{SSCcondition}. It will be shown that it is enough to have $\sqrt{|N^r|}t - m$ and $L$ large.

Since the value of $L$ is a matter of choice, we can take $L$ sufficiently large and redefine $\mathscr{K}^r$ with the reselected $L$. Let $H$ be given as in Assumption \ref{SSCcondition}. Then $H(t) \to 0$ as $t \to \infty$, thus for any fixed $\epsilon>0$, there exists $s^{\ast}(\epsilon) >1$ such that, for all  $t > s^{\ast}(\epsilon), H(t) < \epsilon$. We assume for the rest of the paper that
\begin{equation}
L \geq 6N s^{\ast}(\epsilon),\label{diffusion3}
\end{equation}
where $N$ is chosen as in \eqref{almostLP}.

To make $\sqrt{|N^r|}t - m$ large, for a fixed $t \in [0,T]$, we take the smallest $m$ that satisfies \eqref{diffusiontrange}, which we denote by $m_r(t)$. We need the following lemmas to show that $\sqrt{|N^r|}t - m_r(t)$ is large.

\begin{lemma}
	Let $\{\mathbb{X}^r_{\pi}\}$ be a sequence of $\pi$-parallel random server system processes. Assume that Assumption \ref{heavytraffic} and Assumption \ref{steadycondition} hold. For fixed $L>0$ and $T>0$, and large enough $r$
	\begin{equation}
	y_{r,m+1} \leq 3Ny_{r,m} \label{diffusion4}
	\end{equation}
	for $\omega \in \mathscr{K}^r$ and $m < \sqrt{|N^r|}T$, with the constant $N$ chosen as in \eqref{almostLP}. \label{diffusionlemma2}
\end{lemma}

Let $y_r(m_r(t)) = y_{r,m_r(t)}$.

\begin{lemma}
	Let $\{\mathbb{X}^r_{\pi}\}$ be a sequence of $\pi$-parallel random server systems processes. For fixed $L>0$ and $T>0$, and large enough $r$,
	\begin{equation}
	\sqrt{|N^r|}t - m_r(t) \geq Ly_r(m_r(t)) / 6N \label{diffusion5}
	\end{equation}
	for $\omega \in \mathscr{K}^r$ and $t \in (Ly_{r,0}/\sqrt{|N^r|},T]$, with the constant $N$ chosen as in \eqref{almostLP}. \label{diffusionlemma3}
\end{lemma}
Finally we can use the results from all of these lemmas and propositions to prove Theorem \ref{SSC1} now.
\begin{proof}[Proof of Theorem \ref{SSC1}]
	Assume that Assumption \ref{heavytraffic} and Assumption \ref{steadycondition} hold, $g$ satisfies Assumption \ref{continuity}, the hydrodynamic model satisfies Assumption \ref{SSCcondition}, and condition \eqref{initialcondition} holds.
	
	Fix $\vartheta >0$. By \eqref{goodpoints1} and \eqref{goodpoints2}, there exists $r_0 > 0$ such that
	\begin{equation}
	P(\mathscr{K}^r) \geq P(\mathscr{L}^r) > 1 - \vartheta/2 \label{diffusion6}
	\end{equation}
	for all $r>r_0$. Fix $\epsilon >0 $ and take $L \geq 6Ns^*(\epsilon)$. Then, by \eqref{diffusion1} and Lemma \ref{diffusionlemma3}, for $\omega \in \mathscr{K}^r, t \in (Ly_{r,0}/\sqrt{|N^r|},T]$, and $r$ large enough
	\begin{equation}
	g(\diffuscal{Q}^r(t), \diffuscal{Z}^r(t)) \leq 2\epsilon (y_{r}(m_r(t)))^c. \label{diffusion7}
	\end{equation}
	We have almost proven the result. Now recall by \eqref{diffusionfactor},
	\begin{equation}
	y_r(m_r(t)) = \left\vert \diffuscal{Z}^r \left( \frac{m_r(t)}{\sqrt{|N^r|}}\right) \right\vert \vee 1 \leq  ||\diffuscal{Z}^r(\cdot)||_T \vee 1. \label{diffusion8}
	\end{equation}
	If we include the initial condition \eqref{initialcondition}, we will have \eqref{diffusion2} as discussed above.  Thus combine this and \eqref{diffusion8}, for $t \in [0,Ly_{r,0}/\sqrt{|N^r|}]$ and $\omega \in \mathscr{L}^r$,
	\begin{equation}
	g(\diffuscal{Q}^r(t), \diffuscal{Z}^r(t)) \leq 2\epsilon(y_{r,0})^c \leq 2\epsilon (||\diffuscal{Z}^r(\cdot)||_T \vee 1)^c.\label{diffusion9}
	\end{equation}
	Combining \eqref{diffusion7}, \eqref{diffusion8}, and \eqref{diffusion9} gives
	\begin{equation}
	g(\diffuscal{Q}^r(t), \diffuscal{Z}^r(t)) \leq 2\epsilon(y_{r,0})^c \leq 2\epsilon (||\diffuscal{Z}^r(\cdot)||_T \vee 1)^c \label{diffusion10}
	\end{equation}
	for all $t \in [0,T]$ and $\omega \in \mathscr{L}^r$. Finally, by \eqref{diffusion6} and \eqref{diffusion10}, for large enough $r$,
	\begin{equation}
	P \left\lbrace \frac{||g(\diffuscal{Q}^r(\cdot), \diffuscal{Z}^r(\cdot))||_T}{( ||\diffuscal{Z}^r(\cdot)||_T \vee 1)^c} > 2\epsilon \right\rbrace < \vartheta.
	\end{equation}
	This clearly implies \eqref{multiplicativeSSC} because $\epsilon > 0$ and $\vartheta > 0$ are arbitrary.
\end{proof}

\label{Section-proof-SSC-in-diffusion-limit}

\section{State space collapse in many server queue}
\viva{Although \cite{Dai}'s work is developed for multi-calss queueing networks, such SSC results can also be applied to the many server queue analysis. \viva{To gain more insight into how the SSC is related to many server queues with random servers,}  in this section, we provide a new approach using the SSC technique to prove diffusion limit result in \cite[Theorem 2.1]{Atar}. We state the theorem in Chapter 3. Our analysis is under the assumption of the LISF policy. By proving result of \cite{Atar} using the state space collapse phenomenon, we provide a more generic method for showing diffusion limits for many server queues with random servers. Under other routing policies, one can also prove the diffusion limits using this approach by inventing their unique SSC functions. } \par
Before the proof, we need some preliminary mathematics. First we have the shifted and scaled system process (the total number of customers in the system) as
\begin{equation}
\diffuscal{X}^r(t)=\diffuscal{X}^r(0)+W^r(t)+F^{r}(t)+b^{\systemindex}t, \label{Atar-equation1}
\end{equation}
where
\begin{equation}
b^{\systemindex} =\frac{\lambda^{\systemindex}t-\systemindex \lambda}{\sqrt{\systemindex}}+\sum_{\sindex=1}^{N^r}\left(\frac{\mu_{\sindex}-\bar{\mu}}{\sqrt{\systemindex}} \right)+\bar{\mu}\diffuscal{N}, \label{Atar-equation2}
\end{equation}
and 
\begin{align*}
F^r(t)&=\sum_{i=1}^I F^{r,(i)}(t) =\sum_{i=1}^I \frac{\sum_{\sindex \in K_i}\mu_{\sindex} t-T^{r,(i)}(t)}{\sqrt{\systemindex}}\\
&=\sum_{\sindex=1}^{N^r}\frac{\mu_{\sindex}t-T^r_{\sindex}(t)}{\sqrt{\systemindex}} =\frac{\int_0^t \sum_{\sindex=1}^{N^r}\mu_{\sindex} I^r_{\sindex}(s)ds}{\sqrt{\systemindex}} = \int_0^t \sum_{\sindex=1}^{N^r}\mu_{\sindex} \diffuscal{I}^r_{\sindex}(s)ds, \numberthis \label{Atar-equation3} 
\end{align*}
and all of the notations are as defined in Section 2 of Chapter 3.\par
We want to show that the process $\diffuscal{X}^r(t)$ converges to a diffusion as the system size grows large. The convergence of $W^r(t),$ and $b^rt$ will be the same as in \cite{Atar}, and we also stated this in  3.2 of this thesis. We focus on showing that the item $F^r(t)$ weakly converges to $\gamma \int_{0}^{t} \xi(s)^-ds$. To achieve this, we use the SSC result from \cite{Dai}.\par
In order to show the convergence of $F(t)$, first we consider another sequence of systems - a sequence of inverted-V systems with $I$ server pools. Each pool contains $N^r_i$ identical servers with service rates $\mu^{(i)}, i=1, 2,\dots,I$. $\mu^{(i)}$ are defined in the proof of Theorem \ref{Atar-theorem}. When there are several servers idle at a moment, customers are routed to one of them according to the Longest Idle Server First(LISF) policy, which means that customers are routed to the server who has been idle for the longest time at the time of routing. Such a sequence of systems can be seen as a special case of \cite{Atar}'s model in the sense that the random service rates have a discrete distribution with support $\{ \mu^{(i)}, i = 1,2,\dots,I \}$, and servers are allocated according to different scenarios.\par
Assume the following conditions hold for the number of servers in the inverted-V systems:
\begin{align}
\lim_{r \to \infty} \frac{N}{r} &= 1, \label{No.serversMMn}\\
\lim_{r \to \infty} \frac{N_i}{N} &= \beta_i, \ \ \text{for } i=1,2, \dots, I, \label{No.servers.percent}
\end{align}
where $\sum_{i=1}^I \beta_i = 1.$\par
The heavy traffic condition holds throughout the discussion:
\begin{equation}
\lim_{r \to \infty} \diffufactor \left(\lambda^r - \sum_{i=1}^I \mu^{(i)} N^r_i \right) = \hat{\lambda}, \label{heavytraffic1}
\end{equation}
and $\hat{\lambda}$ should be the same as in \eqref{arrival-2moment-condition}.\par
The idea of our new approach to Atar's results is to show the SSC result of the special model, then let the difference between $\mu^{(i-1)}$ and $\mu^{(i)}$ tends to zero, so that the discrete model will `converge' to a continuous model, and thus the SSC result of the continuous model will also be proven.\par
The SSC result for such inverted-V systems is a simpler case of the results in \cite{Dai}, thus it is well established. For completeness, we repeat the notations here. Let $\mathbb{X}^r = (Q^r, Z^r)$ denote the system processes, where $Q^r(t)$ is number of customers waiting in the queue at time $t$, and $Z^r(t) = (Z^r_i(t), i=1,2,\dots, I)$, where $Z^r_i(t)$ is the number of busy servers in pool $i$ at time $t$. Since there are no abandonments in the queue, we will make a small modification on the hydrodynamic scaling factor $x_{r,m}$. Different to the $x_{r,m}$ defined in \cite{Dai}, let
\begin{equation}
x_{r,m} = \left\vert Z^r \left( \frac{m}{\sqrt{N^r}} \right) - N^r \right\vert^2 \vee |N^r|, \label{hydro-scaling-factor}
\end{equation}
where $N^r = (N^r_1, N^r_2, \dots, N^r_I)$ is an $I$ dimensional vector, and $|N^r|$ is the total number of servers in the systems. Notice that this $x_{r,m}$ does not contain $Q^r(t)$.\par
We can show that the hydrodynamic limits results in \cite{Dai} still hold with the new $x_{r,m}$ defined above. The proof is trivial thus we omit it here. We still use $\tilde{\mathbb{X}}$ to represent the hydrodynamic limit.\par
Finally the diffusive scaling is defined as 
\begin{align*}
\diffuscal{Q}^r(t) = \frac{Q^r(t)}{\sqrt{|N^{r}|}} \mbox{ and } \diffuscal{Z}^r_{i}(t) = \frac{Z^r_i(t) - N^r_i}{\sqrt{|N^r|}}, \mbox{ for } t \geq 0.
\end{align*}

We will need the following lemma in preparation for our SSC proof.
\begin{lemma}
	Let $\{\mathbb{X}^r\}$ be a sequence of inverted-V system processes as defined in the beginning of this section. Assume \eqref{heavytraffic1} and \eqref{norandomZconv} hold, and $\tilde{\mathbb{X}}$ be any hydrodynamic limit of such system, then 
	\begin{equation}
	\lim_{t \to \infty} \left( \frac{\tilde{Z}_i(t)}{\tilde{Z}_l(t)} - \frac{\beta_i}{\beta_l}\right) = 0 \ \forall i \neq l, as \ t \to \infty. \label{hydro.time.conv}
	\end{equation}
	\label{lemma.hydro.time.conv}
\end{lemma}
\vspace{-5mm}
\begin{proof}
	From \eqref{No.serversMMn} and \eqref{norandomZconv}, we know that $\frac{Z^r_i(\cdot)}{N^r} \to \beta_i$ u.o.c. in probability as $r \to \infty$. Recall that $\fluidscal{\mathbb{X}}$ is called a fluid limit of $\{\mathbb{X}^r\}$ if there exists an $\omega \in \mathscr{A}$ ($\mathscr{A}$ is taken from Appendix B. in \cite{Dai}) and a sequence $\{r_n\}$ with $r_n \to \infty$ as $n \to \infty$ such that $\frac{\mathbb{X}^{r_n}(\cdot, \omega)}{N^{r_n}}$ converges u.o.c. to $\fluidscal{\mathbb{X}}$ as $n \to \infty$. Therefore, $(\beta_1, \beta_2, \dots, \beta_I)$ is the fluid limit of $\left( \frac{Z^r_1(\cdot)}{N^r}, \frac{Z^r_2(\cdot)}{N^r}, \dots, \frac{Z^r_I(\cdot)}{N^r}  \right)$. Obviously $(\beta_1, \beta_2, \dots, \beta_I)$ is a time invariant state, thus it is also the steady state of the fluid limit.\par
	Now we try to show \eqref{hydro.time.conv1}. We know that $\tilde{Z}_i(t)$ is the hydrodynamic limit of $Z^{r,m}_i(t)$, i.e.\ if we fix $\epsilon > 0, L > 0$ and $T > 0$, and choose $r$ large enough, then for $\omega \in \mathscr{K}^r$ ($\mathscr{K}^r$ as defined in \cite[Corollary 5.1]{Dai}) and any $m < \sqrt{N^r}T$,
	\begin{equation}
	||Z^{r,m}_i(\cdot) - \tilde{Z}_l(\cdot)||_L \leq \epsilon \label{hydro.limit.conv}
	\end{equation}
	for some hydrodynamic limit $\tilde{Z}_i(\cdot), i=1,2,\dots, I$. Remember that
	\begin{equation}
	Z^{r,m}_i(t) = \frac{1}{\sqrt{x_{r,m}}} \left( Z^r_i \left( \frac{\sqrt{x_{r,m}} t }{N^r} + \frac{m}{\sqrt{N^r}} \right) -N^r_i \right).
	\end{equation}
	Thus in order to show $\frac{\tilde{Z}_i(t)}{\tilde{Z}_l(t)} \to \frac{\beta_i}{\beta_l}$ as $t \to \infty$, we can first show
	\begin{equation}
	\frac{Z^{r,m}_i(t)}{Z^{r,m}_l(t)} = \frac{Z^r_i \left( \frac{\sqrt{x_{r,m}}t}{N^r} + \frac{m}{\sqrt{N^r}} \right) - N^r_i}{Z^r_l \left( \frac{\sqrt{x_{r,m}}t}{N^r} + \frac{m}{\sqrt{N^r}} \right) - N^r_l} \to \frac{\beta_i}{\beta_l}, \text{ as } r \to \infty \text{ and } t \to \infty. \label{hydro.time.conv1}
	\end{equation}
	\vspace{-3mm}
	Since $\frac{N^r_i}{N^r_l} \to \frac{\beta_i}{\beta_l}$ as $r \to \infty$ and of course as $t \to \infty$, we only need to show 
	\begin{equation}
	\frac{Z^r_i \left( \frac{\sqrt{x_{r,m}}t}{N^r} + \frac{m}{\sqrt{N^r}} \right) }{Z^r_l \left( \frac{\sqrt{x_{r,m}}t}{N^r} + \frac{m}{\sqrt{N^r}} \right) } \to \frac{\beta_i}{\beta_l}, \text{ as } r \to \infty \text{ and } t \to \infty. \label{hydro.time.conv2}
	\end{equation}
	From the discussion about fluid limits in the beginning, 
	\begin{equation}
	\frac{Z^r_i(t)}{Z^r_l(t)} = \frac{Z^r_i(t)/N^r}{Z^r_l(t)/N^r} \to \frac{\beta_i}{\beta_l} \text{ u.o.c. in probability as } r \to \infty.
	\end{equation}
	This means that the limit of $\frac{Z^r_i(\cdot)}{Z^r_l(\cdot)}$ does not change with time $t$, so, as long as $r \to \infty$, we have
	\begin{equation}
	\frac{Z^r_i \left( \frac{\sqrt{x_{r,m}}t}{N^r} + \frac{m}{\sqrt{N^r}} \right) }{Z^r_l \left( \frac{\sqrt{x_{r,m}}t}{N^r} + \frac{m}{\sqrt{N^r}} \right) } - \frac{\beta_i}{\beta_l} \to 0, \forall \ t \geq 0.
	\end{equation}
	Then we have shown \eqref{hydro.time.conv2}, hence \eqref{hydro.time.conv1}. Therefore, as $r \to \infty$,
	\begin{align*}
	\frac{\tilde{Z}_i(t)}{\tilde{Z}_l(t)} = \left( \frac{\tilde{Z}_i(t)}{\tilde{Z}_l(t)} - \frac{Z^{r,m}_i(t)}{Z^{r,m}_l(t)}\right) + \frac{Z^{r,m}_i(t)}{Z^{r,m}_l(t)} \to 0 + \frac{\beta_i}{\beta_l} = \frac{\beta_i}{\beta_l}, \forall \ t \geq 0,
	\end{align*}
	and the lemma is proved as required.
\end{proof}
Now we have the SSC result for inverted-V systems.
\begin{lemma}[SSC for inverted-V]
	For the sequence of inverted-V systems mentioned above, there exists a continuous function $g: \mathbb{R}^{I+1} \to \mathbb{R}^+$, such that it satisfies Assumption 4.1 and 4.2 in \cite{Dai}, and if
	\begin{equation}
	g(\diffuscal{Q}^r(0), \diffuscal{Z}^r(0)) \to 0 \mbox{ in probability, as } r \to \infty, \label{SSC-inverted-V-initial}
	\end{equation}
	then, for each $T>0$,
	\begin{equation}
	||g(\diffuscal{Q}^r(t), \diffuscal{Z}^r(t))||_T \to 0 \mbox{ in probability}, \label{SSC-inverted-V}
	\end{equation}
	as $r \to \infty$. 
	\label{lemma-SSC for inverted-V}
\end{lemma}
\begin{proof}
	To prove the SSC result in the inverted-V systems, first we need to check that four assumptions from \cite{Dai} are satisfied. \par	
	By \eqref{heavytraffic1}, \cite[Assumption 3.1]{Dai} is satisfied and the static planning problem in that paper has a unique optimal solution with $x^*_i = 1$ for all $i = 1,2,\dots,I$. Since such systems are special cases of the systems in \cite{Atar}, the condition in Lemma 3.1(\Rnum{2}) from that paper should also be satisfied by these systems. Thus, for fixed $T>0$,
	\begin{equation}
	\left\Vert r^{-1} T^{(i)}(t) - \rho_i t \right\Vert_T \to 0 \ \text{ in probability as } r \to \infty,
	\end{equation}
	which can be rewritten as
	\begin{equation}
	\left\Vert \int_0^t \left( \frac{1}{r} \sum_{\sindex \in K_i} \mu_{\sindex} B_{\sindex}(s) -\beta_i \mu^{(i)} \right) ds \right\Vert_T \to 0 \text{ in probability as } r \to \infty.
	\end{equation}
	Recall that there is no randomness among servers in the same pool, i.e.\ $\mu_{\sindex}=\mu^{(i)}, \forall \sindex \in K_i$. Therefore the convergence above can be simplified to
	\begin{equation}
	\left\Vert \int_0^t \left( \frac{1}{r} Z^r_i(s) -\beta_i \right) ds \right\Vert_T \to 0 \text{ in probability as } r \to \infty. \label{norandomZconv}
	\end{equation}
	Since we do not consider abandonments in this model, the convergence of the (scaled) queue length does not have to be taken into account, thus equation \eqref{norandomZconv} implies that the second part of \cite[Assumption 3.2]{Dai} is satisfied, which is adequate for our analysis.\par
	Now we can define the SSC function $g: \mathbb{R}^{I+1} \to \mathbb{R}$:
	\begin{equation}
	g(q,z_1,z_2,\dots, z_I) = \sum_{i=1}^I z_i \mu^{(i)} - \sum_{i=1}^I z_i \gamma(I), \label{SSCfunction1}
	\end{equation}
	where $\gamma(I) = \frac{\sum_{l=1}^I \beta_l \left(\mu^{(l)}\right)^2 }{ \sum_{l=1}^I \beta_l \mu^{(l)} }$.\par
	It is easy to see that $g$ is continuous and $g(\alpha q,\alpha z_1,\alpha z_2, \alpha z_I) = \alpha g(q,z_1,z_2, \dots, z_I)$ for all $(q, z_1, z_2, \dots, z_I) \in \mathbb{R}^{I+1}$ and for all $0 \leq \alpha \leq 1$. Hence $g$ satisfies Assumption 4.1 in \cite{Dai}. Next, we show that $g$ satisfies Assumption 4.2 in \cite{Dai}.\par
	From \eqref{SSCfunction1}, we have
	\begin{align*}
	&g(\tilde{Q}(t), \tilde{Z}_1(t), \tilde{Z}_2(t), \dots, \tilde{Z}_I(t)) = \sum_{i=1}^I \tilde{Z}_i(t) \mu^{(i)} - \sum_{i=1}^I \tilde{Z}_i(t) \gamma(I)\\
	& = \sum_{i=1}^I \tilde{Z}_i(t) \left( \mu^{(i)} - \frac{\sum_{l=1}^I \beta_l \left(\mu^{(l)}\right)^2 }{\sum_{l=1}^I \beta_l \mu^{(l)} } \right) = \frac{1}{\sum_{l=1}^I \beta_l \mu^{(l)} } \sum_{i=1}^I \tilde{Z}_i(t) \left( \mu^{(i)} \sum_{l=1}^I \beta_l \mu^{(l)} - \sum_{l=1}^I \beta_l \left(\mu^{(l)}\right)^2  \right) \\
	& = \frac{1}{\sum_{i=1}^I \beta_i \mu^{(i)} } \sum_{i=1}^I \tilde{Z}_i(t) \left( \sum_{l=1}^I \left( \mu^{(i)} \beta_l \mu^{(l)} - \beta_l \left(\mu^{(l)}\right)^2 \right)  \right) \\
	& = \frac{1}{\sum_{l=1}^I \beta_l \mu^{(l)} } \sum_{i=1}^I \frac{1}{\mu^{(i)}} \tilde{Z}_i(t) \mu^{(i)} \left( \sum_{l=1}^I \left( \mu^{(i)} \beta_l \mu^{(l)} - \beta_l \left(\mu^{(l)}\right)^2 \right)  \right)\\
	& = \frac{1}{\sum_{l=1}^I \beta_l \mu^{(l)} } \sum_{i=1}^I \sum_{l=1}^I \frac{1}{\mu^{(i)}} \left( \tilde{Z}_i(t) \mu^{(i)}   \beta_l \mu^{(l)}\left(\mu^{(i)} -\mu^{(l)}\right) \right) \\
	& = \frac{1}{\sum_{l=1}^I \beta_l \mu^{(l)} } \sum_{i=1}^I \sum_{l > i}^I \frac{1}{\mu^{(i)}} \left( \tilde{Z}_i(t) \mu^{(i)} \beta_l \mu^{(l)} \left(\mu^{(i)} -\mu^{(l)}\right) + \tilde{Z}_l(t) \mu^{(l)}   \beta_i \mu^{(i)} \left(\mu^{(l)} -\mu^{(i)}\right) \right)\\
	& = \frac{1}{\sum_{l=1}^I \beta_l \mu^{(l)} } \sum_{i=1}^I \sum_{l > i}^I \frac{1}{\mu^{(i)}} \left( \mu^{(i)} \mu^{(l)} \left(\mu^{(i)} -\mu^{(l)}\right)  \left( \tilde{Z}_i(t)\beta_l - \tilde{Z}_l(t)\beta_i \right) \right). \numberthis \label{SSCfunction3} 
	\end{align*}
	By Lemma \ref{lemma.hydro.time.conv}, we know that \eqref{SSCfunction3} $\to 0$ as $t \to \infty$, thus \eqref{SSCfunction1} is a valid SSC function. So far all of the four assumptions in \cite{Dai} are satisfied. However, we can only derive the multiplicative SSC result of Theorem 4.1 under current conditions. What we really need is the strong SSC. Thus by Remark 4.2 in that paper, we need to check the so-called compact containment condition for $\diffuscal{Z}^r$ (We do not need to worry about this condition for $\diffuscal{Q}^r$ because the new hydrodynamic scaling factor \eqref{hydro-scaling-factor} in this case does not include $\diffuscal{Q}^r$).\par
	The compact containment condition for $\{\diffuscal{Z}^r(\cdot)\}$ is defined as
	\begin{equation}
	\lim_{K \to \infty}\limsup_{r \to \infty} \mathbb{P}(||\diffuscal{Z}^r(t)||_T > K) =0, \forall \ T>0. \label{compact-containment}
	\end{equation} 
	If $\{||\diffuscal{Z}^r(t)||_T\}$ is tight, then obviously \eqref{compact-containment} holds, hence we only need to show tightness of $\{||\diffuscal{Z}^r(t)||_T\}$. We will prove this by contradiction. \par
	Fix $T>0$. If $\{||\diffuscal{Z}^r(t)||_T\}$ is not tight, then, $\forall r$, $\exists \epsilon>0$ such that $\forall K>0$, we always have
	\begin{equation}
	\mathbb{P} \left( ||\diffuscal{Z}^r(t)||_T > K \right) = C_K \geq \epsilon >0,
	\end{equation} 
	i.e.\
	\begin{equation}
	\mathbb{P} \left( \max_{i}\left\Vert \frac{Z^r_i(t)-N^r_i}{\sqrt{|N^r|}} \right\Vert_T > K \right) = C_K \geq \epsilon >0. \label{tightness-of-Z-equation1}
	\end{equation} 
	However by \eqref{norandomZconv}, $\forall \delta>0$,
	\begin{equation}
	\lim_{r \to \infty} \mathbb{P} \left( \left\Vert \frac{1}{r} Z^r_i(t) - \beta_i \right\Vert_T > \delta \right) =0. 
	\end{equation}
	Using this and \eqref{No.serversMMn}, and \eqref{No.servers.percent},
	\begin{equation}
	\lim_{r \to \infty} \mathbb{P} \left( \left\Vert \frac{Z^r_i(t)}{N^r_i} - 1 \right\Vert_T > \delta \right) =0, 
	\end{equation}
	which contradicts \eqref{tightness-of-Z-equation1}. Therefore $\{||\diffuscal{Z}^r(t)||_T\}$ is indeed tight, so the condition \eqref{compact-containment} is satisfied. Thus by Theorem 4.1 and Remark 4.2 from \cite{Dai}, \eqref{SSC-inverted-V} holds, i.e.\ for any $T>0$
	\begin{equation}
	\left\Vert \sum_{i=1}^{I}  \diffuscal{Z}^r_i(t) \left( \mu^{(i)} -  \gamma(I) \right) \right\Vert_T \to 0 \mbox{ in probability as } r \to \infty. \label{SSC-inverted-V2}
	\end{equation}
\end{proof}
\vspace{-3mm}
Now that we have shown SSC for the inverted-V systems, we can prove Atar's result, i.e.\ Theorem \ref{Atar-theorem}.
\vspace{-3mm}
\begin{proof}[Proof of Theorem \ref{Atar-theorem}]
	All we need to show is 
	\begin{equation}
	\left\Vert F^r(t) - \gamma \int_{0}^{t} \diffuscal{I}^r(s) ds\right\Vert_T \to 0 \mbox{ as }  r \to \infty \mbox{ for any } T>0.
	\end{equation}	
	Since
	\begin{align*}
	F^r(t) - \gamma \int_{0}^{t} \diffuscal{I}^r(s) ds 
	&= F^r(t) - \gamma(I) \int_{0}^{t} \diffuscal{I}^r(s) ds + \gamma(I) \int_{0}^{t} \diffuscal{I}^r(s) ds - \gamma \int_{0}^{t} \diffuscal{I}^r(s). \numberthis \label{SSC-Atar-equation1}
	\end{align*}
	The first difference from this equation is
	\begin{align*}
	F^r(t) - \gamma(I) \int_{0}^{t} \diffuscal{I}^r(s) ds 
	&= \int_0^t \sum_{\sindex=1}^{N^r}\mu_{\sindex} \diffuscal{I}^r_{\sindex}(s)ds - \gamma(I) \int_{0}^{t} \diffuscal{I}^r(s) ds\\
	&= \int_{0}^{t} \sum_{i=1}^{I} \sum_{k \in K_i}\left( \mu_k \diffuscal{I}_k(s) - \gamma(I) \diffuscal{I}^r_k(s) \right) ds\\
	&= \int_{0}^{t} \sum_{i=1}^{I} \sum_{k \in K_i}\left( \mu_k -\mu^{(i)} + \mu^{(i)} - \gamma(I) \right) \diffuscal{I}^r_k(s) ds\\
	&= \int_{0}^{t} \sum_{i=1}^{I} \sum_{k \in K_i}\left( \mu_k -\mu^{(i)}\right)\diffuscal{I}^r_k(s) ds  + \int_{0}^{t}\sum_{i=1}^{I} \sum_{k \in K_i}\left( \mu^{(i)} - \gamma(I) \right) \diffuscal{I}^r_k(s) ds. \numberthis \label{SSC-Atar-equation2}
	\end{align*}
	The first integral in \eqref{SSC-Atar-equation2} converges to zero u.o.c as $r \to \infty$, which is a result of \cite[Lemma 3.1 (\Rnum{3})]{Atar} -  actually being $e_1$ in the lemma. The second integral is
	\begin{align*}
	&\int_{0}^{t}\sum_{i=1}^{I} \sum_{k \in K_i}\left( \mu^{(i)} - \gamma(I) \right) \diffuscal{I}^r_k(s) ds = \int_{0}^{t} \sum_{i=1}^{I} \left( \mu^{(i)} - \gamma(I) \right) \diffuscal{I}^{r,(i)}(s) ds\\
	&\quad = \int_{0}^{t} \sum_{i=1}^{I} \left( \mu^{(i)} - \gamma(I) \right) \diffufactor (N^r_i-Z^r_i(s)) ds = \int_{0}^{t} \sum_{i=1}^{I} \left( \mu^{(i)} - \gamma(I) \right) (-\diffuscal{Z}^r_i(s)) ds. \numberthis \label{SSC-Atar-equation3}
	\end{align*}
	By Lemma \ref{lemma-SSC for inverted-V} and equation \eqref{SSC-inverted-V2}, \eqref{SSC-Atar-equation3} converges to zero u.o.c as $r \to \infty$. Hence \eqref{SSC-Atar-equation2} converges to zero u.o.c as $r \to \infty$.\par
	For the latter half of equation \eqref{SSC-Atar-equation1}, since $\{\mu^{(i)}\}$ are chosen in a way such that $0 < \mu^{(i)} - \mu^{(i-1)} \leq \epsilon, \forall i=1,2,\dots,I$, and $\epsilon$ is arbitrary, we can let $\epsilon \to 0$, and then $\gamma(I) \to \gamma$ (remember $\gamma=\frac{\int x^2 dm}{\int x dm}$). Thus the second half of \eqref{SSC-Atar-equation1} also converges to zero u.o.c. Therefore, 
	\begin{equation}
	\left\Vert F^r(t) - \gamma \int_{0}^{t} \diffuscal{I}^r(s) ds  \right\Vert_T \to 0, \mbox{ as } r \to \infty,
	\end{equation} 
	and since $\diffuscal{I}^r(t) = \diffufactor(X^r(t) - N^r)^- = \diffuscal{X}^r(t)^-$, we have
	\begin{equation}
	\left\Vert F^r(t) - \gamma \int_{0}^{t} \diffuscal{X}^r(s)^- ds  \right\Vert_T \to 0, \mbox{ as } r \to \infty.
	\end{equation}
	Combining with \eqref{Atar-equation1}, we have $\diffuscal{X}^r(t)$ weakly converge to the solution of SDE
	\begin{equation}
	\xi(t) = \xi(0) + \sigma w(t) + \beta t + \gamma \int_0^t \xi(s)^- ds, t \geq 0,
	\end{equation}
	where all of the coefficients are as in Theorem \ref{Atar-theorem}.
\end{proof}

\chapter{Conclusion and Future Work}

In this work, we study queueing systems with heterogeneous servers and parameter uncertainty. In particular, we consider these systems under the \textit{Quality-Efficiency Driven} regime and heavy traffic condition. Our work is motivated by call centres where agents' service speeds may be influenced by their environment, and thus it is uncertain and unknown prior to the operation of the system. We assume service rates to be i.i.d.\ random variables. Their realisations are given at time zero and are kept fixed during the running of the system. Although it is modelled on call centres, our results are generic and can be applied to other domains such as healthcare, computer science, instant messaging service, among others.

We start by deriving diffusion limits for many server queues under the LISF policy with random service rates and abandonments. \cite{Atar} shows diffusion limits for such systems without abandonments. Then we obtain our results by extending his result to systems with abandonments, and use a martingale method to prove this. Unlike the constant drift of diffusion limits for identical server systems, our diffusion limits have a normal random drift which results from the randomness of the service rates. For systems without abandonments, the existence of such a random drift means that whether the diffusion has steady states depends on the value of the drift. From the perspective of the system, since the arrival rate is fixed, random service rates may cause the system to be unstable and the queue length may become unbounded. However, for systems with abandonments, even if it has random service rates it is always stable because of the abandonment process.

Then, we formulate an optimisation problem for the staffing for such systems. Staffing has always been of significance in call centre management since an unwise decision about it will cost companies immensely. For our systems, it is especially important because we need to take the randomness of servers into account. For systems without abandonments, we simplify the problem by assigning a fixed cost to unstable systems, and focus on the cost for stable systems. For stable systems, we consider both staffing costs and holding costs. The holding cost involves the expected queue length of steady state, thus we use the diffusion limits derived previously to establish a continuous approximation of this cost. In order to show the validity of this estimation, we prove the tightness of the steady state, by which we can show interchangeability of the limits. For systems with abandonments, the systems are always stable and we also simplify the problem by neglecting the holding cost. 

Finally, we show the state space collapse results for systems with random service rates. SSC is an important phenomenon because it reduces dimensions of the processes and significantly simplifies analysis. In particular, it is the central part in proving diffusion limits. Our work is based on the model developed by \cite{Dai}. They show a generic SSC result for queueing networks with multi-class customers and skilled based parallel server pools. We generalise their results to systems with random service rates inside each pool. To show the SSC in such systems, the challenging part is to show the departure processes are almost Lipschitz. We cannot show it directly as in \cite{Dai}, thus we use a coupling method by splitting Poisson processes with the maximum rate. We later use the SSC result from \cite{Dai} to show the diffusion limit in \cite{Atar} again. By using this approach, we identify that such a SSC result can also be applied in systems with random service rates and thus gain more insight into this phenomenon. And since \cite{Atar} uses a method particular to this model, this new approach also indicates that we can use a more general way to prove the limit theorem .

Queueing systems with random service rates exhibit many interesting properties, and we provide a few promising future research directions as follows. In our work, only the LISF routing policy is considered. One can consider other blind policies such as the random routing, where customers are routed to idle servers randomly, or the longest accumulated idle server first, where customers are routed to the idle server which has the longest accumulated idle time. Moreover, we can investigate the optimal routing and scheduling. In Section 3.3, we use a numerical method to analyse how the abandonment rate is influenced by the variation of service rates. We only consider the situation where random service rates are uniformly distributed. We can analyse other distributions as well as the theoretical proof of this result. In Section 3.4, we try to establish a fairness measure for general policies. We can continue this work and develop a general method for proving diffusion limits under different policies.

\appendix

\chapter{Proofs in Chapter 2}
\section{Proposition 3.1 of Atar}
\viva{
This proposition shows how Atar partitions servers such that the systems before and after pooling are equivalent in distribution. For simplicity, we will omit superscript $\systemindex$ from the notation of all random variables and stochastic processes throughout this appendix. The deterministic parameters that depend on $\systemindex$ will still keep the $\systemindex$ in their notation. 
\begin{proposition}(\cite[Proposition 3.1]{Atar})
Fix $r \in \mathbb{N}$. Let $(K_1,\dots,K_I)$ be a partition of $1,\dots, N$ that measurable on $\sigma\{ N, \{\mu_{\sindex}\} \}. Let \{S^{(1)}, \dots, S^{(I)}\}$ be independent standard Poisson processes. For each $i=1,\dots,I$ and for each nonempty subset $\Theta$ of $K_i$, let $\{ e(i,L,l), l \in \mathbb{N} \}$ be a sequence of i.i.d.\ random variables distributed uniformly on $\Theta$, independent across $i$ and $\Theta$. Assume also that the four random objects $(N, \{\mu_{\sindex}\}, X(0), \{B_{\sindex}(0)\}), a, \{S^{(i)}\}$ and $\{e(i,\Theta,l)\}$ are mutually independent. Define
\begin{equation}
D^{(i)}(t)=S^{(i)}(T^{(i)}(t)), i=1,\dots,I, \label{Atar-prop3.1-equ1}
\end{equation}
where
\begin{equation}
T^{(i)}(t) = \sum_{\sindex \in K_i} T_{\sindex}(t), i = 1,\dots, I, \label{Atar-prop3.1-equ2}
\end{equation}
and consider 
\begin{align}
D_{\sindex}(t) = \sum_{s \in (0,t]: \Delta D^{(i)}(s)=1} \mathds{1}_{ \{ e(i,\{ p \in K_i: B_p(s-)=1 \}, D^{(i)}(s)) =\sindex \} }, \quad \sindex \in K_i, i=1,\dots,I \label{Atar-prop3.1-equ3}
\end{align}
as a substitute for equation \eqref{Atar-Dk}. Then the process $\Sigma^{\prime}$, defined analogously to $\Sigma$, with \eqref{Atar-prop3.1-equ1}-\eqref{Atar-prop3.1-equ3} in place of \eqref{Atar-Dk}, is equal in law to $\Sigma$.
\end{proposition}
}
\label{Appendix-Atar-prop3.1}

\section{Proof of Theorem\ref{Atar-theorem} by Atar}

\begin{proof}[Proof of Theorem \ref{Atar-theorem}]
	We rephrase the method of \cite{Atar}'s proof here.
	
	From \eqref{Atar-system-equation0} and \eqref{Atar-diffusion-scaling}, we have
	\begin{align*}
	\diffuscal{X}(t) 
	&= \diffufactor(X(0) - N) + \diffufactor A(t) - \diffufactor \sum_{\sindex=1}^{N} D_{\sindex}(t)\\
	&= \diffuscal{X}(0) + \diffufactor (A(t) -\lambda^{\systemindex} t) + \diffufactor \lambda^{\systemindex} t - \diffufactor \sum_{\sindex=1}^{N} D_{\sindex}(t). \numberthis \label{Atar-system-equation1}
	\end{align*}
	Denote
	\begin{equation}
	\diffuscal{A}(t) = \diffufactor (A(t) -\lambda^{\systemindex} t). \label{Arrivalscaling}
	\end{equation}
	
	Proposition 3.1 in \cite{Atar} provides a way to consider departure processes $D_{\sindex}(t)$ aggregately instead of individually. In order to apply \cite[Proposition 3.1]{Atar} to $\sum_{\sindex=1}^{N} D_{\sindex}(t)$, first we need to define some notations. 
	
	Let $\epsilon$ be given and let $I \in \mathbb{N}$ and $\mu^{(i)} \in \mathbb{R}_+$, $i=1,\dots,I$ satisfying the following conditions:
	\begin{itemize}
		\item $\mu^{(1)}=0,\mu^{(I)} \geq 1$,
		\item $0<\mu^{(i)}-\mu^{(i-1)} \leq \epsilon, i=1,\dots,I$,
		\item $\int_{[\mu^{(I)},\infty)} x^2 dm \leq \epsilon$,
		\item for $i=2,\dots,I, \ \mu^{(i)}$ is a continuity point of $x \mapsto m([0,x]) \equiv P(\mu_{\sindex}^2 \leq x)$.
		\end{itemize}
		Set $\mu^{(I+1)} = \infty$, and
		\begin{equation}
		K_i = \{ \sindex \in \{ 1,2,\dots,N^{\systemindex}\}: \mu_{\sindex} \in [\mu^{(i)}, \mu^{(i+1)}) \}, i=1,2,\dots,I.
		\end{equation}
		Let $\{S^{(1)},\dots,S^{(q)}\}$ be independent standard Poisson processes. Now we are ready to use Proposition 3.1. 
		\begin{equation}
		\sum_{\sindex=1}^{N} D_{\sindex}(t) = \sum_{i=1}^I D^{(i)}(t) = \sum_{i=1}^I S^{(i)}(T^{(i)}(t)), \label{Atardeparture}
		\end{equation}
		where
		\begin{equation}
		T^{(i)}(t) = \sum_{\sindex \in K_i} \mu_{\sindex} T_{\sindex}(t). 
		\end{equation}
		Hence, by \eqref{Arrivalscaling} and \eqref{Atardeparture}, \eqref{Atar-system-equation1} can be rewritten as 
		\begin{align*}
		\diffuscal{X}(t) 
		&= \diffuscal{X}(0) + \diffuscal{A}(t) + \diffufactor \lambda^{\systemindex} t - \diffufactor \sum_{i=1}^I D^{(i)}(t) \\
		&= \diffuscal{X}(0) + \diffuscal{A}(t) + \diffufactor \lambda^{\systemindex} t - \diffufactor \sum_{i=1}^I \left( S^{(i)}(T^{(i)}(t)) - T^{(i)}(t) + T^{(i)}(t)\right). \numberthis \label{Atar-proof-1} 
		\end{align*}
		Define $W^{(i)}(t)=S^{(i)}\left(T^{(i)}\right)-T^{(i)}(t)$. Then \eqref{Atar-proof-1} becomes
		\begin{align*}
		\diffuscal{X}(t) &=\diffuscal{X}(0)+\diffuscal{A}(t)+\frac{\lambda^{\systemindex}t}{\sqrt{\systemindex}}-\sum_{i=1}^I \left(W^{(i)}(t)-\frac{T^{(i)}(t)}{\sqrt{\systemindex}}\right)\\
		&=\diffuscal{X}_0+\diffuscal{A}(t)+\frac{\lambda^{\systemindex}t}{\sqrt{\systemindex}}-\sum_{i=1}^I \left(W^{(i)}(t)-\frac{T^{(i)}(t)-\sum_{\sindex \in K_i}\mu_{\sindex} t}{\sqrt{\systemindex}}\right)+\sum_{k=1}^{N^{\systemindex}}\frac{\mu_{\sindex}}{\sqrt{\systemindex}}t. \numberthis \label{Atar-proof-2}
		\end{align*}
		Let $F^{(i)}(t)=-\frac{T^{(i)}(t)-\sum_{\sindex \in K_i}\mu_{\sindex}t}{\sqrt{r}}$. Then \eqref{Atar-proof-1} is
		\begin{align}
		\diffuscal{X}(t)= \diffuscal{X}(0)+\diffuscal{A}(t)+\frac{\lambda^{\systemindex}t}{\sqrt{\systemindex}}-\sum_{i=1}^I \left(W^{(i)}(t)+F^{(i)}(t)\right)+\sum_{\sindex = 1}^{N}\frac{\mu_{\sindex}}{\sqrt{\systemindex}}t. \label{Atar-proof-3}
		\end{align}
		To see this more clearly, let $W(t)= \diffuscal A(t)-\sum_{i=1}^I W^{(i)}(t)$ so \eqref{Atar-proof-3} is
		\begin{align*}
		\diffuscal{X}(t)
		&=\diffuscal{X}(0)+W(t)+\frac{\lambda^{\systemindex}t}{\sqrt{\systemindex}}+\sum_{i=1}^I F^{(i)}(t)+\sum_{\sindex = 1}^{N}\frac{\mu_{\sindex}}{\sqrt{\systemindex}}t\\
		&=\diffuscal{X}(0)+W(t)+\frac{\lambda^{\systemindex}t}{\sqrt{\systemindex}}+\sum_{i=1}^I F^{(i)}(t)+\sum_{\sindex=1}^{N}\frac{\mu_{\sindex}}{\sqrt{\systemindex}}t. \numberthis \label{Atar-proof-4}
		\end{align*}
		Recall that $\lambda=\bar{\mu}$, thus
		\begin{align*}
		\diffuscal{X}(t)
		&=\diffuscal{X}(0)+W(t)+\frac{\lambda^{\systemindex}t-N\lambda t-\systemindex \lambda t+ \systemindex \lambda t}{\sqrt{\systemindex}}+\sum_{i=1}^I F^{(i)}(t)+\sum_{\sindex = 1}^{N}\frac{\mu_{\sindex}-\bar{\mu}}{\sqrt{\systemindex}}t\\
		&=\diffuscal{X}(0)+W(t)+\frac{\lambda^{\systemindex}t - \systemindex\lambda t}{\sqrt{\systemindex}}+\sum_{i=1}^I F^{(i)}(t)+\sum_{\sindex=1}^{N}\left(\frac{\mu_{\sindex}-\bar{\mu}}{\sqrt{\systemindex}}t \right) + \bar{\mu}\frac{N-\systemindex}{\sqrt{\systemindex}}t\\
		&=\diffuscal{X}(0)+W(t)+\frac{\lambda^{\systemindex}t-\systemindex \lambda t}{\sqrt{\systemindex}}+\sum_{i=1}^I F^{(i)}(t)+\sum_{\sindex=1}^{N}\left(\frac{\mu_{\sindex}-\mu}{\sqrt{\systemindex}}t \right)+\bar{\mu}\diffuscal{N}t.
		\end{align*}
		Finally define $b^{\systemindex}=\frac{\lambda^{\systemindex}t-\systemindex \lambda}{\sqrt{\systemindex}}+\sum_{\sindex=1}^{N}\left(\frac{\mu_{\sindex}-\bar{\mu}}{\sqrt{\systemindex}} \right)+\bar{\mu}\diffuscal{N}$. Thus
		\begin{align}
		\diffuscal{X}(t) &=\diffuscal{X}(0)+W(t)+\sum_{i=1}^I F^{(i)}(t)+b^{\systemindex}t. \label{Atar-proof-5}
		\end{align}
		$\sum_{i=1}^I F^{(i)}(t)$ is a troublesome component which needs our particular attention. For this reason, we define
		\begin{align*}
		F(t)&=\sum_{i=1}^I F^{(i)}(t)=\sum_{i=1}^I\frac{\sum_{\sindex \in K_i}\mu_{\sindex} t-T^{(i)}(t)}{\sqrt{\systemindex}}\\
		&=\sum_{\sindex=1}^{N}\frac{\mu_{\sindex}t-T_{\sindex}(t)}{\sqrt{\systemindex}}=\frac{\int_0^t \sum_{\sindex=1}^{N}\mu_{\sindex} I_{\sindex}(s)ds}{\sqrt{\systemindex}}.
		\end{align*}
		Now define $I^{(i)}(t)$ to be the number of idle servers in pool $i$. We can write 
		\[
		I^{(i)}(t-)=D^{(i)}(t-)-D^{(i)}(H(t)-)+e_0^{(i)}(t),
		\]
		where $H(t)$ is defined in (2.16) in \cite{Atar} to be the time at which the longest idle server at $t$ became idle. So, the difference of $D_{\sindex}(t-) - D_{\sindex}(H(t)-)$ can be 1 or 0 based on whether server $\sindex$ is idle at time $t-$. Also, $e_0$ accounts for the fact that servers can be idle from 0 to $t$. This uses the fact that we are routing customers to the longest idle server first. If we use another policy, this will be the key point in proving SSC.
		
		The inequality below follows because some of the servers idle at 0 can start serving: 
		$$\sum_{i=1}^I \hat{e}_0^i(t) = \frac{\sum_{i=1}^I e_0^i(t)}{\sqrt{\systemindex}} \leq \frac{I(0)I_{\{H(t)=0\}}}{\sqrt{\systemindex}}=\diffuscal{X}(0)^-(t)I_{\{H(t)=0\}}.$$
		
		Now, we will do the following manipulation:
		\begin{align*}
		\diffuscal{I}^{(i)}(t-)&=\frac{I^{(i)}(t-)}{\sqrt{\systemindex}}\\&=\frac{D^{(i)}(t-)-D^{(i)}(H(t)-)+e_0^{(i)}(t)}{\sqrt{\systemindex}}\\
		&=\frac{S^{(i)}(T^{(i)}(t-))-S^{(i)}(T^{(i)}(H(t)-)) + e_0^{(i)}(t)}{\sqrt{\systemindex}}\\
		&=\frac{S^{(i)}(T^{(i)}(t-)) - T^{(i)}(t-) + T^{(i)}(t-) - S^{(i)}(T^{(i)}(H(t)-))-T^{(i)}(H(t)-) + T^{(i)}(H(t)-)}{\sqrt{\systemindex}}\\
		& \quad +\hat{e}_0^{(i)}(t)\\
		&=W^{(i)}(t-)-W^{(i)}(H(t)-)+\frac{T^{(i)}(t-)-T^{(i)}(H(t)-)}{\sqrt{\systemindex}}+\hat{e}_0^{(i)}(t).
		\end{align*}
		Recall $h(t)=t-H(t)$. Hence 
		\begin{align}
		\diffuscal{I}^{(i)}(t-)&=W^{(i)}(t-)-W^{(i)}(H(t)-)+\frac{T^{(i)}(t-)-T^{(i)}(H(t)-)}{\sqrt{\systemindex}}+\frac{\mu^{(i)}N^{(i)}h(t)}{\sqrt{\systemindex}}-\frac{\mu^{(i)} N^{(i)} h(t)}{\sqrt{\systemindex}}+\hat{e}_0^{(i)}(t). \label{Atar-proof-6}
		\end{align}
		To manipulate this even further, we need to realize
		\begin{align*}
		T^{(i)}(t-)-T^{(i)}(H(t)-)&=\sum_{\sindex \in K_i} \mu_{\sindex} \int_{H(t)}^t B(s)ds,\\
		N^{(i)}h(t)&=\int_{H(t)}^t(I(s)+B(s))ds.
		\end{align*}
		Now, substituting these into \eqref{Atar-proof-6},
		\begin{align*}
		\diffuscal{I}^{(i)}(t-)
		&=W^{(i)}(t-)-W^{(i)}(H(t)-)+\sum_{\sindex \in K_i}\left(\left(\mu_{\sindex} - \mu^{(i)}\right)\int_{H(t)}^t \diffuscal{B}(s)ds-\mu^{(i)}\int_{H(t)}^t \diffuscal{I}(s)ds \right)\\
		&\quad+\frac{\mu^{(i)}N^{(i)}h(t)}{\sqrt{\systemindex}}+\hat{e}_0^{(i)}(t),
		\end{align*}
		and aggregating terms gives us
		\begin{align}
		\diffuscal{I}^{(i)}(t-)&=E^{(i)}(t)+\frac{\mu^{(i)}N^{(i)}h(t)}{\sqrt{\systemindex}}. \label{Atar-proof-7}
		\end{align}
		Now, we note the following relation
		\[
		h(t)=\sqrt{\systemindex}\frac{\sum_{i=1}^I \diffuscal{I}^{(i)}(t-)-\sum_{i=1}^I E^{(i)}(t)}{\sum_{i=1}^I \mu^{i}N^{(i)}},
		\]
		then use this to write
		\begin{equation}
		\label{eq:Ihat}
		\diffuscal{I}^{(i)}(t)=E^{(i)}(t)+\frac{\mu^{(i)}N^{(i)}}{\sum_{i=1}^I \mu^{(i)}N^{(i)}} \left(\sum_{i=1}^I \diffuscal{I}^{(i)}(t-)-\sum_{i=1}^I E^{(i)}(t)\right).
		\end{equation}
		Now, let us try to rewrite $F(t)$:
		\begin{align*}
		F(t)&=\displaystyle \int_0^t\sum_{\sindex=1}^{N}\mu_{\sindex} \diffuscal{I}_{\sindex}(s)ds\\
		&=\displaystyle \int_0^t \left(\sum_{\sindex=1}^{N}\mu_{\sindex} \diffuscal{I}_{\sindex}(s)-\sum_{i=1}^I \mu^{(i)} \diffuscal{I}^{(i)}(s)+\sum_{i=1}^I \mu^{(i)} \diffuscal{I}^{(i)}(s) \right)ds\\
		&=\displaystyle \int_0^t \left(\sum_{i=1}^I \sum_{\sindex \in K_i} (\mu_{\sindex}-\mu^{(i)}) \diffuscal{I}_{\sindex}(s)+\sum_{i=1}^I \mu^{(i)}\diffuscal{I}^{(i)}(s)\right)ds. \numberthis \label{Atar-proof-8}
		\end{align*}
		We define the four $e_i$s of \cite[Lemma 3.1 (\Rnum{3})]{Atar} as follows.\\
		Let $e_1(t)=\sum_{i=1}^I \sum_{\sindex \in K^i}(\mu_{\sindex}-\mu^{(i)}) \int_0^t \diffuscal{I}_{\sindex}(s)ds$. Then \eqref{Atar-proof-8} is equal to
		\begin{align*}
		F(t)
		&=e_1(t)+\sum_{i=1}^I\mu^{(i)} \int_0^t\diffuscal{I}^{(i)}(s)ds.
		\end{align*}
		Use \eqref{eq:Ihat}, we have
		\begin{align*}
		F(t)
		&=e_1(t)+\sum_{i=1}^I\mu^{(i)}\int_0^t\left(E^{(i)}(s)+\frac{\mu^{(i)}N^{(i)}}{\sum_{i=1}^I \mu^{(i)}N^{(i)}}\left(\sum_{i=1}^I \diffuscal{I}^{(i)}(s-)-\sum_{i=1}^I E^{(i)}(s)\right)\right)ds.
		\end{align*}
		Define $e_2(t)=\sum_{i=1}^I\mu^{(i)}\int_0^t E^{(i)}(s)ds$. Then
		\begin{align*}
		F(t)
		&=e_1(t)+e_2(t)+\sum_{i=1}^I\int_0^t\left(\frac{(\mu^{(i)})^2 N^{(i)}}{\sum_{i=1}^I \mu^{(i)}N^{(i)}} \left(\sum_{i=1}^I \diffuscal{I}^{(i)}(s-)-\sum_{i=1}^I E^{(i)}(s)\right)\right)ds.
		\end{align*}
		Further define $e_3(t)=\frac{\sum_{i=1}^I(\mu^{(i)})^2N^{(i)}}{\sum_{i=1}^I \mu^{(i)}N^{(i)}}\sum_{i=1}^I\int_0^t E^{(i)}(s)ds$, so
		\begin{align*}
		F(t)
		&=e_1(t)+e_2(t)+e_3(t)+\sum_{i=1}^I\int_0^t\left(\frac{(\mu^{(i)})^2N^{(i)}}{\sum_{i=1}^I \mu^{i}N^{(i)}}\left(\sum_{i=1}^I \diffuscal{I}^{(i)}(s-)\right)\right)ds.
		\end{align*}
		Remember that $\diffuscal{I}=\sum \diffuscal{I}^{(i)}$ and add and subtract $\gamma\int_0^t \diffuscal{I}(s)ds$ to get
		\begin{align*}
		F(t)&=e_1(t)+e_2(t)+e_3(t)+e_4(t)+\gamma\int_0^t\diffuscal{I}(s)ds.
		\end{align*}
		Notice that $e_4(t)=\left(\frac{\sum_{i=1}^I (\mu^{(i)})^2N^{(i)}}{\sum_{i=1}^I \mu^{(i)}N^{(i)}}-\gamma\right)\int_0^t\diffuscal{I}(s)ds$, where $\gamma=E(\mu_1^2)/E(\mu_1)$.
		
		Summarising all of these gives
		\[
		\diffuscal{X}(t)=\diffuscal{X}(0)+W(t)+bt+\gamma\int_0^t\diffuscal{X}(s)^-ds+\sum_{i=1}^4 e_i(t).
		\]
		
		First, we use Gronwall's inequality:
		\begin{equation}
		||\diffuscal{X}-\xi||_{t}\leq (|\diffuscal{X}(0)-\xi(0)|+|b-\beta|+||W-\sigma w||_{t}+||e||_{t})\exp(\gamma t). \label{Gronwall's inequality}
		\end{equation}

Now we need to analyse these terms separately:
\begin{enumerate}
	\item $\diffuscal{X}(0)$ is simply the usual Central Limit Theorem applied to the initial random variables. 
			
	\item $W(t)=\hat{A}(t)-\sum_{i=1}^IW^{(i)}(t)$
\begin{enumerate}
\item $\hat{A}(t)\Rightarrow B_1(t)$, where $B_1(t)$ is a Brownian Motion with 0 drift and diffusion coefficient $\sqrt{\lambda}C_{\check{U}}$ using the Functional Central Limit Theorem for renewal processes.
				
\item $W^{(i)}(t)=S^{(i)}\left(T^{(i)}(t)\right)-T^{(i)}(t)$. By \cite[Lemma 3.1(\Rnum{2})]{Atar}, we have $r^{-1} T^{(i)} \to \rho_i t$, where $\rho_i=\int_{\mu^{(i)}}^{\mu^{(i+1)}} x dm$. Again using the FCLT for renewal processes, $S^{(i)}(t)-t$ converges to a standard Brownian motion. Then replace $T^{(i)}(t)$ with $t$. By the random change of time, we show that $W^{(i)}(t)$ weakly converges to Brownian motion $B_2(t)$ with zero mean and diffusion coefficient $\sqrt{\rho_i}$.
				
\end{enumerate}
Summarising (a) and (b), we show that $W(t)$ converges weakly to $\sigma w$, where $w$ is a standard Brownian motion and $\sigma^2 = \lambda C^2_{\check{U}} + \mu$.
\item
For $j=1,2,3,4$, $||e_j||_{\theta} \to 0$ in probability as $\systemindex \to \infty$. This is shown in \cite[Lemma 3.1 (\Rnum{3})]{Atar}.
\end{enumerate}
				
By steps 1, 2 and 3, the right hand side of equation \eqref{Gronwall's inequality} converges to zero uniformly on a compact set, thus $\diffuscal{X}(t) \Rightarrow \xi(t)$ as required.
				
\end{proof}

\label{Appendix-Atar-proof}

\chapter{Proofs in Chapter 4}
\section{Proof of fluid limits in Theorem \ref{fluidlimit}}
\begin{proof}[Proof of Theorem \ref{fluidlimit}]
	\vspace{-2mm}
	The proof is similar to that in Theorem B.1 in \cite{Dai}. Specifically, the proof of the precompactness of $\fluidscal{A}, \fluidscal{A}_q, \fluidscal{A}_s, \fluidscal{Q}$, and $ \fluidscal{C}$ is the same as in Theorem B.1 in \cite{Dai}, and equations \eqref{fluidequ1}, \eqref{fluidequ2}, \eqref{fluidequ7}, \eqref{fluidequ8}, and \eqref{fluidequ9} are also proved in the same paper, so in the rest of the proof we use these results directly and focus on the precompactness of $\fluidscal{Z}, \fluidscal{T}, \fluidscal{I}$ and equations \eqref{fluidequ3}, \eqref{fluidequ4}, and \eqref{fluidequ6} here.\par
	First we prove the precompactness. Assume \eqref{No.servers} and \eqref{arriallimits} hold. With a slight abuse of notation, consider a sequence of numbers that is denoted as $\{r\}$. We show that $\{ \fluidscal{\mathbb{X}}^r(\cdot, \omega) \}$ has a convergent subsequence, for all $\omega \in \mathscr{A}$. Fix $\omega$ in the rest of the proof. One can observe that
	\begin{equation*}
	\left\vert \frac{T^r(t_2, \omega)}{N^r} - \frac{T^r(t_1, \omega)}{N^r} \right\vert \leq |t_2 - t_1|,
	\end{equation*}
	for all $0 \leq t_1 \leq t_2$. Hence $\{\fluidscal{T}^r(\cdot, \omega)\}$ is uniformly bounded and uniformly continuous, which means, by \cite[Theorem 12.3]{Billingsley}, that there exists a subsequence $\{r_l\}$ such that $\fluidscal{T}^{r_l}(\cdot, \omega)$ converges u.o.c. to some continuous function $\fluidscal{T}$.\par
	We define the fluid scaled total idle process for the $i$th server pool by
	\begin{equation}
	\fluidscal{I}^r_i(t) = \frac{N^r_i}{|N^r|}t - \sum_{j \in \mathcal{J}(i)}\fluidscal{T}^r_{ij}(t).
	\end{equation} 
	Then obviously $\fluidscal{I}^{r_l}_i(\cdot)$ is precompact.\par
	For departure processes, we need to treat each server individually. Since 
	$$\fluidscal{D}^{r_l}_{ij\sindex}(t) = \frac{1}{|N^r|} S_{ij\sindex}(\mu_{ij\sindex}|N^r| \fluidscal{T}^{r_l}_{ij\sindex}(t)),$$ we have
	\begin{align*}
	\fluidscal{D}^{r_l}_{ij }(t) 
	&= \sum_{\sindex = 1}^{N^r_i} \fluidscal{D}^{r_l}_{ij\sindex}(t) = \frac{1}{|N^r|} \sum_{\sindex = 1}^{N^r_i} S_{ij\sindex}(|N^r| \mu_{ij\sindex} \fluidscal{T}^{r_l}_{ij\sindex}(t)) \overset{d}{=} \frac{1}{|N^r|} S\left( |N^r| \sum_{\sindex = 1}^{N^r_i} \mu_{ij\sindex}  \fluidscal{T}^{r_l}_{ij\sindex}(t)\right)
	\end{align*}
	where $S$ is a standard Poisson process. $\overset{d}{=}$ means equal in distribution and it is a basic statement of the superposition of Poisson processes.\par
	To see the convergence of $\fluidscal{D}^{r_l}_{ij}(t)$, fix $L > 0$, and consider
	\begin{align*}
	&\left\vert \sum_{\sindex = 1}^{N^r_i} \mu_{ij\sindex} \fluidscal{T}^{r_l}_{ij\sindex}(t) -  \bar{\mu}_{ij} \fluidscal{T}_{ij}(t) \right\vert\\
	& = \left\vert \sum_{\sindex = 1}^{N^r_i} \left( \mu_{ij\sindex} \fluidscal{T}^{r_l}_{ij\sindex}(t) - \bar{\mu}_{ij}\fluidscal{T}^{r_l}_{ij\sindex}(t) \right) +  \sum_{\sindex = 1}^{N^r_i} \left( \bar{\mu}_{ij} \fluidscal{T}^{r_l}_{ij\sindex}(t) - \bar{\mu}_{ij} \fluidscal{T}_{ij}(t)\frac{1}{N^r_i} \right) \right\vert\\
	&= \left\vert \frac{1}{|N^r|} \sum_{\sindex = 1}^{N^r_i} \left( \mu_{ij\sindex} - \bar{\mu}_{ij} \right) T^{r_l}_{ij\sindex}(t) + \bar{\mu}_{ij} \sum_{\sindex = 1}^{N^r_i}  \left( \fluidscal{T}^{r_l}_{ij\sindex}(t) - \fluidscal{T}_{ij}(t) \frac{1}{N^r_i} \right) \right\vert\\
	& \leq \left\vert \frac{N^r_i}{|N^r|} \frac{1}{N^r_i} \sum_{\sindex = 1}^{N^r_i} \left( \mu_{ij\sindex} - \bar{\mu}_{ij} \right) t \right\vert + \left\vert \bar{\mu}_{ij}  \left( \fluidscal{T}^{r_l}_{ij}(t) - \fluidscal{T}_{ij}(t) \right) \right\vert. \numberthis \label{depaconv1}
	\end{align*}
	By the Law of Large Numbers, the first term in \eqref{depaconv1} converges u.o.c. to $0$, and the second term also converges u.o.c. to $0$ as proved already. This means $\sum_{\sindex = 1}^{N^r_i} \mu_{ij\sindex} \fluidscal{T}^{r_l}_{ij\sindex}(\cdot)$ converges u.o.c. to $ \bar{\mu}_{ij} \fluidscal{T}_{ij}(\cdot)$. Then using \cite[Lemma 11]{ata2005heavy} and the Functional Strong Law of Large Numbers we can get
	\begin{equation}
	\fluidscal{D}^{r_l}_{ij}(\cdot) \text{ converges u.o.c. to } \fluidscal{D}_{ij}(\cdot), \label{fluidDconv}
	\end{equation}
	where $\fluidscal{D}_{ij}(t) = \bar{\mu}_{ij} \fluidscal{T}(t)$. By the precompactness of $\fluidscal{A}^{r_l}_s, \fluidscal{C}^{r_l}$ and \eqref{fluidDconv}, and the process equation \eqref{originprocess3}, $\fluidscal{Z}^{r_l}(\cdot, \omega)$ is precompact.\par
	Next we show that every fluid limit satisfies \eqref{fluidequ3}-\eqref{fluidequ6}. \eqref{fluidequ4}-\eqref{fluidequ6} are trivial. For \eqref{fluidequ4}, let $\fluidscal{\mathbb{X}}$ be a fluid limit and for notational convenience assume that
	\begin{equation}
	\fluidscal{\mathbb{X}}^r (\cdot, \omega) \to \fluidscal{\mathbb{X}} \ \ \ \text{u.o.c. as } r \to \infty \text{ for some } \omega \in \mathscr{A}. \label{fluidconv}
	\end{equation}
	Then equation \eqref{fluidequ3} follows from \eqref{fluidDconv}, the convergence of $\fluidscal{Z}^r(0, \omega), \fluidscal{A}^r_s(\cdot, \omega)$, and $\fluidscal{C}^r(\cdot, \omega)$. 
\end{proof}
\label{Appendix-fluidlimit}

\vspace{-4mm}
\section{Proof of \eqref{prop1equ3}}
\begin{proof}[Proof of \eqref{prop1equ3}]
	By setting $\epsilon=1, t_2=L$ and $t_1=0$ in \eqref{mequ0}, and since $D^r_{ij}(t)$ has the same probability distribution as $\breve{D}^r_{ij}(t)$, we have that
	\begin{equation}
	P \left\lbrace D^r_{ij} \left( \frac{\sqrt{x_{r,0}}}{|N^r|}L\right) \geq 2NL\sqrt{x_{r,0}} \right\rbrace \leq \frac{\epsilon}{\sqrt{|N^r|}}.
	\end{equation}
	Then for $r$ large enough, 
	\begin{equation*}
	D^r_{ij} \left( \frac{\sqrt{x_{r,0}}}{|N^r|}L\right) +1 \leq 3NL\sqrt{x_{r,0}}.
	\end{equation*}
	Also notice that $D^r_{ij \sindex}(t) \leq D^r_{ij}(t)$ for all $t \geq 0$ and all $a=1,2,\dots,|N^r|$, we have that
	\begin{equation}
	D^r_{ij \sindex} \left( \frac{\sqrt{x_{r,0}}}{|N^r|}L\right) +1 \leq 3NL\sqrt{x_{r,0}}.
	\end{equation}
	Let $e = 0$ or $1$. It follows from \cite[Proposition 4.2]{Bramson} that, for large enough $n$,
	\begin{equation}
	P \left\lbrace \left\Vert V_{ij \sindex}(l) - \frac{l}{\mu_{ij \sindex}} \right\Vert_n \geq \epsilon n \right\rbrace \leq \frac{\epsilon}{n}.
	\end{equation}
	By setting $n = 3NL\sqrt{x_{r,0}}$, we get
	\begin{equation}
	P \left\lbrace \left\Vert V_{ij \sindex}(D^r_{ij \sindex}(t) + e) - \frac{D^r_{ij \sindex}(t)}{\mu_{ij \sindex}} \right\Vert_{(\sqrt{x_{r,0}}/|N^r|)L} \geq 3NL\sqrt{x_{r,0}} \epsilon \right\rbrace \leq  B_2 \frac{\epsilon}{\sqrt{|N^r|}},
	\end{equation}
	for $B_2 \geq 2/(3NL)$. By enlarging $\epsilon$ appropriately, we get, for $\tilde{b} = (1,0)$ or $(0,0)$,
	\begin{equation}
	P \left\lbrace \left\Vert V^{r,0}_{ij \sindex}(D^{r,0}_{ij \sindex}(t) , \tilde{b}) - \frac{D^{r,0}_{ij \sindex}(t)}{\mu_{ij \sindex}} \right\Vert_{L} \geq \epsilon \right\rbrace \leq \frac{\epsilon}{\sqrt{|N^r|}}.
	\end{equation}
	Multiplying the error bound $\lceil \sqrt{|N^r|}T \rceil $ and enlarging $\epsilon$ appropriately, we obtain 
	\begin{equation}
	P \left\lbrace \max_{m<\sqrt{|N^r|}T} \left\Vert V^{r,m}_{ij \sindex}(D^{r,m}_{ij \sindex}(t) , \tilde{b}) - \frac{D^{r,0}_{ij \sindex}(t)}{\mu_{ij \sindex}} \right\Vert_{L} \geq \epsilon \right\rbrace \leq \epsilon. \label{V1}
	\end{equation}
	For $b=(0,1)$ and $\tilde{b}=(0,0)$, by \eqref{srvscaling},
	\vspace{-4mm}
	\begin{align*}
	P &\left\lbrace \max_{m<\sqrt{|N^r|}T} \left\Vert V^{r,m}_{ij \sindex} (D^{r,m}_{ij \sindex}, \tilde{b}) - V^{r,m}_{ij \sindex}(D^{r,m}_{ij \sindex}(t),b) \right\Vert_L \geq \epsilon \right\rbrace\\
	&= P \left\lbrace \max_{m<\sqrt{|N^r|}T} \left\vert V_{ij \sindex} \left( D^r_{ij \sindex} \left( \frac{m}{\sqrt{|N^r|}} \right) \right) - V_{ij \sindex} \left( D^r_{ij \sindex} \left( \frac{m}{\sqrt{|N^r|}} \right) + 1 \right) \right\vert \geq \sqrt{x_{r,m}} \epsilon \right\rbrace. \numberthis \label{V2}
	\end{align*}
	Observe that, by \eqref{srvtinequa}, $V_{ij \sindex}(D^r_{ij \sindex}(m/\sqrt{|N^r|})) \leq |N^r|T$ and, by Lemma \ref{lemma-first-residual-time},
	\begin{equation}
	P \{ v_{ij \sindex}^{r,T,\max} \geq \sqrt{x_{r,m}}\epsilon \} \leq \epsilon \label{V3}
	\end{equation}
	for large enough $r$. Thus, we get \eqref{prop1equ3} by combining \eqref{srvtinequa} with \eqref{V1}-\eqref{V3}.
\end{proof}
\label{Appendix-Proof-of-prop1equ3}
\vspace{-5mm}
\section{Proof of Proposition \ref{prop5.3}}
\begin{proof}
	We use the bounds established in Proposition \ref{prop5.1}. Fix $L$, $T$ and $\epsilon > 0 $. Let $\mathscr{V}^r$ be the intersection of the complements of the events given in \eqref{prop1equ1}-\eqref{prop1equ3}, so $P\{ \mathscr{V}^r > 1 - \epsilon \}$. We show that for $r$ large enough and all $\omega \in \mathscr{V}^r$
	\vspace{-5mm}
	\begin{equation}
	\max_{m<\sqrt{|N^r|}T} \sup_{0 \leq t_1 \leq t_2 \leq L} |\mathbb{X}^{r,m}(t_2) - \mathbb{X}^{r,m}(t_1)| \leq \tilde{N} |t_2 - t_1| + \epsilon \label{Lipgeneral}
	\end{equation}
	for some $\tilde{N}$ that depends only on $\lambda$. We fix $\omega \in \mathscr{V}^r$ for the rest of the proof and so omit it from the notation. Let $t_1, t_2 \in [0,T]$ and $m \geq 0$. We first show that
	\begin{equation}
	|C^{r,m}(t_2) - C^{r,m}(t_1)| \leq N_0|t_2 - t_1| + \epsilon \label{LipC}
	\end{equation}
	for some $N_0 > 0$. Because $C^{r,m}_{ij}$ is nondecreasing, we have, by \eqref{generalscaling} and \eqref{process3}, that
	\begin{equation}
	0 \leq C^{r,m}_{ij}(t_2) - C^{r,m}_{ij}(t_1) \leq D^{r,m}_{ij}(t_2) - D^{r,m}_{ij}(t_1). \label{LiprelationCD}
	\end{equation}
	Then \eqref{LipC} is immediately satisfied from \eqref{prop1equ2}, and $N_0 = N$. Combining \eqref{LiprelationCD} with \eqref{process3} yields
	\begin{equation}
	|Z^{r,m}_{ij}(t_2) - Z^{r,m}_{ij}(t_1)| \leq 2 \left\vert D^{r,m}_{ij }(t_2) - D^{r,m}_{ij}(t_1)\right\vert +   |A_{sij}^{r,m}(t_2) - A_{sij}^{r,m}(t_1)|.
	\end{equation}
	By \eqref{prop1equ1}, $|A^{r,m}_i(t_2) - A^{r,m}_i(t_1)| < 2|\lambda||t_2 - t_1| + \epsilon$ for $r$ large enough. By setting $N_1 = 2N_0 + 2|\lambda|$, and using \eqref{prop1equ2}, we get
	\begin{equation}
	|Z^{r,m}(t_2) - Z^{r,m}(t_1)| \leq N_1 |t_2 -t_1| + \epsilon.
	\end{equation}
	Combining the results above with \eqref{process2} gives
	\begin{equation}
	|Q^{r,m}(t_2) - Q^{r,m}(t_1)| \leq N_2|t_2 - t_1| + \epsilon,
	\end{equation}
	for $N_2 = N_0 + 2|\lambda|$. Note that $N_1 \geq N_2$. Also, for $r$ large enough, by \eqref{timeshift}, \eqref{generalscaling}, and the fact that $T^r_{ij}(t) = \sum_{\sindex = 1}^{N^r_i} T^{r}_{ij \sindex}(t)$, 
	\begin{equation}
	|T^{r,m}_{ij}(t_2) - T_{ij}^{r,m}(t_1)| \leq 2 \beta_i |t_2 - t_1|.
	\end{equation}
	Note that, by definition of $\mathscr{V}^r$, the inequalities above hold for all $m < \sqrt{|N^r|}T.$ This shows that \eqref{Lipgeneral} holds for $r$ large enough with $\tilde{N} = N_1 \vee 2 $.
	
	Summarise the discussion above, processes $$\{\mathbb{X}^{r,m} = (A^{r,m}, A_s^{r,m}, A_q^{r,m}, Q^{r,m}, Z^{r,m}, C^{r,m}, T^{r,m}, D^{r,m})\}$$ are almost Lipschitz.
	\end{proof}
\label{Appendix-prop5.3}
\vspace{-6mm}
\section{Proof of Proposition \ref{HDSSCfunction}}
\begin{proof}
	The proof is similar to that of Proposition 5.5 in \cite{Dai}. Fix $L>0$ and let $\tilde{\mathbb{X}}$ be a hydrodynamic limit of $\mathscr{E}$. By the definition of hydrodynamic limit, $|\tilde{\mathbb{X}}(0)| \leq 1$. By Proposition \ref{mainbridge}, $\tilde{\mathbb{X}}$ satisfies the hydrodynamic model equations \eqref{HDfirst}-\eqref{HDlast} on $[0,L]$, thus by \eqref{HDfirst}, \eqref{HDmodel2}, and \eqref{HDmodel4}, one can easily find a $R_L$ such that
	\begin{equation}
	||\tilde{\mathbb{X}}(t)||_L \leq R_L. \label{HDSSC1}
	\end{equation}
	Fix $\epsilon >0$. Since $g$ is continuous, there exists $\delta>0$ such that 
	\begin{equation}
	|g(x)-g(y)|<\epsilon \label{HDSSC2}
	\end{equation}
	if $|x-y|<\delta$ and $x,y \in [-2R_L, 2R_L]$.\par
	By Corollary \ref{convergencecorollary}, fix $T>0$, and also fix $0<\delta<R_L$. For $\omega \in \mathscr{K}^r$, any $m < \sqrt{|N^r|}T$, and choosing $r$ large enough, there exists a hydrodynamic limit $\tilde{\mathbb{X}}$ such that
	\begin{equation}
	|| \mathbb{X}^{r,m}(t) - \tilde{\mathbb{X}}(t) ||_L \leq \delta. \label{HDSSC3}
	\end{equation}
	Together with \eqref{HDSSC1}, we have
	\begin{equation}
	||\mathbb{X}^{r,m}(t)||_L \leq || \mathbb{X}^{r,m}(t) - \tilde{\mathbb{X}}(t) ||_L + || \tilde{\mathbb{X}}(t) ||_L \leq 2R_L. \label{HDSSC4}
	\end{equation}
	Thus by \eqref{HDSSC1}, \eqref{HDSSC2}, \eqref{HDSSC4}, and Assumption \ref{SSCcondition}, $\forall t \in [0,L]$,
	\begin{align*}
	g(Q^{r,m}(t),Z^{r,m}(t)) 
	&= g(Q^{r,m}(t),Z^{r,m}(t)) - g(\tilde{Q}(t),\tilde{Z}(t)) + g(\tilde{Q}(t),\tilde{Z}(t))\\
	&\leq \epsilon + H(t), \label{HDSSC5}
	\end{align*}
	which is \eqref{HDfuncconverge1}.\par
	\eqref{HDfuncconverge2} is obtained similarly. Let $\tilde{\mathbb{X}}$ be a hydrodynamic limit of $\mathscr{E}_g$. Then there exists a subsequence $r_k$ of $r$ such that 
	\begin{equation}
	||\mathbb{X}^{r_k,0}(t) - \tilde{\mathbb{X}}(t)||_L \to 0, \label{HDSSC6}
	\end{equation}
	as $k \to \infty$. By the definition of $\mathscr{E}_g$, and the fact that $g(Q^{r_k,0}(0),Z^{r_k,0}(0)) \to 0$, combined with \eqref{HDSSC6} and the continuity of $g$, we have $g(\tilde{Q}(0),\tilde{Z}(0))=0$, thus, by the last statement of Assumption \ref{SSCcondition}, 
	\begin{equation}
	g(\tilde{Q}(t),\tilde{Z}(t)) = 0 \mbox{ for any } t \geq 0. \label{HDSSC7}
	\end{equation}
	Similar to \eqref{HDSSC3}, for $\omega \in \mathscr{L}^r$, and $r$ large enough, we can find a hydrodynamic limit $\tilde{\mathbb{X}}$ in $\mathscr{E}_g$ such that
	\begin{equation*}
	|| \mathbb{X}^{r,m}(t) - \tilde{\mathbb{X}}(t) ||_L \leq \delta, \label{HDSSC8}
	\end{equation*}
	and by \eqref{HDSSC7} and \eqref{HDSSC1}
	\begin{equation*}
	g(Q^{r,0}(0),Z^{r,0}(0)) \leq \epsilon.
	\end{equation*}
	\end{proof}
\label{Appendix-HDSSCfunction}

\bibliography{notes}
\end{document}